\documentclass{amsart}

\usepackage{amsthm,amssymb,latexsym,amsxtra}

\usepackage[cmtip,all]{xy}

\newcommand{\lto}{\longrightarrow}
\newcommand{\eto}{\hookrightarrow} 
\newcommand{\tto}{\twoheadrightarrow}

\DeclareMathOperator{\Spec}{Spec}
\DeclareMathOperator{\Supp}{Supp}

\DeclareMathOperator{\Deck}{Deck}
\DeclareMathOperator{\Aut}{Aut}
\DeclareMathOperator{\Ker}{Ker}
\DeclareMathOperator{\Hom}{Hom}
\DeclareMathOperator{\Var}{Var}
\DeclareMathOperator{\Sets}{Sets}

\theoremstyle{plain}
\newtheorem{theorem}{Theorem}[section]
\newtheorem{corollary}[theorem]{Corollary}
\newtheorem{lemma}[theorem]{Lemma}
\newtheorem{proposition}[theorem]{Proposition}

\theoremstyle{definition}
\newtheorem{definition}[theorem]{Definition}
\newtheorem{remark}[theorem]{Remark}
\newtheorem{block}[theorem]{}

\theoremstyle{remark}
\newtheorem*{notation}{Notation and conventions}
\newtheorem*{acknowledgments}{Acknowledgments}

\renewcommand{\theenumi}{\roman{enumi}}

\title[Hurwitz moduli varieties]{Hurwitz moduli varieties parameterizing 
Galois covers of an algebraic curve}
\author[V.~Kanev]{Vassil Kanev}
\date{April 3, 2024}
\address{V. Kanev\\Dipartimento di Matematica e Informatica, 
         Universit\`{a} di Palermo\\Via Archirafi, 34\\
         90123 Palermo (ITALY)}
                                \email{vassil.kanev@unipa.it}

\thanks{This version of the article has been accepted for publication, after peer review,  
in Serdica\\
 Mathematical  Journal,  https://doi.org/10.55630/serdica.2024.50.47-102\\ 
\copyright $\langle 2024\rangle$ Serdica Math. J}
\thanks{This research was supported by the University of Palermo [grant "Piano straordinario 
per il miglioramento della qualit\`{a} della ricerca e dei risultati della VQR 2020-2024 - Misura A, SSD MAT/03"].}

\begin{document}

\begin{abstract}
Given a smooth, projective curve $Y$, a finite group $G$ and a positive integer $n$ we study smooth, proper families 
$X\to Y\times S\to S$ of Galois covers of $Y$ with Galois group isomorphic to $G$ branched in $n$ points, parameterized by algebraic varieties $S$. When $G$ is with trivial center we prove that the Hurwitz space $H^G_n(Y)$ is a fine moduli variety for this moduli problem
and construct explicitly the universal family. For arbitrary $G$ we prove that $H^G_n(Y)$ is a coarse moduli variety. For families of pointed Galois covers of $(Y,y_0)$ we prove that the Hurwitz space $H^G_n(Y,y_0)$ is a fine moduli variety, and
construct explicitly the universal family, for arbitrary group $G$. We use classical tools of algebraic topology and of complex algebraic geometry. 
\end{abstract}

\maketitle

\section{Introduction}\label{s1}
Fulton constructed in \cite{Fu}, with an approach via fundamental groups, the Hurwitz spaces, complex manifolds $H^{d,n}$, whose points are in bijective correspondence with the equivalence classes of covers of degree $d$ of $\mathbb{P}^1_{\mathbb{C}}$ simply branched in $n$ points. These manifolds are connected by a classical result of L\"{u}roth, Clebsch and Hurwitz (cf. \cite[Proposition~1.5]{Fu}, \cite[Lemma~10.15]{Vo1}). 
Given $d\geq 3$ and $n\geq 2d-2$, Fulton studied in \cite{Fu} families of simple covers of $\mathbb{P}^1_{\mathbb{Z}}$ of degree $d$ branched in $n$ points, parameterized by schemes over $\mathbb{Z}$. He constructed a universal family and proved that over $\mathbb{C}$ its parameter scheme, endowed with the canonical complex space structure, is biholomorphic to the Hurwitz space  $H^{d,n}$.
\par
The construction of the Hurwitz spaces may be extended as follows. Given a smooth, projective, irreducible curve $Y$, a transitive subgroup 
$G\subset S_d$, conjugacy classes $O_1,\ldots,O_k$ in $G$ and positive integers $n_1,\ldots,n_k$, one constructs a complex manifold whose 
points are in bijective correspondence with the equivalence classes of covers of $Y$ of degree $d$, whose monodromy group is $G$, with the following branching data: the number of branch points is $n=n_1+\cdots+n_k$ and $n_i$ of the branch points have local monodromies in $O_i$ for every $i$. Similarly one may consider Hurwitz spaces which parameterize Galois covers of $Y$, with Galois group isomorphic to $G$ and branching data as above, up to $G$-equivariant isomorphisms over $Y$. These types of Hurwitz spaces were first introduced by Fried in 
\cite{Fr} for covers of $\mathbb{P}^1$ as a tool for the study of the arithmetic of the field extensions of $\mathbb{Q}[t]$, in particular in 
connection with the Inverse Galois Problem.
\par
A lot of work by various authors was devoted to determining the connected components of the Hurwitz spaces. To the author's knowledge the strongest result for $G=S_d$, $Y=\mathbb{P}^1$ and branching data with arbitrary set of conjugacy classes  
$O_1,\ldots,O_k$, $O_i\subset S_d$
was obtained by Kulikov, who proved that the Hurwitz spaces are connected, provided 
the branching data contains at least $3d-3$ transpositions \cite[Theorem~3.3]{Ku1}. This result was extended in \cite{Ve} to covers of a fixed curve $Y$ of genus $\geq 1$. The papers \cite{Ku2, Ku3, B-Ku} are devoted to determining the number of connected components of the Hurwitz spaces when every $n_i$ of the branching data is large enough. The paper 
\cite{K6} extends the connectivity result of Clebsch and Hurwitz to Hurwitz spaces of Galois covers of $\mathbb{P}^1$ with Galois group isomorphic to a Weyl group and branching data consisting of reflections. The Hurwitz spaces of Galois covers of $\mathbb{P}^1$ with Galois group isomorphic to the dihedral group $D_n$ were studied in \cite{CLP} and their connectedness was proved when a certain numerical type, related to the branching data is fixed.
\par
Given a projective, nonsingular, irreducible curve $Y$, a finite group $G$ and a positive integer $n$, we study smooth, proper families of Galois covers of $Y$, branched in $n$ points, with Galois group isomorphic to $G$ ($G$-covers), parameterized by algebraic varieties. We are concerned with the problem of whether the Hurwitz spaces are moduli varieties for appropriate categories of families of $G$-covers of $Y$, which means constructing universal families, parameterized by the Hurwitz spaces, or, when such families do not exist, proving that the Hurwitz spaces are coarse moduli varieties 
(cf. \cite[Definition~5.6]{M2}). We consider two types of families of covers.
\par
Let $y_{0}\in Y$ be a marked point.  A  smooth,  proper  family  of  pointed 
$G$-covers of $(Y,y_0)$  branched  in  $n$  points,  parameterized  by  an 
algebraic variety $S$, is a pair of morphisms 
($p:X\to Y\times S, \eta:S\to X$), where $\pi_{2}\circ p:X\to  S$  is  proper, 
smooth with connected fibers, $G$ acts by automorphisms on  $X$,  such  that 
every fiber $p_{s}:X_{s}\to Y\times \{s\}$ is a $G$-cover branched in $n$ 
points    contained    in    $Y\setminus    \{y_{0}\}$    and    $\eta(s)\in 
p_{s}^{-1}(y_{0})$ for $\forall s\in S$. We prove that the Hurwitz space 
$H^G_n(Y,y_0)$ which parameterizes the $G$-equivalence classes of the pointed 
$G$-covers of $(Y,y_0)$ branched in $n$ points is a fine moduli  variety. 
Namely, we construct explicitly a family 
\[
(p:\mathcal{C}(y_0)\to  Y\times  H^G_n(Y,y_0),\zeta : H^G_n(Y,y_0) \to \mathcal{C}(y_0))
\]
  and  prove  that  it  is 
universal in the category of families of pointed $G$-covers of  $(Y,y_0)$ 
branched in $n$ points, parameterized by algebraic varieties.
\par
We denote by $H^G_n(Y)$ the Hurwitz space which parameterizes the 
$G$-equivalence classes of the $G$-covers of $Y$ branched in  $n$  points. 
If the center of $G$ is trivial we prove that $H^G_n(Y)$ is  a  fine  moduli 
variety. Namely, we construct explicitly a family of $G$-covers 
$\pi:\mathcal{C}\to Y\times H^G_n(Y)$ and prove that it is universal in  the 
category of smooth, proper families of $G$-covers of $Y$ branched in  $n$ 
points, parameterized by algebraic varieties. If $G$ is an arbitrary group we 
prove that $H^G_n(Y)$  is  a  coarse  moduli  variety  for 
this category.
\par
Fixing the branching data, by choosing conjugacy classes 
$O_1,\ldots,O_k$ in $G$ and positive integers $n_1,\ldots,n_k$ as above, one 
obtains Hurwitz spaces, which are unions of connected components of 
$H^G_n(Y,y_0)$ or $H^G_n(Y)$. These Hurwitz spaces are fine moduli varieties
for the categories of families of pointed $G$-covers of $(Y,y_0)$, resp. 
families of $G$-covers of $Y$, provided $Z(G)=1$, with the prescribed branching data,
and they are coarse moduli varieties for these families for arbitrary $G$.
\par
The problem of constructing the Hurwitz moduli spaces was already studied in 
\cite{We} and \cite{ACV} in a more general set-up, over arbitrary algebraically closed fields and families 
parameterized by schemes. The existence of universal families, or coarse moduli schemes was proved
by very complicated constructions in the framework of the theory of stacks. Fulton wrote in \cite[p. 547]{Fu} 
that over $\mathbb{C}$ it is not difficult to construct analytically the universal families of simple 
covers of $\mathbb{P}^1$ parameterized by $H^{d,n}$. We give, over $\mathbb{C}$, a simple construction
of the universal families of $G$-covers of $Y$ with an approach via fundamental groups, as explicit as 
the classical construction of branched covering maps $X\to Y$ of a given Riemann surface \cite{For}. 
We use classical tools of algebraic topology and of complex algebraic geometry, in particular the 
GAGA theory \cite{Se3,SGA1}. The closely related topic of smooth, proper families of covers of degree
$d$ of a fixed curve $Y$ with monodromy group a fixed transitive subgroup $G$ of $S_d$ will be treated 
in a paper of the author in preparation. The smooth, proper families of pointed covers of $(Y,y_0)$ of 
degree $d$ with a fixed monodromy group $G\subset S_d$ are studied in \cite{K8}

We think that our approach, via fundamental groups, to the 
Hurwitz spaces, as moduli varieties of appropriate categories of families of covers, will be accessible
to a wider range of mathematicians who are interested in the Hurwitz spaces and the familes of covers 
of a fixed curve.
\par
The covers $X\to Y$ with restricted monodromy group and the related Galois covers $C\to Y$ yield 
polarized abelian varieties isogenous to abelian subvarieties of the Jacobian variety $J(X)$ 
\cite{Do, K5, CLRR}. 
The  smooth, proper families of such covers give morphisms of their parameter varieties 
to certain  moduli  spaces  of  polarized  abelian  varieties  by  means  of 
the variations of the associated polarized Hodge structures  of  weight  one. 
This indicates a perspective in the  study  of  the  abelian  varieties  of  low 
dimension and of their moduli by means  of  the  rich  geometry  of  curves.
The  unirationality  of  the  moduli   spaces   of 
three-dimensional abelian varieties $\mathcal{A}_{3}(1,1,d)$ and 
$\mathcal{A}_{3}(1,d,d)$ with $d\leq 4$ was proved in  \cite{K3,K4} 
by means of families of simply ramified  covers  of  elliptic  curves  of 
degree $d$ branched in 6 points. In \cite{Al} it was proved that 
every  sufficiently  general  principally  polarized  abelian   variety   of 
dimension 6 is  isomorphic  to  a  Prym-Tyurin  variety  of  a  cover  of 
$\mathbb{P}^{1}$ of degree 27, branched in 24 points, with monodromy group 
$W(E_6)\subset  S_{27}$. 
\par
The  Hurwitz  spaces  of  $G$-covers  of  $\mathbb{P}^{1}$ were intensively
studied in connection with the Inverse Galois Problem. 
We refer to \cite{De, Em2, RW} for  surveys  on  this  subject.  The problem
of constructing families of covers of $\mathbb{P}^{1}$ parameterized by the 
Hurwitz spaces, such that every fiber is a cover of the corresponding
equivalence class, was addressed in \cite[Section~4]{FV} and \cite{Em1}
(see also \cite[Chapter~10]{Vo1}). The constructed families, however, are not proper families of curves
over the Hurwitz spaces, but families of \'{e}tale covers of open subsets of $\mathbb{P}^{1}$.
\par
The monograph \cite{B-R} is devoted to  the 
Hurwitz schemes (or stacks) and of their  natural  compactifications. The authors
work with equivalence of covers different from the  one  considered  so  far and so
different are the sets of equivalence classes of covers. 
Namely two $G$-covers $\pi:C\to D$ and  $\pi':C'\to  D'$  are  considered 
equivalent if there is a $G$-equivariant  isomorphism  $f:C\to  C'$  and  an 
isomorphism $h:D\to D'$ such that $\pi'\circ f=h\circ \pi$.  In  comparison, 
in our set-up $D=D'=Y$ is a fixed curve and $h=id_{Y}$. 
\par
In  Section~\ref{s2}  we  prove  some 
properties of smooth, proper families of covers of $Y$, 
$X\to Y\times S\to S$ related to the branch locus $B\subset Y\times  S$.  
\par
In Section~\ref{s3} we give an explicit construction of a smooth, proper family 
of pointed $G$-covers of $(Y,y_0)$ branched in $n$ points 
$(p:\mathcal{C}(y_0)\to Y\times H^G_n(Y,y_0),\zeta)$  such  that  the  fiber 
over every element of $H^G_n(Y,y_0)$ is a pointed $G$-cover of the 
$G$-equivalence class represented by the element (Theorem~\ref{3.33}). In 
Proposition~\ref{3.24} we give  the  explicit  form  of  $p$  locally  at  the 
ramification points in analytic coordinates and later in Proposition~\ref{6.18} 
this is done for every smooth, proper family of $G$-covers of $Y$. 
\par
In Section~\ref{Serre}  we  give 
some generalizations of a result of Serre \cite[Proposition~20]{Se2} related 
to lifting of morphisms. 
\par
In Section~\ref{s5} we prove that the family 
$(p:\mathcal{C}(y_0)\to Y\times H^G_n(Y,y_0),\zeta)$ constructed in 
Section~\ref{s3} is universal, thus proving that $H^G_n(Y,y_0)$  is  a  fine 
moduli variety (Theorem~\ref{5.8}). We mention that a key ingredient in the proof is the use  of 
the criterion for extending morphisms from \cite{K2}. 
\par
Section~\ref{s4}  is 
devoted to the $G$-covers of $Y$  branched  in  $n$  points.  We  give  a 
structure of an  algebraic  variety  of  the  Hurwitz  space  $H^G_n(Y)$  by 
patching affine charts $U(y)$, $y\in Y$, which are quotients of 
$H^G_n(Y,y)$ with respect to a natural action of $G/Z(G)$ 
(Proposition~\ref{4.3}). If the center $Z(G)$ of $G$ is trivial we  construct 
a smooth, proper family of $G$-covers of $Y$ branched in $n$ points,
$\pi:\mathcal{C}\to Y\times H^G_n(Y)$, such that the fiber over every  point 
of $H^G_n(Y)$ is a $G$-cover of the $G$-equivalence class represented  by 
the point (Theorem~\ref{4.16}).  
\par
In Section~\ref{s6} we prove that the family 
$\pi:\mathcal{C}\to Y\times H^G_n(Y)$ is universal, provided $G$ has trivial center, 
thus proving that 
$H^G_n(Y)$ is a fine moduli variety (Theorem~\ref{6.3}). If $G$ is arbitrary, we prove
in Theorem~\ref{6.8a}, verifying the conditions 
of \cite[Definition~5.6]{M2}, that $H^G_n(Y)$ is a coarse moduli variety 
for the category of smooth, proper families of $G$-covers of $Y$ branched 
in $n$ points, parameterized by algebraic varieties. The 
construction of $\pi:\mathcal{C}\to Y\times H^G_n(Y)$ as well as the
proofs of Theorem~\ref{6.3} and Theorem~\ref{6.8a} are reduced  to  the 
universal family 
$(p:\mathcal{C}(y_0)\to    Y\times    H^G_n(Y,y_0),\zeta)$    of     pointed 
$G$-covers   by   means   of  a second action 
of  $G$  on  $\mathcal{C}(y_0)$, constructed in Section~\ref{s4}, 
which  lifts  a natural  
action of $G$ on $Y\times H^G_n(Y,y_0)$ and  commutes  with  the action  of  $G$ 
relative to the Galois cover $p:\mathcal{C}(y_0)\to Y\times  H^G_n(Y,y_0)$.
 Finally in Theorem~\ref{5.16}, Theorem~\ref{6.14} and 
Theorem~\ref{6.15} we give variants of the main theorems in which 
the families of $G$-covers of $Y$ have local monodromies at the branch points in
fixed conjugacy classes of $G$.

\begin{notation}
We assume the base field is $\mathbb{C}$. Algebraic varieties are reduced,  separated, 
possibly reducible schemes of finite type, \emph{points} are closed  points. 
Fiber  products  and 
pullbacks are those defined in the category of schemes over $\mathbb{C}$. A 
cover $f:X\to Y$ of algebraic varieties is  a  finite,  surjective 
morphism. If $G$ is a finite group which  acts faithfully by automorphisms
on $X$, i.e. $G\to Aut(X)$ is injective, $f$ is $G$-invariant, and 
$\overline{f}:X/G \to Y$ is an isomorphism, then $f$ is a Galois cover
with Galois group isomorphic to $G$.
Given an algebraic variety $(X,\mathcal{O}_{X})$  
the  canonically 
associated complex space is denoted by $(X^{an},\mathcal{O}_{X^{an}})$ 
\cite{SGA1}. Its  topological  space  is  denoted  by  $|X^{an}|$.  Given  a 
topological space $M$ and two paths $\alpha : I\to M$ and 
$\alpha' : I\to M$, $I=[0,1]$, 
we write $\alpha \sim  \alpha'$  if $\alpha'$ has the same end points as $\alpha$ and  is 
homotopic to $\alpha$ (with homotopy leaving the endpoints fixed) \cite[Chapter~2, 
\S~2]{Mas}. The set of paths homotopic to $\alpha$ is denoted by $[\alpha]$. 
The product of the paths $\alpha$ and $\beta$ is denoted  by  $\alpha  \cdot 
\beta$ and equals the path $\gamma:I\to M$, where $\gamma(t)=\alpha(2t)$  if 
$t\in [0,\frac{1}{2}]$, $\gamma(t)=\beta(2t-1)$ if  $t\in  [\frac{1}{2},1]$. 
Given a covering space $p:M\to N$ of the topological  space  $N$,  the 
map $p$ is called topological covering map. Lifting a path $\alpha$ of $N$ from
initial point $z\in M$ the end point is denoted by $z\alpha$.
\end{notation}

\section{Smooth families of covers of a curve}\label{s2}
Throughout the paper, with the exception of Section~\ref{Serre}, $Y$ is  a smooth, 
projective, irreducible curve of genus $g\geq 0$, $n$ is a positive integer and 
$G$ is a finite group.
\begin{block}\label{2.0}
We  recall 
some  facts  about  the  Hilbert  scheme  $Y^{[n]}$  which  parameterizes  the 
$0$-dimensional subschemes of length  $n$  of  $Y$  \cite{FG}.  There  is  a 
bijective correspondence between the effective divisors of $Y$ of degree $n$ 
and the $0$-dimensional subschemes of $Y$ of length $n$:  every  divisor  $D 
=\sum_{i=1}^{r}n_{i}y_{i}$ corresponds to the closed subscheme of $Y$  whose 
closed subset is $\Supp(D)=\{y_{1},\ldots,y_{r}\}$ and the  structure  sheaf 
is the skyscraper sheaf
$\oplus_{i=1}^{r}\mathcal{O}_{Y,y_{i}}/\mathfrak{m}_{y_{i}}^{n_{i}}$. 
Abusing the notation we will denote it again by $D$. We write $\deg D = n  = 
\ell (D)$.
\end{block}

\begin{block}\label{2.1}
Let  $Y^{(n)}$  be  the   symmetric   product   $Y^{(n)}=Y^{n}/S_{n}$   (cf. 
\cite[Ch.~III \S~14]{Se}). This is a projective variety \cite[Lecture~10]{Ha} 
and it parameterizes the effective divisors of $Y$ of degree $n$.  We  denote 
by $Y^{(n)}_{\ast}$ the open  subset,  which  corresponds  to  the  divisors 
without  multiple  points.  It   is   the   complement   of   the   quotient 
$\Delta/S_{n}$, where $\Delta \subset Y^{n}$ is the big diagonal. For  every 
partition $\nu  =  (n_{1},\ldots,n_{r})$, of length $\ell(\nu)=r$,  
$n_{1}\geq \cdots  \geq  n_{r}$, 
$n_{1}+\cdots +n_{r}=n$, let us denote by $Y^{(n)}_{\nu}$ the set 
\[
Y^{(n)}_{\nu} = \{n_{1}y_{1}+\cdots +n_{r}y_{r} | y_{i}\neq y_{j}\; 
\text{for}\; 
i\neq j\}.
\]
Let us denote by $Y^{(n)}_{r}$ the set $\{D\in Y^{(n)}| |\Supp(D)|=r\}$  and 
by $Y^{(n)}_{\leq r}$ the set \linebreak
$\{D\in Y^{(n)}| |\Supp(D)|\leq r\}$. Consider 
the composition of morphisms $Y^{r}\to Y^{n}\to Y^{(n)}$,  where  the  first 
one is 
\[
(y_{1},\ldots,y_{r})\mapsto (\underset{n_{1}}{y_{1},\ldots,y_{1}},\ldots,
\underset{n_{r}}{y_{r},\ldots,y_{r}}).
\]
Its image is a closed, irreducible subset of $Y^{(n)}$ equal to the  closure 
$\overline{Y^{(n)}_{\nu}}$. One has 
$Y^{(n)}_{\leq r}=\bigsqcup_{\ell(\nu)\leq r}\overline{Y^{(n)}_{\nu}}$ and 
$Y^{(n)}_{\nu}=\overline{Y^{(n)}_{\nu}}\setminus    Y^{(n)}_{\leq     r-1}$. 
Therefore  $Y^{(n)}_{\nu}$  is  an  irreducible  locally  closed  subset  of 
$Y^{(n)}$     of      dimension      $r$.      Represent      $\nu$      as 
$(1^{r_{1}},2^{r_{2}},\ldots,s^{r_{s}})$, where $r_{i}$  is  the  number  of 
times $i$ occurs in $(n_{1},\ldots,n_{r})$. Every $D\in  Y^{(n)}_{\nu}$  may 
be written in a unique way  as  $D=D_{1}+2D_{2}+\cdots  +sD_{s}$  where  the 
divisor \linebreak
$D=D_{1}+D_{2}+\cdots +D_{s}$ has no multiple points. Let us  denote 
the set of such $s$-tuples by 
$\left(Y^{(r_{1})}\times  \cdots  \times  Y^{(r_{s})}\right)_{\ast}$.  One  
obtains   a 
bijective map 
\begin{equation}\label{e2.2}
\left(Y^{(r_{1})}\times  \cdots  \times  Y^{(r_{s})}\right)_{\ast} \lto 
Y^{(n)}_{\nu}
\end{equation}
Let us denote by $A$  the  universal  divisor  $A=\{(y,D)|y\in  \Supp  D\}$, 
$A\subset  Y\times  Y^{(n)}$.  Let  $\Delta_{Y}\subset  Y\times  Y$  be  the 
diagonal. Then $A$ is the image of $\Delta_{Y}\times Y^{n-1}$  with  respect 
to the quotient morphism $Y\times Y^{n}\to Y\times Y^{(n)}$, so  $A$  is  an 
irreducible, closed subvariety of $Y\times Y^{(n)}$ of codimension 1.
\end{block}
\begin{proposition}\label{2.2}
Let $(n,r)$ be a pair of positive integers, such that $r\in [1,n]$.  In  the 
set-up of \S~\ref{2.1} the following properties hold:
\begin{enumerate}
\item
$Y^{(n)}$ is isomorphic to the Hilbert scheme $Y^{[n]}$  which  parameterizes 
the $0$-dimensional subschemes of $Y$ of length $n$ and the closed subscheme 
$A\subset Y\times Y^{(n)}$ is the corresponding universal family.
\item
The set $Y^{(n)}_{r}$ is locally closed and it is a disjoint  union  of  the 
irreducible, locally  closed  subsets  $Y^{(n)}_{\nu}$  with  $\ell(\nu)=r$, 
which are moreover smooth of dimension $r$. 
\item
For every partition $\nu$ of $n$ of length $r$, represented in the form 
$(1^{r_{1}},\ldots,s^{r_{s}})$, $r_{1}+\cdots +r_{s}=r$ the map 
$\left(Y^{(r_{1})}\times  \cdots  \times  Y^{(r_{s})}\right)_{\ast} \lto 
Y^{(n)}_{\nu}$ given   by         
\begin{equation}\label{e2.3}
(D_{1},\ldots,D_{s})\mapsto D_{1}+2D_{2}+\cdots+sD_{s} 
\end{equation}
is an isomorphism.
\end{enumerate}
\end{proposition}
\begin{proof}
(i) The variety $Y\times Y^{(n)}$ is smooth, so $A$ is an effective  Cartier 
divisor. The projection $p: A \to Y^{(n)}$ is proper with finite fibers,  so 
by Zariski's main theorem it is finite.  Furthermore  it  is  surjective  and 
flat.  Indeed,  let  $a\in   A$,   $b=p(a)$.   Let   $\mathcal{O}_{a}$   and 
$\mathcal{O}_{b}$ be  the  fibers  of  the  structure  sheaves  of  $Y\times 
Y^{(n)}$   and   $Y^{(n)}$   respectively.   Let   $J_{A,a}=(f)$.   Applying 
\cite[20.E]{Ma1} to $u:\mathcal{O}_{a}\to \mathcal{O}_{a}$, where $u(x)=xf$, 
one    concludes     that     $\mathcal{O}_{A,a}=\mathcal{O}_{a}/(f)$     is 
$\mathcal{O}_{b}$-flat.   The   variety   $Y^{(n)}$   is   irreducible,   so 
$p_{\ast}\mathcal{O}_{A}$ is a locally free sheaf of  rank  $n$,  hence  every 
fiber of $p:A\to Y^{(n)}$ is of length $n$. By  the  universal  property  of 
Hilbert   schemes   there   exists    a    unique    classifying    morphism 
$\varphi:Y^{(n)}\to Y^{[n]}$ such that $A\subset Y\times Y^{(n)}$ is the 
pullback of the universal family $\mathcal{W}\subset Y\times Y^{[n]}$. Now, 
$Y^{[n]}$  is  a  smooth  scheme  \cite[Theorem~4.3.5]{Sern}  and  $\varphi$ 
induces a  bijection  of  the  closed  points  of  $Y^{(n)}$  and  $Y^{[n]}$ 
(cf. \S~\ref{2.0}). Therefore $\varphi$ is an isomorphism.
\par
(ii) One has $Y^{(n)}_{r}=Y^{(n)}_{\leq r}\setminus Y^{(n)}_{\leq r-1}$,  so 
$Y^{(n)}_{r}$ is locally closed. Clearly \linebreak
$Y^{(n)}_{r}=\bigsqcup_{\ell(\nu) = r}Y^{(n)}_{\nu}$. The  last 
claim follows from (iii).
\par
(iii) Consider the map 
$\psi:Y^{(r_{1})}\times \cdots \times Y^{(r_{s})}\to Y^{(n)}$ given by 
\eqref{e2.3}.
 It is the quotient by 
$S_{n}$ and by $S_{r_{1}}\times \cdots \times S_{r_{s}}$ of the  product  of 
the diagonal morphisms $Y^{r_{1}}\times \cdots \times  Y^{r_{s}}\to  Y^{n}$, 
therefore $\psi$ is a morphism. Its image is closed, irreducible and  equals 
$\overline{Y^{(n)}_{\nu}}$. The map of (iii) is the restriction of $\psi$ on 
the preimage  of 
$\overline{Y^{(n)}_{\nu}}\setminus \overline{Y^{(n)}_{\leq r-1}}$ and it  is 
bijective. By \cite[Corollary~14.10]{Hart} it  suffices  to  verify  that  the 
differential $d\psi$ is injective at every point of 
$\left(Y^{(r_{1})}\times  \cdots  \times   Y^{(r_{s})}\right)_{\ast}$.  This 
holds since  for  every  $m\in  \mathbb{N}$  the  morphism  $Y\to  Y^{(m)}$, 
$y\mapsto my$, has injective differential at every point. Indeed, let  $p\in 
Y$ and let $U\ni y$ be an embedded open  disk  with  $t:U\to  \mathbb{C}$  a 
coordinate at $p$. Then $U^{m}/S_{m}$ is a coordinate neighborhood  of  $mp$ 
in $Y^{(m)}$ with local coordinates the $m$ elementary symmetric  polynomials 
of $t\circ p_{i}:U^{m}\to U\to \mathbb{C}$, $i=1,\ldots,m$.  The  map  $U\to 
U^{m}/S_{m}$ has the form  $t\mapsto  (mt,\ldots,\binom{m}{i})t^{i},\ldots)$ 
with derivative $(m,0,\ldots,0)$ at $t=0$.
\end{proof}
\begin{definition}\label{2.6}
Let $n$ be a positive integer. Let $X$ and $S$  be  algebraic  varieties.  A 
morphism $f:X\to Y\times S$ is called a \emph{smooth family  of  covers  of 
$Y$ branched in  $n$  points}  if  $\pi_{2}\circ  f:X\to  Y\times  S$  is  a 
proper, smooth morphism such that for every $s\in S$ the fiber $X_{s}$ is an 
irreducible curve and $f_s:X_s\to Y$ is a cover branched in $n$ points
\end{definition}
\begin{lemma}\label{6.16}
Let $p:M\to N$ be a finite morphism of algebraic varieties.  Let  $G$  be  a 
finite  group  which  acts  by  automorphisms  on  $M$  so   that   $p$   is 
$G$-invariant.
\begin{enumerate}
\item
The quotient set $M/G$ and the quotient map  $M\to  M/G$  have  a  structure 
of an algebraic  variety  and  a  finite morphism  and  $p$  equals the  composition  of 
the induced finite morphisms $M\to M/G \to N$.
\item
Suppose that $M/G\to N$ is an isomorphism. Then $|M^{an}|/G\to |N^{an}|$  is 
a homeomorphism.
\end{enumerate}
\end{lemma}
\begin{proof}
\par
(i) Let $x\in M$. Let $U$ be an  affine  open  set  in  $N$  which  contains 
$p(x)$. Then $p^{-1}(U)$ is an affine open  set  which  contains  the  orbit 
$Gx$. Apply \cite[Ch.~III Prop.~19]{Se}.
\par
(ii) $p^{an}:M^{an}\to N^{an}$ is a finite holomorphic map 
\cite[Prop.~3.2(vi)]{SGA1}.  The  induced  map  $|M^{an}|/G\to  |N^{an}|$  is 
bijective and continuous. It is a closed map, in  fact,  the  image  of  every 
closed subset $Z\subset |M^{an}|/G$ equals the  image  by  $p^{an}$  of  its 
preimage in $|M^{an}|$ which is closed  since  $p^{an}$  is  a  finite  map. 
Therefore $|M^{an}|/G\to |N^{an}|$ is a homeomorphism.
\end{proof}
\begin{proposition}\label{2.7}
Let $n$ be a positive integer. Let $f:X\to Y\times S$ be a smooth family  of 
covers of $Y$ branched in $n$ points. Then 
\begin{enumerate}
\item
$f$ is finite, surjective and flat.
\item
The discriminant scheme $D$ of $f:X\to Y\times S$ (cf. \cite[Ch.~VI  n.6]{A-K}) 
is an effective relative Cartier divisor with respect to 
$\pi_{2}:Y\times  S\to S$ (cf. \cite[Lecture~10]{M1}).
\item
Let $B\subset Y\times S$ be  the  support  of  $D$.  Let 
$X'=f^{-1}(Y\times S\setminus B)$. Then 
$f|_{X'}:X'\to  Y\times  S\setminus B$  is  a  finite, 
\'{e}tale, surjective morphism.
\item 
$|X^{an}|\setminus f^{-1}(B) \to |(Y\times S)^{an}|\setminus B$ is
a topological covering map.
\item
Suppose $S$ is connected. Then there exists an integer $N$ such that 
$\ell(D_{s})=N$ for $\forall s\in S$. 
\item
For every $s\in S$ let $B_{s}$ be the branch locus of $f_s:X_s\to  Y$.  Then 
the map $\beta:S\to Y^{(n)}$ given by $\beta(s)=B_{s}$ is a morphism.
\item
The projection $B \to S$ is finite, \'{e}tale, surjective of degree $n$.
\item
Let $G$ be a finite group which acts by automorphisms on $X$, so that $f$  is 
$G$-invariant and the morphism $X_{s}/G \to Y$ induced by $f_s$ is  an  isomorphism 
for every $s\in S$. Then the morphism $X/G \to Y\times S$ induced by $f$ is an 
isomorphism.
\end{enumerate}
\end{proposition}
\begin{proof}
(i) For every $s\in S$, $f_s:X_s\to Y\times \{s\}$ is a  finite,  surjective 
morphism, so $f:X\to Y\times S$ is  surjective  with  finite  fibers.  It  is 
proper  since  $\pi_{2}\circ  f:X\to  S$  is  proper  by   hypothesis   (cf. 
\cite[Ch.~3 Prop.~3.16]{Liu}). By Zariski's main theorem $f:X\to Y\times  S$ 
is finite. The morphisms $X\to S$ and $Y\times S\to  S$  are  flat  and  for 
every $s\in S$, $X_s\to Y\times \{s\}$ is flat, therefore $f:X\to Y\times S$ 
is flat (cf. \cite[(20.G)]{Ma1}).
\par
(ii) The statement is local, so we may assume that $S$ is connected. By  (i) 
$f_{\ast}\mathcal{O}_{X}$ is a locally free sheaf.  The  discriminant  ideal 
sheaf
$J_{D}$ is the image of the invertible sheaf 
$\left(\wedge^{\text{max}}f_{\ast}\mathcal{O}_{X}\right)^{\otimes      2}\to 
\mathcal{O}_{Y\times  S}$  (cf.  \cite[p.124]{A-K}).  Let  $z\in  \Supp  D$, 
$s=\pi_{2}(z)$ and let $d_{z}$ be the generator of $(J_{D})_z$.  Consider  the 
local homomorphism $\mathcal{O}_{s}=\mathcal{O}_{S,s}\to 
\mathcal{O}_{Y\times S,z} = \mathcal{O}_{z}$. The image of $d_{z}$ in 
$\mathcal{O}_{Y\times  S,z}\otimes  \mathbb{C}(s)   =   \mathcal{O}_{Y\times 
\{s\},z}$ generates the discriminant ideal of $X_{s}\to  Y\times  \{s\}$  at 
the point $z$, so it is a non-zero-divisor. Applying  \cite[(20.E)]{Ma1}  to 
$u:\mathcal{O}_{z}\to \mathcal{O}_{z}$. where $u(a) = ad_{z}$ one  concludes 
that   $d_{z}$   is   a   non-zero-divisor    in    $\mathcal{O}_{z}$    and 
$\mathcal{O}_{D,z}=\mathcal{O}_{z}/(d_{z})$    is    $\mathcal{O}_{s}$-flat. 
Therefore $D$ is an effective Cartier divisor of $Y\times S$ and 
$\pi_{2}|_{D}:D\to S$ is flat. This proves (ii).
\par
(iii) This follows from \cite[Ch.~6 Proposition~(6.6)]{A-K}.
\par
(iv) 
$f^{an}:X^{an}\setminus f^{-1}(B) \to (Y\times 
S)^{an}\setminus     B$     is     unramified     and     flat     by 
\cite[Prop.~3.1(iii)]{SGA1}, hence it is locally biholomorphic by 
\cite[Th\'{e}or\`{e}me~3.1]{Gr1}. Furthermore it is proper by
\cite[Prop.~3.2(v)]{SGA1}, hence it is a topological covering map by 
\cite[Prop.~4.22]{For}.
\par
(v) The projection $g=\pi_{2}|_{D}:D\to S$ is  a  finite,  surjective,  flat 
morphism of schemes. Since $S$ is connected, there  exists  an  integer  $N$ 
such that $g_{\ast}\mathcal{O}_{D}$ is a locally free sheaf of rank $N$. One 
has $\ell(D_{s}):= h^{0}(\mathcal{O}_{D_{s}})=N$ for every $s\in S$. 
\par
(vi) We may assume, without loss of generality,  that $S$  is  connected.  For 
every $s\in S$ one has $B_{s}=\Supp D_{s}$, $\ell(D_{s})=N$, $|B_{s}|=n$. By 
(ii)  the  closed  subscheme  $D$  of  $Y\times  S$  is  a  flat  family  of 
$0$-dimensional  subschemes  of  $Y$   of   length   $N$.   Let   us   apply 
Proposition~\ref{2.2} for the pair $(N,n)$. The classifying morphism 
$h:S\to Y^{(N)}$ has image contained in $Y^{(N)}_{n} = 
\bigsqcup_{\nu,\ell(\nu) =  n}Y^{(N)}_{\nu}$,  therefore   this 
image is contained in the locally closed  subset  $Y^{(N)}_{\nu}$  for  some 
partition $\nu$ of $N$. Write $\nu$ in the form 
$(1^{n_{1}},2^{n_{2}},\ldots,k^{n_{k}})$, where $n_1+\cdots+n_{k}=n$ and 
$n_1+2n_2+\cdots kn_{k}=N$. Then  the  map  
$\beta:S\to  Y^{(n)}_{\ast}\subset Y^{(n)}$ is the composition of morphisms
\[
S  \overset{h}{\lto}  Y^{(N)}_{\nu}   \lto   (Y^{(n_{1})}\times   \cdots    
\times 
Y^{(n_{k})})_{\ast} \lto Y^{(n)}_{\ast}
\]
where   the   middle   one   is   the   inverse   of   the   isomorphism   of 
Proposition~\ref{2.2}(iii) and the last one is obtained from 
$Y^{n_{1}}\times \cdots \times Y^{n_{k}} \overset{=}{\lto} Y^{n}$ taking the 
quotient of $Y^{n}$ by $S_{n}$ and of 
$Y^{n_{1}}\times \cdots \times Y^{n_{k}}$ by $S_{n_{1}}\times \cdots  \times 
S_{n_{k}}$.
\par
(vii) The universal divisor $A\subset Y\times Y^{(n)}$ has the property  that  the 
projection $A\to Y^{(n)}$ is finite, surjective, flat and is unramified over 
$Y^{(n)}_{\ast}$. The morphism $\beta:S\to Y^{(n)}$ has image contained in 
$Y^{(n)}_{\ast}$, so the pullback $A_{S}=S\times_{Y^{(n)}}A$  is  a  closed 
subscheme  of  $Y\times  S$  and  the  morphism  $A_{S}\to  S$  is   finite, 
surjective, flat and unramified. The underlying reduced subscheme of $A_{S}$ 
is $B$ and $A_{S}$ is reduced by  \cite[p.184]{Ma2},  so  $A_{S}$  coincides 
with $B$. 
\par
(viii) By (i) and Lemma~\ref{6.16} the    morphism  $X/G\to  Y\times  S$ 
induced by $f$ is finite. It is bijective, and fits  in  the 
commutative diagram 
\begin{equation*}
\xymatrix{
X/G\ar[rr]\ar[rd]&&
Y\times S\ar[dl]\\
&S
}
\end{equation*}
whose vertical  morphisms  are  proper.  The  scheme-theoretical  fibers  of 
$X/G\to S$ (over the closed points of $S$) are isomorphic  to  $X_{s}/G$  by  \cite[Prop.~A.7.1.3]{KM}.  The 
assumption that $X_{s}/G\to Y\times \{s\}$ is an isomorphism for every $s\in 
S(\mathbb{C})$ implies by \cite[Prop.~4.6.7]{EGAIII} that every $s\in  S(\mathbb{C})$  has  an  open 
neighborhood $U$ such that $(X/G)_{U}\to Y\times U$ is a  closed  embedding. 
This implies that $X/G\to Y\times S$  is  an  isomorphism  since  $X/G$  and 
$Y\times S$ are reduced schemes.
\end{proof}
\section{Parameterization of pointed $G$-covers}\label{s3}
In the rest of the paper the elements $D\in Y_{\ast}^{(n)}$ are considered 
as subsets of $Y$ of cardinality $n$.
 We start with some definitions and recall some known facts (see e.g. 
\cite[Section~1]{K1})
\begin{definition}\label{3.1} 
Let $G$ be a finite group.
\begin{enumerate}
\item
A $G$-cover of $Y$ is a cover $p:C\to  Y$,  where  $C$  is  a  smooth, 
irreducible, projective curve such that $G$ acts faithfully on  the  left  on  $C$  by 
automorphisms of $C$, $p$ is $G$-invariant and $\overline{p}:C/G \to Y$ is an isomorphism. 
\item
Two $G$-covers of $Y$, $p:C\to  Y$  and  $p_{1}:C_{1}\to  Y$  are  called 
$G$-equivalent  if  there  exists  a  $G$-equivariant   isomorphism   $f:C\to 
C_{1}$ such that  $p=p_{1}\circ  f$.  If  the  center  of  $G$  is  trivial, 
$Z(G)=1$, such an isomorphism is unique if it exists.
\item
Let $y_0\in Y$. A pointed $G$-cover of $(Y,y_0)$  is  a  couple  $(p:C\to 
Y,z_{0})$, where $p:C\to Y$ is a  $G$-cover  unramified  at  $y_{0}$  and 
$z_{0}\in p^{-1}(y_0)$.
\item
Let $(p:C\to Y,z_0)$ and $(p_1:C_1\to Y,w_0)$ be two  pointed  $G$-covers 
of $(Y,y_0)$. They are called $G$-equivalent if there is  a  $G$-equivariant 
isomorphism $f:C\to C_{1}$ such that $p=p_{1}\circ f$ and  $f(z_{0})=w_{0}$. 
Such an isomorphism is unique if it exists.
\end{enumerate}
\end{definition}
\begin{block}\label{3.2}
Let $(p:C\to Y,z_0)$ be a pointed $G$-cover of $(Y,y_0)$ branched in  $n$ 
points, $n\geq 1$. Let $D=\{b_1,\ldots ,b_n\}$  be  its  branch  locus.  Let 
$C'=p^{-1}(Y\setminus  D)$,  $p'=p|_{C'}$.  Endowing  $C$  and  $Y$  with  the 
canonical Euclidean topologies of $|C^{an}|$ and $|Y^{an}|$ respectively, 
 $p':C'\to  Y\setminus  D$  is  a  topological covering map.
\par
Let  $\alpha:I\to  Y\setminus  D$,  $I=[0,1]$,  be  a   closed   path   with 
$\alpha(0)=\alpha(1)=y_{0}$. Let us denote by $z_{0}\alpha$ the end point of 
its    lifting    $\alpha'_{z_{0}}:I\to    C'$    with     initial     point 
$\alpha'_{z_{0}}(0)=z_{0}$. Let $g\in G$ be the  unique  element  such  that 
$gz_{0}=\alpha'_{z_{0}}(1)=z_{0}\alpha$. One associates in this way with every 
element $[\alpha]\in \pi_1(Y\setminus D,y_0)$ an element $g\in  G$.  We  let 
$g=m_{z_{0}}([\alpha])$. The map $m_{z_{0}}:\pi_1(Y\setminus D,y_0)\to G$  is 
a surjective homomorphism.
\par
Let $\overline{U}_{1},\ldots,\overline{U}_{n}$ be embedded closed disks in 
$Y\setminus y_0$
which are disjoint and such  that  $b_{i}\in 
U_{i}$  for  $\forall  i$, where $U_{i}$ is the interior of 
$\overline{U}_{i}$. For every 
$i=1,\ldots,n$ let us 
choose a path $\eta_{i}:I\to  Y\setminus \cup_{j=1}^{n} U_{j}$  such  that  
$\eta_{i}(0)=y_{0}$, 
$\eta_{i}(1)\in  \partial   \overline{U}_{i}$   and   let   $\gamma_{i}:I\to 
Y\setminus D$ be the closed path which  starts  at  $y_{0}$,  travels  along 
$\eta_{i}$,   then   makes   a   counterclockwise   loop   along   $\partial 
\overline{U}_{i}$ and returns back to  $y_{0}$  along  $\eta_{i}^{-}$.  The 
condition that the branch locus of $p:C\to Y$ equals $D$  is  equivalent  to 
the condition that $p':C\setminus p^{-1}(D)\to Y\setminus D$  is  unramified 
and
\begin{equation}\label{e3.3}
m_{z_{0}}([\gamma_{1}])\neq 1, \ldots, m_{z_{0}}([\gamma_{n}])\neq 1.
\end{equation}
Let      $i\in      [1,n]$.      Varying      
$\overline{U}_{1},\ldots,\overline{U}_{n}$       and 
$\eta_{1},\ldots,\eta_{n}$ the elements $m_{z_{0}}([\gamma_{i}])$ belong  to 
the same conjugacy class of $G$.
\end{block}
\begin{definition}\label{3.3a}
Given a pointed $G$-cover $(C,z_{0})\to (Y,y_0)$  branched  in  $D$, 
$D\subset 
Y\setminus y_{0}$ the homomorphism $m_{z_{0}}:\pi_1(Y\setminus D,y_0)\to  G$ 
and the pair $(D,m_{z_{0}})$ are  called  respectively  \emph{the  monodromy 
homomorphism} and \emph{the monodromy invariant} associated  with  the  pointed 
$G$-cover.
\end{definition}
\begin{remark}\label{3.2a}
We will use the terminology pointed topological $G$-covering map and monodromy  homomorphism 
also for topological Galois covering maps $p:M\to  N$,  $p(z_{0})=y_{0}$,  where 
$M$ and $N$ are connected, locally connected topological spaces and  
$\theta : G \to \Deck(M/N)$ is a fixed isomorphism with
the group of covering transformations of $p:M\to  N$.
\end{remark}
\begin{block}\label{3.3}
Let $y_{0}\in Y$. Associating with a pointed $G$-cover $(C,z_0)\to (Y,y_0)$ 
its monodromy invariant $(D,m_{z_{0}})$  Riemann's  existence  theorem  
 establishes  a  one-to-one  correspondence 
between the set of $G$-equivalence classes  $[p:C\to  Y,z_{0}]$  of  pointed 
$G$-covers of $(Y,y_0)$ branched in $n$  points  and  the  set  of  pairs 
$(D,m)$,  where  $D\in  (Y\setminus  y_0)^{(n)}_{\ast}$   and   $m:\pi_1(Y\setminus 
D,y_0)\to   G$   is    a    surjective    homomorphism    which    satisfies 
Condition~(\ref{e3.3}). We briefly recall why this correspondence is 
bijective (cf. \cite{Mas}, \cite{For}, \cite{Vo1}, \cite[Prop.~1.3]{K1}). 
\par
Let $p:(C,z_0)\to (Y,y_0)$ and $p_{1}:(C_{1},z_1)\to (Y,y_0)$ be two  pointed 
$G$-covers with the same monodromy invariant $(D,m)$. Let 
$C'=p^{-1}(Y\setminus D)$, $C'_{1}= \linebreak
p_{1}^{-1}(Y\setminus  D)$. 
Then there  is  a 
$G$-equivariant  covering  homeomorphism   
$f':|C'^{an}|\to   |C'^{an}_{1}|$   
such that 
$f'(z_{0})=z_{1}$.  It   is   biholomorphic   since   both   
$p^{an}|_{C'^{an}}$   and 
$p_{1}^{an}|_{C'^{an}_{1}}$ are locally biholomorphic and it  
can be extended to a 
$G$-equivariant biholomorphic covering map of the Riemann  surfaces  
$f:C^{an}\to C_{1}^{an}$, which yields a $G$-equivariant 
covering isomorphism of the algebraic 
curves $C$ and $C_{1}$. Hence the correspondence is injective.
\par
Given a pair $(D,m)$  one  constructs  a  pointed  $G$-cover  $(C,z_0)\to 
(Y,y_0)$ whose monodromy invariant is $(D,m)$  as  follows.  Let  $\Gamma  = 
\Ker(m)$. Let
\begin{equation}\label{e3.4}
C' = \{\Gamma[\alpha]|\alpha:I\to Y\setminus D\:  \text{is  a  path  with}\: 
\alpha(0)=y_{0}\}.
\end{equation}
Here $[\alpha]$ is the homotopy class of $\alpha$ in $Y\setminus 
D$. The map 
$p':C'\to Y\setminus D$ is defined by $p'(\Gamma  [\alpha])=\alpha(1)$.  One 
lets $z_{0}=\Gamma [c_{y_{0}}]$, where  $c_{y_{0}}$  is  the  constant  path 
$c_{y_{0}}(t)=y_{0}$ for $\forall t\in I$. One defines an action of  $G$  on 
$C$ as follows: if $z=\Gamma [\alpha]$ and $g=m([\sigma])$, one lets 
$gz=\Gamma [\sigma\cdot \alpha]$. The group $G$  acts  transitively  without 
fixed points on the fibers of $p':C'\to Y\setminus D$. One endows $C'$  with 
a Hausdorff topology by the following basis of open sets: for every path 
$\alpha :I\to Y\setminus D$ with $\alpha(0)=y_{0}$ and every  embedded  open 
disk $U\subset Y\setminus D$ which contains $\alpha(1)$ one lets
\[
N_{\alpha}(U) = \{\Gamma [\alpha \cdot \tau]|\tau :I\to U \: 
\text{is  a  path such that }\: \tau(0)=\alpha(1)\} 
\]
Then $(C',p')$ is a connected covering space of $Y\setminus  D$,  the  group 
$G$ acts by covering transformations and $p'$ induces a homeomorphism 
$C'/G \overset{\sim}{\lto}Y\setminus D$. Let us verify  that  the  monodromy 
homomorphism of the pointed topological $G$-covering map \linebreak
$p':(C',z_{0})\to (Y\setminus D,y_{0})$ coincides with 
$m:\pi_1(Y\setminus  D,y_0)\to  G$.  
Let  $\sigma  :I\to  Y\setminus   D$   be 
a  path in $Y\setminus D$  with  $\sigma(0)=y_{0}$.  The  map 
$\tilde{\sigma}:I\to C'$ defined by  $\tilde{\sigma}(s)=\Gamma[\sigma_{s}]$, 
where $\sigma_{s}:I\to Y\setminus D$ is the path $\sigma_{s}(t)=\sigma(st)$, 
is   a   lifting   of   $\sigma$   in   $C'$   with initial   point 
$z_{0}=\Gamma[c_{y_{0}}]$ and terminal point 
$\tilde{\sigma}(1)=\Gamma[\sigma]$.  
In particular, if $\sigma$ is a closed path and $m([\sigma]) = g$, the  
terminal   
point   of   $\tilde{\sigma}$    is 
$\Gamma[\sigma]=gz_{0}$. This  shows  that  the  monodromy 
homomorphism of $p':(C',z_{0})\to (Y\setminus D,y_{0})$ coincides with 
$m:\pi_1(Y\setminus D,y_0)\tto G$. 
\par
One endows $C'$ with the unique complex analytic structure such that 
$p':C'\to Y\setminus D$ is a holomorphic, locally biholomorphic map. 
Compactifying one obtains a holomorphic map of compact Riemann surfaces 
$p:C\to Y$ branched in $D$ and the action of $G$ on $C'$ is extended to an 
action of $G$ on $C$ by biholomorphic maps. Finally $C$ has a 
structure of a projective, nonsingular, irreducible curve, whose associated 
structure of a complex analytic variety coincides with the one above,
$p:C\to Y$ is a morphism, the action of $G$ is by algebraic automorphisms 
and $p:C\to Y$ is a Galois cover with Galois group $G$.
\end{block}
\begin{definition}\label{3.5a}
Let $g=g(Y)$, let $n$ be a positive integer.
Let $G$ be a finite group which can be generated by $2g+n-1$  elements.  Let 
$y_{0}\in Y$. We denote by $H^G_n(Y,y_0)$ the set of pairs $(D,m)$ where 
$D\in (Y\setminus y_0)^{(n)}_{\ast}$ and $m:\pi_1(Y\setminus D,y_0)\tto G$ is 
a surjective homomorphism which satisfies Condition~(\ref{e3.3}).
\end{definition}
\begin{block}\label{3.5b}
The set $H^G_n(Y,y_0)$ is nonempty  and  it  is  bijective  to  the  set  of 
$G$-equivalence  classes  $[p:C\to  Y,z_0]$  of  pointed  $G$-covers   of 
$(Y,y_0)$ (cf. \S~\ref{3.3}). 
 One endows $H^G_n(Y,y_0)$  with  a  Hausdorff  topology  as 
follows. Let $D=\{b_1,\ldots,b_n\}$,
let $\overline{U}_{1},\ldots,\overline{U}_{n}$ be embedded closed disks in 
$Y\setminus y_0$
which are disjoint and such  that  $b_{i}\in 
U_{i}$  for  $\forall  i$, where $U_{i}$ is the interior of 
$\overline{U}_{i}$. 
 Let 
$N_D(U_1,\ldots,U_n)\subset (Y\setminus y_0)^{(n)}_{\ast}$  be  the  open  set 
consisting of $E=\{y_{1},\ldots,y_{n}\}$ such that $y_{i}\in U_{i}$ for every $i$. 
The inclusion \linebreak
$Y\setminus \cup_{i=1}^{n} U_{i}\hookrightarrow Y\setminus 
D$ is a deformation retract, so  for  every  homomorphism \linebreak
$m:\pi_1(Y\setminus D,y_0)\to G$
 and  every  $E\in N_D(U_1,\ldots,U_n)$ there is a unique  homomorphism 
$m(E): \pi_1(Y\setminus E,y_0)\to G$ such that the following diagram commutes
\begin{equation}\label{e3.6}
\xymatrix{
\pi_1(Y\setminus D,y_0)\ar[rd]_{m}&
\pi_1(Y\setminus \cup_{i=1}^{n} U_i,y_0)\ar[l]_{\cong}\ar[r]^{\cong}\ar[d]&
\pi_1(Y\setminus E,y_0)\ar[dl]^{m(E)}\\
&G
}
\end{equation}
Given a closed path $\gamma :I\to Y\setminus E$ based at $y_{0}$ we denote by 
$[\gamma]_{E}$ 
its homotopy class in $Y\setminus E$.  The  homomorphism  $m(E)$ 
is  uniquely 
determined by the following property. For every closed path
$\gamma:I\to Y\setminus \cup_{i=1}^{n} \overline{U}_i$ based at $y_{0}$  the 
following equality holds:
\begin{equation}\label{e3.6a}
m([\gamma]_{D})   =   m(E)([\gamma]_{E})\: \:   \text{for}\: \:   \forall    
E\in 
N_D(U_1,\ldots,U_n).
\end{equation}
We   recall   some   known   facts    about    $H^G_n(Y,y_0)$    (see    e.g. 
\cite[Section~1]{K1}). Let
\begin{equation}\label{e3.6a1}
N_{(D,m)}(U_1,\ldots,U_n) = \{(E,m(E))|E\in N_D(U_1,\ldots,U_n)\}.
\end{equation}
One defines Hausdorff topology on $H^G_n(Y,y_0)$ by choosing as a basis  the 
family of all sets $N_{(D,m)}(U_1,\ldots,U_n)$. 
Let
$\delta:H^G_n(Y,y_0)\to (Y\setminus y_0)^{(n)}_{\ast}$ be the map
defined by $\delta((D,m))=D$. Then $\delta$ is a topological  covering  map. 
The topological covering space $H^G_n(Y,y_0)$ inherits the  structure  of  a 
complex manifold from $(Y\setminus  y_0)^{(n)}_{\ast}$  and  $\delta$  is  a 
holomorphic map. Furthermore by  Th\'{e}or\`{e}me~5.1,  Proposition~3.1  and 
Proposition~3.2 of \cite{SGA1} $H^G_n(Y,y_0)$ has a structure of an algebraic  
variety, $\delta:H^G_n(Y,y_0)\to   
(Y\setminus y_0)^{(n)}_{\ast}$  is  a  finite,  \'{e}tale,  surjective  morphism,
and    the 
associated complex analytic space and holomorphic map coincide with the ones 
defined above. Furthermore 
the algebraic variety $H^G_n(Y,y_0)$ is  nonsingular and affine since 
this is true for $(Y\setminus y_0)^{(n)}_{\ast}$.
\par
Consider the subset $B\subset Y\times H^G_n(Y,y_0)$ defined by
\begin{equation}\label{e3.6a2}
B = \{(y,(D,m))|y\in D\}.
\end{equation}
The map 
$id_{Y}\times  \delta  :   Y\times   H^G_n(Y,y_0)\to   Y\times   (Y\setminus 
y_0)^{(n)}_{\ast}$ is a finite, \'{e}tale, surjective morphism  and  $B$  is 
the   preimage   of   the    universal    divisor    $Y\times    (Y\setminus 
y_0)^{(n)}_{\ast}\cap A$.
\end{block}
\begin{definition}\label{3.7}
Let $G$ be a finite group. Let $n$ be a positive integer. Let  $\mathcal{C}$ 
and $S$ be algebraic varieties. A morphism $p:\mathcal{C}\to Y\times  S$  is 
called a \emph{smooth family of $G$-covers of $Y$ branched in  $n$  points} 
if:
\begin{enumerate}
\item
$p$ satisfies the conditions of Definition~\ref{2.6};
\item
$G$ acts on $\mathcal{C}$ on the left by automorphisms of $\mathcal{C}$, 
$p:\mathcal{C}\to Y\times S$ is $G$-invariant and  
$p_s:\mathcal{C}_s\to Y\times \{s\}$ is a $G$-cover for $\forall s\in S$ (cf. Definition~\ref{3.1}).
\item
Two such families $p:\mathcal{C}\to Y\times S$ and $p_{1}:\mathcal{C}_{1}\to 
Y\times S$ are  called  $G$-equivalent  if  there  exists  a  $G$-equivariant 
isomorphism $f:\mathcal{C}\to \mathcal{C}_{1}$ such that $p=p_{1}\circ f$.
\end{enumerate}
It is clear that two families are equivalent if and only if there is an $S$-isomorphism
$f: \mathcal{C}\to \mathcal{C}_1$ which is a covering $G$-isomorphism over $Y$ for $\forall s\in S$,
i.e. $p_s=(p_1)_s\circ f_s$ for $\forall s\in S$.
\end{definition}
\begin{definition}\label{3.7a}
Let $y_{0}\in Y$. A smooth family  of  pointed  $G$-covers  of  $(Y,y_0)$ 
branched in $n$ points is a  pair 
$(p:\mathcal{C}\to  Y\times  S,\zeta:S\to \mathcal{C})$, where $p$ satisfies 
the conditions of Definition~\ref{3.7},
$p_{s}:\mathcal{C}_{s}\to Y\times \{s\}$ is unramified  at  $(y_{0},s)$  for 
$\forall s\in S$ and  $\zeta:S\to  \mathcal{C}$  is  a  morphism  such  that 
$\zeta(s)\in p^{-1}_{s}(y_{0},s)$ for $\forall s\in S$. Two such families 
$(p:\mathcal{C}\to  Y\times  S,\zeta)$ and
$(p_{1}:\mathcal{C}_{1}\to Y\times  S,\zeta_1)$  are  called  $G$-equivalent  if 
there    exists    a    $G$-equivariant    isomorphism     $f:\mathcal{C}\to 
\mathcal{C}_{1}$ such that $p=p_{1}\circ f$ and $\zeta_{1}=f\circ \zeta$
\end{definition}
If two families are $G$-equivalent,  then  the  $G$-equivariant  isomorphism 
$f:\mathcal{C}\to   \mathcal{C}_{1}$   is   unique   in    the    case    of 
Definition~\ref{3.7a} and, provided $G$ has trivial center, it is unique in  the 
case of Definition~\ref{3.7}.
\begin{block}\label{3.8}
Our goal in this section is: given a positive integer $n$ and a point $y_0\in 
Y$ to construct a smooth family of pointed $G$-covers of $(Y,y_0)$  branched 
in $n$ points
\[
\left(p:\mathcal{C}(y_{0})\to             Y\times              H^G_n(Y,y_0), 
\zeta:H^G_n(Y,y_0)\to \mathcal{C}(y_{0})\right)
\]
with the property that every fiber 
$\left(\mathcal{C}(y_{0})_{(D,m)}\to Y, \zeta(D,m)\right)$ is  a  pointed  
$G$-cover 
of $(Y,y_0)$ with monodromy invariant $(D,m)$. This is obtained by the 
following steps:
\begin{enumerate}
\item
One   constructs   explicitly   a   set $\mathcal{C}(y_0)'$, a surjective map 
   $p': \mathcal{C}(y_0)'\to    \linebreak
        Y\times H^G_n(Y,y_0)\setminus B$ and an action of $G$ on 
$\mathcal{C}(y_0)'$.
\item
One endows $\mathcal{C}(y_0)'$ with a Hausdorff topology such that 
$p':  \mathcal{C}(y_0)'\to  \linebreak
Y\times  H^G_n(Y,y_0)\setminus  B$   becomes   a 
topological  covering  map,  one  verifies  that  $G$   acts   by   covering 
transformations and that $p'$ is a topological $G$-covering map.
\item
One endows $\mathcal{C}(y_0)'$ with a structure of a complex analytic manifold 
inherited from $Y\times H^G_n(Y,y_0)\setminus B$. The map  $p'$  becomes  an 
\'{e}tale Galois holomorphic covering map.
\item
Using \cite[Th\'{e}or\`{e}me~5.1]{SGA1} one endows $\mathcal{C}(y_0)'$ with  a 
structure of an algebraic variety  and  $p'$  becomes  an  \'{e}tale  Galois 
cover.
\item
One constructs $\mathcal{C}(y_0)$ and 
$p:\mathcal{C}(y_0)\to Y\times H^G_n(Y,y_0)$ by the normal closures  of  the 
irreducible components of $Y\times H^G_n(Y,y_0)$ in the fields  of  rational 
functions of the irreducible components of $\mathcal{C}(y_0)'$. 
\end{enumerate}
\end{block}
\begin{remark}
The construction of $\mathcal{C}(y_0)'$ and its topology in (i) and (ii) is similar  to  the 
construction of universal covering spaces (see e.g. \cite[Ch.~V  Theorem~10.2]{Mas} 
or \cite[Ch.~1 \S~5]{For}).
For $Y = \mathbb{P}_1$ Parts (i) -- (iv) are equivalent to Emsalem's construction 
of the family of pointed \'{e}tale $G$-morphisms of $\mathbb{P}_1$ parameterized 
by the Hurwitz space $H^G_n(\mathbb{P}_1,y_0)$
(cf. \cite[\S~6 and \S~7.1]{Em1}).
\end{remark}
\begin{block}\label{3.10}
If $\alpha, \alpha'$ are paths in $Y\setminus D$ we denote by 
$\alpha  \sim_{D}  \alpha'$  the  homotopy  of  $\alpha$  and 
$\alpha'$ in $Y\setminus D$ and by $[\alpha]_{D}$ the  homotopy   
class of $\alpha$ in $Y\setminus D$. In the set-up of \S~\ref{3.5b} let 
$E\in N_D(U_1,\ldots,U_n)$ and let $\alpha, \alpha'$ be paths in 
$Y\setminus \cup_{i=1}^nU_i$. Then $\alpha \sim_{D} \alpha'$ if  and  only 
if $\alpha \sim_{E} \alpha'$. Let $(D,m)\in H^G_n(Y,y_0)$.  We  denote  by 
$\Gamma_{m}$ the kernel of $m:\pi_1(Y\setminus D,y_0)\tto G$. 
Let $\alpha, \alpha'$ be paths in 
$Y\setminus \cup_{i=1}^nU_i$ such that $\alpha(0)=\alpha'(0)=y_{0}$ and 
$\alpha(1)=\alpha'(1)$.     Then     by     \eqref{e3.6} 
$[\alpha'\cdot \alpha^{-}]_{D}\in \Gamma_{m}$ if and only if 
$[\alpha'\cdot \alpha^{-}]_{E}\in \Gamma_{m(E)}$, hence 
$\Gamma_{m}[\alpha]_{D} = \Gamma_{m}[\alpha']_{D}$ if and only if 
$\Gamma_{m(E)}[\alpha]_{E} = \Gamma_{m(E)}[\alpha']_{E}$.
\end{block}
\begin{block}\label{3.11}
We denote by $\mathcal{C}(y_0)'$ the set
\[
\mathcal{C}(y_0)' =  \{\left(\Gamma_m[\alpha]_D,D,m\right)|(D,m)\in 
H^G_n(Y,y_0),\: 
\alpha:I \to Y\setminus D,\: \alpha(0)=y_{0}\}.
\]
Let $p': \mathcal{C}(y_0)'\to Y\times H^G_n(Y,y_0)\setminus B$ be the map
\[
p'\left(\Gamma_m[\alpha]_D,D,m)\right) = (\alpha(1),D,m).
\]
We define a left action of $G$ on $\mathcal{C}(y_0)'$ as follows. For  every 
$g\in G$, if $g=m([\sigma]_{D})$, one lets
\[
g(\Gamma_m[\alpha]_D,D,m) = (\Gamma_m[\sigma \cdot \alpha]_D,D,m).
\]
The map $p'$ is $G$-invariant,
the isotropy subgroup of every $z\in \mathcal{C}(y_0)'$ is trivial  and  $G$ 
acts transitively on every fiber of $p'$, hence $p'$ induces a bijection
$\mathcal{C}(y_0)'/G \overset{\sim}{\lto} Y\times H^G_n(Y,y_0)\setminus B$.
\par
Let $z=(\Gamma_m[\alpha]_D,D,m)\in \mathcal{C}(y_0)'$. Let 
$p'(z)=(y,(D,m))$, where $D=\{b_1,\ldots,b_n\}$, $y\in Y\setminus D$. Let 
$\overline{U},\overline{U}_{1},\ldots,\overline{U}_{n}$ be disjoint, embedded
closed disks in $Y$ with interiors $U,U_1,\ldots,U_n$ respectively, with  the  
property  that $\overline{U}_{i}\subset Y\setminus y_0$ for $\forall i$, 
$y\in U$, $b_i\in U_i$ for $\forall i$.
 Let $\alpha :I\to Y\setminus \cup_{i=1}^n\overline{U}_i$ 
be a path such that $\alpha(0)=y_{0}, \alpha(1)=y$. Consider  the  following 
subset of $\mathcal{C}(y_0)'$:
\begin{equation}\label{e3.12}
\begin{split}
N_{(\alpha,D,m)}(U,U_1,\ldots,U_n) =
\{&(\Gamma_{m(E)}[\alpha \cdot \tau]_E,E,m(E))| \\
&E\in N_D(U_1,\ldots,U_n), 
\tau:I\to U, \tau(0)=y\}
\end{split}
\end{equation}
\end{block}
\begin{proposition} \label{3.12}
Let $p':\mathcal{C}(y_0)'\to  Y\times  H^G_n(Y,y_0)\setminus  B$ be the 
$G$-invariant map defined in \S~\ref{3.11}.
\begin{enumerate}
\item
The family of sets defined in \eqref{e3.12} is a basis of a topology of 
$\mathcal{C}(y_0)'$.
\item
The map  $p'$  is  a 
topological covering map.
\item
The topology defined in (i) is Hausdorff.
\item
The group $G$ acts on $\mathcal{C}(y_0)'$ freely by Deck transformations and 
$p'$ induces a homeomorphism 
$\mathcal{C}(y_0)'/G  \overset{\sim}{\lto}  Y\times   H^G_n(Y,y_0)\setminus 
B$. 
\item
The map $\zeta : H^G_n(Y,y_0) \to \mathcal{C}(y_0)'$ defined by
\[
\zeta(D,m) = \left(\Gamma_m[c_{y_{0}}]_D,D,m\right),
\]
where $c_{y_{0}}$ is the constant loop, is a continuous section of 
$\pi_{2}\circ p': \mathcal{C}(y_0)'\to H^G_n(Y,y_0)$, such that 
$p'\circ  \zeta  (D,m)  =  \left(y_{0},(D,m)\right)$  for   every   $(D,m)\in \linebreak
H^G_n(Y,y_0)$.
\item
For every $(D,m)\in H^G_n(Y,y_0)$ the couple 
\[
(p'_{(D,m)}:\mathcal{C}(y_0)'_{(D,m)}\to Y\setminus D,
\zeta(D,m))
\]
is  a  pointed  topological  $G$-covering map of  $(Y\setminus  D,y_{0})$  with 
monodromy homomorphism equal to $m:\pi_1(Y\setminus D,y_0)\tto G$.
\end{enumerate}
\end{proposition}
\begin{proof}
(i)  It  is  obvious  that  $\mathcal{C}(y_0)'$  is  a  union  of  the  sets 
\eqref{e3.12}. Let 
\[
W'=N_{(\alpha',D',m')}(U',U'_1,\ldots,U'_n),\quad
W''=N_{(\alpha'',D'',m'')}(U'',U''_1,\ldots,U''_n).
\] 
Let $z\in W'\cap  W''$ 
and  let  $z=(\Gamma_m[\beta]_D,D,m)$.  Let  $p'(z)=(y,(D,m))$,  $D= \linebreak
\{b_1,\ldots,b_n\}$.  Then  $y\in  U'\cap  U''$,   $b_{i}\in   U'_i\cap   U''_i$  
for $i=1,\ldots,n$ and furthermore 
$y\in Y\setminus \left(\cup_{i=1}^n\overline{U}'_i\bigcup 
\cup_{i=1}^n\overline{U}''_i\right)$. One has $m=m'(D)=m''(D)$, 
\[
\Gamma_m[\beta]_D = \Gamma_{m'(D)}[\alpha'\cdot \tau']_D = 
\Gamma_{m''(D)}[\alpha''\cdot \tau'']_D.
\]
 Let us choose  disjoint closed disks 
$\overline{U},\overline{U}_{1},\ldots, \overline{U}_{n}$
with interiors $U,U_1,\ldots,U_n$ respectively, such that $U\ni y$, 
$U_{i}\ni b_{i}, i=1,\ldots,n$, $\overline{U}\subset U'\cap 
U'',\: \overline{U}_{i}\subset U'_{i}\cap U''_{i}$ for $\forall  i$  and  
 $\beta(I)\subset Y\setminus \cup_{i=1}^n\overline{U}_i$.
Let $W = N_{(\beta,D,m)}(U,U_1,\ldots,U_n)$.  
We  claim  that  $W\subset  W'\cap W''$. Let $x\in W$,  
$x= (\Gamma_{m(E)}[\beta\cdot  \tau]_E,E,m(E))$.  It  is 
clear from  Diagram~\eqref{e3.6} that $m(E)=m'(E)=m''(E)$. One has 
$D\in N_{D'}(U'_1,\ldots,U'_n)$, so $\beta \sim_{D} \beta'$, where $\beta'$  is 
a path in $Y\setminus \cup_{i=1}^nU'_i$. The equality 
$\Gamma_m[\beta']_D = \Gamma_m[\beta]_D = \Gamma_m[\alpha'\cdot \tau']_D$ 
implies 
$\Gamma_{m(E)}[\beta']_E = \Gamma_{m(E)}[\alpha'\cdot \tau']_E$ since
$E\in N_{D'}(U'_1,\ldots,U'_n)$ (cf. \S~\ref{3.10}). Therefore
\[
x = (\Gamma_{m(E)}[\beta\cdot\tau]_E,E,m(E)) = 
(\Gamma_{m'(E)}[\alpha'\cdot\tau'\cdot\tau]_E,E,m'(E))
\]
belongs to $W'$, since $\tau'\cdot\tau$ is an arc in $U'$.  Similarly  $x\in 
W''$. This shows that $W\subset W'\cap W''$.
\par
(ii) Let $(y,(D,m))\in Y\times H^G_n(Y,y_0)\setminus B$,  let  
$D=\{b_1,\cdots,b_n\}$. 
Choose  paths  $\alpha_1,\ldots,\alpha_{|G|}$  in   $Y\setminus   D$   with 
$\alpha_i(0)=y_{0},    \alpha_{i}(1)=y$    such    that    $[\alpha_{i}\cdot 
\alpha_{j}^{-}]_{D}\notin \Gamma_{m}$ if $i\neq j$. Then for every path 
$\alpha:I\to Y\setminus D$ with $\alpha(0)=y_{0}$, $\alpha(1)=y$ there is  a 
unique $\alpha_{i}$ such that $\Gamma_m[\alpha]_D = 
\Gamma_m[\alpha_{i}]_D$. Let 
$\overline{U}_{1},\ldots,\overline{U}_{n}$ be disjoint embedded closed disks 
in $Y$
with interiors $U_{1},\ldots,U_{n}$  respectively, such that $b_i\in U_i$ for 
$i=1,\ldots,n$ and
$\alpha_{j}(I)\subset Y\setminus \cup_{i=1}^n\overline{U}_i$ 
for every $j=1,\ldots,|G|$.  Let $\overline{U}\subset Y$ be an embedded closed 
disk  such  that 
$y$ belongs to its interior $U$ and 
$\overline{U}\subset Y\setminus \cup_{i=1}^n\overline{U}_i$. 
Then
\begin{equation}\label{e3.15}
p'^{-1}\left(U\times N_{(D,m)}(U_1,\ldots,U_n)\right) = 
\bigcup_{j=1}^{|G|}N_{(\alpha_{j},D,m)}(U,U_1,\ldots,U_n)
\end{equation}
and moreover 
$N_{(\alpha_{i},D,m)}(U,U_1,\ldots,U_n)\cap 
N_{(\alpha_{j},D,m)}(U,U_1,\ldots,U_n) = \emptyset$ if $i\neq j$. 
In fact, it is clear that the left-hand set of \eqref{e3.15} contains the 
right-hand set. 
Let $z=(\Gamma_{\mu}[\beta]_E,E,\mu)$ be a point of the left-hand set. Then 
$E\in N_D(U_1,\ldots,U_n)$, $\mu = m(E)$. Let $\beta \sim_{E}\beta'$, where 
$\beta'$ is a path in $Y\setminus \cup_{i=1}^n\overline{U}_i$. Let 
$\Gamma_m[\beta']_D=\Gamma_m[\alpha_{i}\cdot \tau]_D$ for some 
$\tau:I\to        U$.        Then        $\Gamma_{m(E)}[\beta']_{E}        = 
\Gamma_{m(E)}[\alpha_i\cdot \tau]_{E}$ (cf. \S~\ref{3.10}). Hence 
$z=(\Gamma_{m(E)}[\alpha_{i}\cdot \tau]_E,E,m(E)) \in 
N_{(\alpha_{i},D,m)}(U,U_1,\ldots,U_n)$. This proves  Equality~\eqref{e3.15} 
and in particular shows that $p'$ is a continuous map.
Suppose that \linebreak
$(\Gamma_{m(E)}[\alpha_{i}\cdot \tau]_E,E,m(E)) =
(\Gamma_{m(E)}[\alpha_{j}\cdot \tau']_E,E,m(E))$ for some $i\neq j$ and some 
paths $\tau, \tau'$ in $U$. One  has  $\tau  \sim_{E}\tau'$  since  $U$  is 
simply   connected,   therefore 
$[\alpha_{j}\cdot   \alpha_{i}^{-}]_{E}\in \Gamma_{m(E)}$. This implies that 
$[\alpha_{j}\cdot  \alpha_{i}^{-}]_{D}\in  \Gamma_{m}$,  which   contradicts 
the choice of $\{\alpha_{i}\}_{i}$. This shows that the right-hand  side  of 
\eqref{e3.15} is a disjoint union. 
\par
Every open set $N_{(\alpha,D,m)}(U,U_1,\ldots,U_n)$ as in  \eqref{e3.12}  is 
mapped by $p'$ bijectively onto the open subset $U\times  
N_{(D,m)}(U_1,\ldots,U_n)$
of $Y\times H^G_n(Y,y_0)\setminus B$.  Since 
every open subset of $N_{(\alpha,D,m)}(U,U_1,\ldots,U_n)$ is a union of sets 
$N_{(\beta,A,\mu)}(V,V_1,\ldots,V_n)$ as in \eqref{e3.12}, this bijection, 
being a continuous  map,  is  open, 
hence it is a homeomorphism. This proves (ii).
\par
(iii) This follows from  (ii)  since  
$Y\times H^G_n(Y,y_0)\setminus B$ is a Hausdorff topological space.
\par
(iv) By \S~\ref{3.11} it suffices to prove that for every $g\in G$  the  map 
$\varphi_{g}:\mathcal{C}(y_0)'\to      \mathcal{C}(y_0)'$      defined       
by 
$\varphi_{g}(z)=gz$ is continuous. Every open subset $W$ of  
$\mathcal{C}(y_0)'$ 
is  a  union  of  subsets  of   the   topology   basis   \eqref{e3.12}   and 
$\varphi_{g}^{-1}(W)=g^{-1}W$.  So  it  suffices  to  prove   that   
$\varphi_{g}$ 
transforms every $V=N_{(\alpha,D,m)}(U,U_1,\ldots,U_n)$ into  an  open  set. 
Let $g=m([\sigma]_{D})$, where $\sigma$ is  a  closed  path  in  $Y\setminus 
\cup_{i=1}^n\overline{U}_i$.                  Let 
$x= \linebreak
(\Gamma_{m(E)}[\alpha\cdot\tau]_E,E,m(E))\in    V$.     One     has     by 
Diagram~\eqref{e3.6} that $g=m(E)([\sigma]_{E})$, hence 
$gx=(\Gamma_{m(E)}[\sigma\cdot\alpha\cdot\tau]_E,E,m(E))$. One obtains that 
$gV\subset  N_{(\sigma\cdot\alpha,D,m)}(U,U_1,\ldots,U_n)$.  Replacing   the 
pair $g,\alpha$  by  $g^{-1},\sigma\cdot\alpha$  one  obtains  the  opposite 
inclusion, therefore
\begin{equation}\label{e3.17}
gN_{(\alpha,D,m)}(U,U_1,\ldots,U_n)                                        = 
N_{(\sigma\cdot\alpha,D,m)}(U,U_1,\ldots,U_n).
\end{equation}
\par
(v) Suppose that $\zeta(D,m)\in N_{(\alpha,A,\mu)}(V,V_1,\ldots,V_n)$. 
Let $\underline{V}=(V_1,\ldots,V_n)$. Then $D\in N_A(V_1,\ldots,V_n)$, 
$m=\mu(D)$, therefore $\zeta^{-1}N_{(\alpha,A,\mu)}(V,\underline{V})
\subset N_{(A,\mu)}(\underline{V})$. Moreover one has 
$(\Gamma_m[c_{y_{0}}]_D,D,m)=(\Gamma_{\mu(D)}[\alpha\cdot \tau]_D,D,\mu(D))$ 
for some path $\tau:I\to V$ such that $\tau(0)=\alpha(1)$,  $\tau(1)=y_{0}$, 
so $[\alpha\cdot \tau]_D\in \Gamma_{\mu(D)}$. This implies that  
$[\alpha\cdot \tau]_{E}\in  \Gamma_{\mu(E)}$  for  every  
$E\in   N(\underline{V})$   (cf. \S~\ref{3.10}). Hence for every  
$(E,\mu(E))\in  N_{(A,\mu)}(\underline{V})$ 
one has 
\[
\zeta(E,\mu(E)) = (\Gamma_{\mu(E)}[c_{y_{0}}]_E,E,\mu(E))
=  (\Gamma_{\mu(E)}[\alpha\cdot \tau]_E,E,\mu(E))       
\in N_{(\alpha,A,\mu)}(V,\underline{V}),
\]
so             $\zeta^{-1}N_{(\alpha,A,\mu)}(V,\underline{V})              = 
N_{(A,\mu)(\underline{V})}$. This proves that $\zeta$ is  a  continuous  map. 
The other statements of (v) are obvious.
\par
(vi) Let $(D,m)\in H^G_n(Y,y_0)$. Identifying $(Y\setminus D)\times \{(D,m)\}$ 
with $Y\setminus D$, the set $\mathcal{C}(y_0)'_{(D,m)}$, the action  of  $G$ 
on it, its induced topology, the map $p'_{(D,m)}$ and the point $\zeta(D,m)$ 
are the same as those  defined  in  \S~\ref{3.3},  therefore  the  monodromy 
homomorphism of the pointed topological $G$-covering map
$(p'_{(D,m)},\zeta(D,m))$ of \linebreak
$(Y\setminus D,y_0)$ equals $m:\pi_1(Y\setminus 
D,y_0)\tto G$.
\end{proof}
\begin{proposition}\label{3.19}
Let $p':\mathcal{C}(y_0)'\to   Y\times   H^G_n(Y,y_0)\setminus   B$  be  the 
topological $G$-covering map of Proposition~\ref{3.12}.
\begin{enumerate}
\item
Consider $Y\times H^G_n(Y,y_0)$ with the structure  of  a  complex  analytic 
manifold of dimension $n+1$ (cf.  \S~\ref{3.5b}).  Then  $\mathcal{C}(y_0)'$ 
has a unique structure of a complex analytic  manifold  of  dimension  $n+1$ 
such that $p'$ is a holomorphic map. Furthermore $p'$ is a finite, \'{e}tale 
holomorphic covering map (cf. \cite[\S~{5.0}]{SGA1}). 
\item
Consider $Y\times H^G_n(Y,y_0)$ with the structure of an  algebraic  variety 
as in  \S~\ref{3.5b}.  Then  $\mathcal{C}(y_0)'$  and  $p'$  have  a  unique 
structure of an algebraic variety and a finite, \'{e}tale surjective morphism to 
$Y\times H^G_n(Y,y_0)\setminus B$ whose associated  complex  analytic  space 
and  holomorphic   map   are   those   of   (i).   The   algebraic   variety 
$\mathcal{C}(y_0)'$ is nonsingular, equidimensional of dimension $n+1$. 
\item
The group  $G$  acts freely by  covering automorphisms 
on the algebraic variety $\mathcal{C}(y_0)'$ and $p'$ induces an isomorphism 
$\mathcal{C}(y_0)'/G \cong Y\times H^G_n(Y,y_0)\setminus B$.
\item
The  map  $\zeta:  H^G_n(Y,y_0)\to  \mathcal{C}(y_0)'$  is  a  morphism   of 
algebraic varieties. It is a closed embedding, a section of the morphism
$\pi_2\circ p':\mathcal{C}(y_0)'\to H^G_n(Y,y_0)\setminus B$ and 
$p'\circ \zeta(D,m)=(y_0,(D,m))$ for $\forall (D,m)\in H^G_n(Y,y_0)$.
\end{enumerate}
\end{proposition}
\begin{proof}
(i) One defines a complex analytic structure  on  the  topological  covering 
space   $\mathcal{C}(y_0)'$   as    in    \cite[Ch.~IX    \S~1]{Sha2}.    Let 
$M=\mathcal{C}(y_0)'$ be the obtained complex manifold. Suppose that $M'$ is 
another complex analytic manifold whose underlying topological space is 
$\mathcal{C}(y_0)'$, such that $p'$ is a holomorphic map. We claim that 
$id:M\to  M'$  is  biholomorphic.  Let  $z\in  \mathcal{C}(y_0)'$  and   let 
$x=p'(z)$. Let $(W,\psi)$ be a complex analytic chart of $x$ such  that  $W$ 
is evenly covered and let $U$ be a neighborhood of $z$ in $M$ such that 
$p'|_{U}:U\to W$ is biholomorphic. Let $(V,\phi)$ be a chart  of  $M'$  such 
that $z\in V\subset U$. Then $\psi\circ p'\circ  \phi^{-1}:\phi(V)\to  (\psi 
\circ   p')(V)$   is   a   holomorphic   bijective   map   of   domains   in 
$\mathbb{C}^{n+1}$. Hence by Clements'  theorem  it  is  biholomorphic 
\cite[Ch.~5 Theorem~5]{Nar}. Therefore the restriction of $id:M\to M'$ on $V$ 
is biholomorphic.  This  shows  that  $id:M\to  M'$  is  biholomorphic.  The 
topological  covering  map  $p'$  is  proper  with  finite  fibers,  so  the 
holomorphic map $p'$ is  finite.  It  is  locally  biholomorphic,  so  it  is 
\'{e}tale.  It  remains  to  verify  that  every  irreducible  component  of 
$\mathcal{C}(y_0)'$  dominates   an   irreducible   component   of   $Y\times 
H^G_n(Y,y_0)\setminus B$. Since these are complex manifolds, the irreducible 
components coincide with the connected components (cf. \cite[Ch.~9 \S~2 
n.1]{GR}), so this property is obvious.
\par
(ii) By \cite[Th\'{e}or\`{e}me~5.1]{SGA1} the pair
$(\mathcal{C}(y_0)',p')$ has a unique structure of a scheme  of  finite  type  
over $\mathbb{C}$ and a finite \'{e}tale  morphism  of  schemes,  such that the 
associated complex analytic space and holomorphic map are those of (i).  The 
scheme $\mathcal{C}(y_0)'$ is separated and smooth by  Proposition~3.1(viii) 
and Proposition~2.1(iv) of \cite{SGA1}.
\par
(iii) The group $G$ acts by  covering  transformations  of  the  topological 
covering  map  $p'$  (cf.  Proposition~\ref{3.12}(iv)),  so  $G$   acts   by 
biholomorphic covering maps of the finite, \'{e}tale,  holomorphic  covering map 
$p'$.  By  \cite[Th\'{e}or\`{e}me~5.1~(1)]{SGA1}  the  group  $G$  acts   by 
covering automorphisms  of  the  finite,  \'{e}tale  morphism  of  algebraic 
varieties 
$p':\mathcal{C}(y_0)'\to Y\times H^G_n(Y,y_0)\setminus B$. By Lemma~\ref{6.16} 
$p'$ induces a finite morphism 
$\overline{p}':\mathcal{C}(y_0)'/G\to Y\times  H^G_n(Y,y_0)\setminus  B$. It  is 
bijective (cf.  \S~\ref{3.11})  and  $Y\times  H^G_n(Y,y_0)\setminus  B$  is 
smooth, therefore $\overline{p}'$ is an isomorphism.
\par
(iv) $p'^{-1}(\{y_{0}\}\times H^G_n(Y,y_0))$ is a closed algebraic subset of 
$\mathcal{C}(y_0)'$ and the restriction of $p'$ on it is a finite  \'{e}tale 
morphism onto $\{y_{0}\}\times H^G_n(Y,y_0)$. Identifying $H^G_n(Y,y_0)$ with 
$\{y_{0}\}\times H^G_n(Y,y_0)$ the map $\zeta$ is its continuous section  in 
the Euclidean topology by Proposition~\ref{3.12}(v), hence 
$\zeta(H^G_n(Y,y_0))$ is a union of connected components of 
$p'^{-1}(\{y_{0}\}\times  H^G_n(Y,y_0))$.  These  connected  components  in 
the Euclidean topology are connected  components  in  the  Zariski  topology 
\cite[Corollaire~2.6]{SGA1}, so $\zeta(H^G_n(Y,y_0))$ is a closed  algebraic 
subset of $\mathcal{C}(y_0)'$. The restriction of $p'$ on it  is  a  finite, 
\'{e}tale, bijective morphism onto the smooth variety 
$\{y_{0}\}\times H^G_n(Y,y_0)$, hence it is an isomorphism. This shows that 
$\zeta: H^G_n(Y,y_0)\to \mathcal{C}(y_0)'$ is a morphism, it is a closed 
embedding which is a 
section of the morphism $p'$. The last statement is from 
Proposition~\ref{3.12}(v).
\end{proof}
\begin{block}\label{3.23}
Let $H$ be a connected component of $H^G_n(Y,y_0)$. It is irreducible,  since 
$H^G_n(Y,y_0)$  is  smooth.  Let   $\mathcal{C}(y_0)'_{H}=p'^{-1}(H)$.   Its 
closed   subspace  $\zeta(H)$ is   pathwise   connected    and    every    
fiber    of 
$\mathcal{C}(y_0)'_{H}\to H$ is pathwise connected by 
Proposition~\ref{3.12}(vi). 
Therefore $\mathcal{C}(y_0)'_{H}$ is pathwise connected.  It  is  a  Zariski 
open subset of  the  smooth  algebraic  variety  $\mathcal{C}(y_0)'$,  hence 
$\mathcal{C}(y_0)'_{H}$    is    irreducible.    Let    us     denote     by 
$\mathcal{C}(y_0)_H$ the normalization of  $Y\times  H$  in  the  field  of 
rational     functions     of      $\mathcal{C}(y_0)'_{H}$      and      let 
$p_{H}:\mathcal{C}(y_0)_{H}\to Y\times  H$  be  the  corresponding  finite, 
surjective morphism.  The  action  of  $G$  on  $\mathcal{C}(y_0)'$  can  be 
uniquely extended  to  a faithful  action  of  $G$  on  $\mathcal{C}(y_0)_H$   by 
algebraic  automorphisms  (cf.  Proposition~\ref{3.19}(iii)).  
The uniqueness of normalizations implies that there is 
 a $G$-equivariant     open      embedding      
$j_{H}:\mathcal{C}(y_0)'_{H}\to 
\mathcal{C}(y_0)_H$ with image $p^{-1}(Y\times H\setminus B$) such that 
$p\circ j_{H}=p'$. 
The $G$-invariant morphism 
$p_{H}:\mathcal{C}(y_0)_H \to Y\times H$  is  a  Galois cover.  
 In fact the morphism
$\mathcal{C}(y_0)_H/G \to  Y\times H$ is finite, birational by 
Proposition~\ref{3.19}(iii),  and  $Y\times  H$  is  smooth,  so  it  is   an 
isomorphism.
\end{block}
\begin{block}\label{3.23b}
In the next proposition we study $p_{H}$ at the ramification points. Let 
$(D,m)\in H$ and let $(b,(D,m))\in (Y\times H) \cap B$, $D=\{b_1,\ldots,b_k,\ldots,b_n\}$, $b=b_k$.
Let us choose
local analytic  coordinates  $s_{i}$ at $b_{i}$, such that $s_{i}(b_{i})=0$, $i=1,\ldots,n$.
Let $\epsilon \in \mathbb{R}^+, \epsilon \ll 1$ be such that the 
open   sets   $U_{i}=\{y|s_{i}(y)<\epsilon\}$   have    disjoint    closures 
$\overline{U}_{i}$,     $i=1,\ldots,n$ and $y_{0}\in Y\setminus \cup_{i=1}^n\overline{U}_i$.   
Let $\gamma_1,\ldots,\gamma_n$ be closed paths based at $y_0$ as in 
\S~\ref{3.2}.
 Let    $U=U_{k}$    and    let    $V=U\times N_{(D,m)}(U_1,\ldots,U_n)$.
For every $v=(y,(E,m(E))\in V$, where $y\in U$, $E=\{y_1,\ldots,y_n\}, y_i\in U_i$ let $t_i(v) = s_i(y_i)$
and let $t(v)=s_k(y)$.
\end{block}
\begin{proposition}\label{3.24}
The algebraic variety $\mathcal{C}(y_0)_H$ is smooth. Let 
\[
x = (b,(D,m)) \in (Y\times H) \cap B, \quad D=\{b_1,\ldots,b_k,\ldots,b_n\}, \quad b=b_k
\]
 and let 
$p_{H}^{-1}(x)=\{w_{1},\ldots,w_{r}\}$. Let $G(w_{i})$ be the isotropy group 
of $w_{i}$. Then
\begin{enumerate}
\item
For every $i=1,\ldots,r$  $G(w_{i})$  is a  cyclic  group  generated  by  an 
element conjugated with $g_{k}=m([\gamma_{k}])$, $G(w_{i})$ is of order 
$e = |g_{k}|$, where $er=|G|$. 
\item
Let $V = U\times N_{(D,m)}(U_1,\ldots,U_n)$ be as in \S~\ref{3.23b}. Then 
$p_{H}^{-1}(V)  =  \bigsqcup_{i=1}^{r}W_{i}$,  where  $W_{i}$  is  an  open, 
connected  neighborhood  of   $w_{i}$, it is $G(w_i)$-invariant,   and   $p_{H}(W_{i})=V$   
for   every $i=1,\ldots,r$.
\item
Let $i\in [1,r]$ and let $(W,w) = (W_{i},w_{i})$. Let $E\subset \mathbb{C}\times V$
be the analytic subset defined by the equation $z^e = t - t_k$ and let $p_1:E\to V$
be the projection map.
\begin{enumerate}
\renewcommand{\theenumi}{\alph{enumi}}
\item
There exists a biholomorphic map $\varphi: W\to E$ such that $p_1\circ \varphi = p_H|_W$.
\item
The composition $\psi = (z,t_1,\ldots,t_n)\circ \varphi : W\to \mathbb{C}^{n+1}$ maps $W$
biholomorphically onto an open subset of $\mathbb{C}^{n+1}$.
\item
There exists a primitive character $\chi:G(w)\to \mathbb{C}^*$ such that $\varphi$ and $\psi$
are $G(w)$-equivariant with respect to the actions of $G(w)$ on $E$ and $\mathbb{C}^{n+1}$ defined 
respectively by $g(z,v)=(\chi(g)z,v)$ and $g(z,z_1,\ldots,z_n)=(\chi(g)z,z_1,\ldots,z_n)$.
\end{enumerate}
\item
There is a $G$-equivariant biholomorphic map $p_{H}^{-1}(V) \cong 
G\times^{G(w)}W$.
\end{enumerate}
\end{proposition}
\begin{proof}
We may assume, without loss of generality,  that $k=1$. Let us denote 
$p_{H}:\mathcal{C}(y_0)_H\to Y\times H$ by $p:M\to N$. The map 
$p^{an}:M^{an}\to  N^{an}$  is  a  finite,  surjective, 
holomorphic map and $M^{an}$ is a reduced, normal complex space
\cite[Proposition~2.1]{SGA1}.  Let  $p^{-1}(V)  =  \bigsqcup_{i}W_{i}$  be  the 
disjoint  union  of  connected  components  of  $p^{-1}(V)$.  By  \cite[Ch.~9 
\S~2 n.1]{GR} every $W_{i}$  is  an  open  subset  of  $M^{an}$.  The 
restriction $p|_{W_{i}}:W_{i}\to V$ is a finite holomorphic map. In fact,  it 
has finite fibers and if $A$ is a closed subset of $W_{i}$ then $A$ is closed 
in $p^{-1}(V)$, so $p(A)$ is closed in $V$ since 
$p^{an}:M^{an}\to N^{an}$ is a closed map. Moreover
$p|_{W_{i}}:W_{i}\to V$ is an open map by the Open Mapping Theorem  
(cf.  \cite[Ch.~5 \S~4 n.3]{GR}).  Therefore  $p(W_{i})=V$,  since  $V$   is 
connected. We see that $W_{i}\cap p^{-1}(x)\neq \emptyset$ for every $i$, so 
$p^{-1}(V)$ has a finite number of 
connected components: $p^{-1}(V)=\bigsqcup_{i=1}^{\ell}W_{i}$.
\par
Let $w\in p^{-1}(x)$ and let $W$ be the connected component  of  $p^{-1}(V)$ 
which contains $w$. The set $p^{-1}_H(B\cap Y\times H)$ is  a  proper,  closed 
algebraic subset of the algebraic variety $\mathcal{C}(y_0)_H=M$.  Hence  it 
has no interior points in the Euclidean topology of $M$ (cf.
 \cite[Ch.~VII \S~2 Lemma~1]{Sha2}), i.e. it is thin in $M^{an}$.  The 
open set $W'=W\setminus p^{-1}_H(B\cap Y\times H)$ is connected 
\cite[Ch.~7 \S~4 n.2]{GR}. Let $V'=V\setminus B = V\setminus 
\{v|(t-t_{1})(v)=0\}$. Then $p^{-1}(V')\to V'$  is  a  topological  covering 
map, since it is a restriction of $\mathcal{C}(y_0)'\to 
Y\times H^G_n(Y,y_0)\setminus B$, furthermore $W'$ is a connected  component 
of $p^{-1}(V')$, so $p|_{W'}:W'\to V'$ is a topological covering map as well.
Let $b_{0}\in U\setminus b_{1}$ and let $v_{0}=(b_{0},(D,m))\in 
V'$. We claim that $\pi_1(V',v_{0})\cong \mathbb{Z}$. Following the notation 
of  \cite{FN}  let  $F_{0,2}(U)  =  \{(y_{1},y_{2})|y_{i}\in  U,   y_{1}\neq 
y_{2}\}$.  The  topological  space  $V'$  is  homotopy  equivalent   to 
$F_{0,2}(U)$.   The   projection   map   $F_{0,2}(U)\to    U$    given    by 
$(y_1,y_2)\mapsto y_1$  is  a  locally  trivial  fiber  bundle  with  fibers 
homeomorphic    to     $U\setminus     b_{1}$     (cf.     \cite{FN},     or 
\cite[Theorem~1.2]{Bir}).  By  \cite[Ch.~6  Sect.~6]{Hu}  one  has  an   exact 
sequence of homotopy groups
\[
    \pi_2(U,b_{1})      \to       \pi_1((U\setminus b_{1})\times 
\{b_{1}\},(b_{0},b_{1}))\to \pi_1(F_{0,2}(U),(b_{0},b_{1}))\to 
\pi_1(U,b_{1}) \to 1.
\]
Therefore     $\pi_1(V',v_{0})\cong     \pi_1(F_{0,2}(U),(b_{0},b_{1}))\cong 
\mathbb{Z}$. This implies that $p|_{W'}:W'\to V'$ is  a  topological  Galois 
covering map whose group of Deck transformations $\Deck(W'/V')$ is isomorphic to 
the cyclic group $C_{e}\subset \mathbb{C}$ of order $e$ for a  certain 
integer $e \geq 1$.
\par
Let $E\subset \mathbb{C} \times V$ be the analytic subset defined by 
$z^{e}=t-t_{1}$. This is a complex manifold of dimension $n+1$, the holomorphic map 
$\phi:E\to \mathbb{C}^{n+1}$ defined by $\phi(z,v)=(z,t_1(v),\ldots,t_n(v))$ is 
injective and the Jacobian map $T_{(z,v)}\phi$ is an isomorphism for every $(z,v)\in E$.
Therefore $\phi(E)$ is an open subset of $\mathbb{C}^{n+1}$ and $\phi:E\to \phi(E)$ is 
a biholomorphic map (cf. \cite[Prop.~2.4]{Fi}).
The projection map $p_1:E\to  V$ 
given by $(z,v)\mapsto v$ is a finite, surjective, holomorphic map 
\cite[Ch.~2 \S~3 n.5]{GR}. Let  $E'=p_1^{-1}(V')=E\setminus  \{z=0\}$. 
Then the restriction of $p_1$, $p_1':E'\to V'$ is locally biholomorphic.  It 
is  also  proper  \cite[Ch.~9  \S~2  n.4]{GR},  hence  $p_1':E'\to  V'$  is   a 
topological covering map. The cyclic group $C_{e} = \{\omega^{q}|q\in \mathbb{Z}\}$, 
$\omega   =   exp\left(\frac{2\pi   i}{e}\right)$,   acts   on   $E$    by 
$\omega^{q}(z,v)=(\omega^{q}z,v)$  and  this  action   is   by   holomorphic 
automorphisms of $E$ which preserve the fibers of $p_1:E\to V$.  Furthermore 
$E'$ is connected, $p_1':E'\to V'$ is a topological Galois covering map and 
$\Deck(E'/V')\cong C_{e}$. Indeed, $C_{e}$ acts transitively on 
the fibers of $p_1'^{-1}(v_{0})$, then,
in  order  to  prove  that  $E'$   is   connected,   let   us   verify   that 
$\pi_1(V',v_{0})$ acts transitively on the  right  on  $p_1'^{-1}(v_0)$.  It 
suffices to prove the analogous statement replacing $V$  by  $\Delta  \times 
\Delta$, where $\Delta = \{z\in \mathbb{C}|\, |z|<\epsilon\}$ and $V'$ by 
$F_{0,2}(\Delta) = \{(z_{1},z_{2})|z_{i}\in \Delta, z_{1}\neq z_{2}\}$.  Let 
$(z,z_{1},z_{2})\in    \mathbb{C}    \times     F_{0,2}(\Delta)$     satisfy 
$z^{e}=z_{1}-z_{2}$. Let $\omega^{q}\in C_{e}$. Then the path 
\begin{equation}\label{e3.29b}
\alpha(\tau)=(\omega(\tau)^{q}z,\omega(\tau)^{qe}z_{1},
\omega(\tau)^{qe}z_{2}),
\quad \omega(\tau)=
\exp\left(\frac{2\pi i}{e}\tau\right)
\end{equation}
is a lifting of a closed path in $F_{0,2}(\Delta)$ based at 
$(z_{1},z_{2})$ 
which connects $(z,z_{1},z_{2})$ with $(\omega^{q}z,z_{1},z_{2})$.
\par
We conclude that there is a covering homeomorphism $\varphi'$
\begin{equation}\label{e3.29a}
\xymatrix{
W'\ar[rr]^{\varphi'}\ar[rd]_{p'|_{W'}}&&
E'\ar[dl]^{p_1'}\\
&V'
}
\end{equation}
which is $C_{e}$-equivariant. It is biholomorphic since both $p'|_{W'}$  and 
$p_1'$       are       locally       biholomorphic.       According       to 
\cite[Proposition~5.3]{SGA1} there is a unique extension of  $\varphi'$ 
to a biholomorphic covering map $\varphi$
\begin{equation}\label{e3.29}
\xymatrix{
W\ar[rr]^{\varphi}\ar[rd]_{p|_{W}}&&
E\ar[dl]^{p_1}\\
&V
}
\end{equation}
The composition $\psi = \phi \circ \varphi: W\to \phi(E)$ is a biholomorphic map.
We see that $M^{an}$ is a nonsingular complex space, therefore $\mathcal{C}(y_0)_H$  
is  a  smooth  algebraic variety.
\par
The restriction of \eqref{e3.29} to $V\cap B$ yields three  bijective  maps, 
hence $p^{-1}(x)\cap W$ consists of one point. This shows that the number of 
connected components of $p^{-1}(V)$ equals $r=|p^{-1}(x)|$. We may  enumerate 
them    so     that     $w_{i}\in     W_{i},i=1,\ldots,r$.     Let     $i\in 
[1,r]$ and let $(W,w)=(W_{i},w_{i})$. Let $G(w)=G(w_{i})\subset  G$  be  the 
isotropy group of $w=w_{i}$.  The  group  $G$  acts  on  $M^{an}$  
by biholomorphic covering maps, hence it permutes  transitively  the  connected 
components of $p^{-1}(V)$. Therefore $W=W_{i}$ is invariant under the action 
of $G(w)$, $p|_{W'}:W'\to V'$ is a topological Galois covering map and $G(w)$ is 
isomorphic to $Deck(W'/V')$. The composition 
$G(w)\overset{\sim}{\lto}Deck(W'/V')\overset{\sim}{\lto}C_e$ is a primitive 
character $\chi:G(w)\to C_e$ for which the statements of Part~(iii) hold. Furtermore
one has $e = |G(w)| = \frac{|G|}{r}$ as claimed in  Part~(i). 
\par
Next we  prove  that  $G(w)$  is  generated  by  an  element  conjugated  with 
$m([\gamma_{1}])$. Let us consider the loop  $\beta_{1}:I\to  U_{1}\setminus 
b_{1}$ based at $b_{0}$ $(U=U_1)$, defined by $s_{1}(\beta_{1}(\tau))=
e^{2\pi i\tau}s_{1}(b_{0})$ (cf. \S~\ref{3.23b}). Let $\beta:I\to V'$ be the 
loop     $\beta(\tau)=(\beta_{1}(\tau),(D,m))$.     Let     $w_{0}\in      W', 
p'(w_{0})=v_{0}$. Lifting $\beta$ in $W'$ with  initial  point  $w_{0}$  the 
terminal point is $w_{0}\beta = gw_{0}$, where $g\in G(w)$ and 
$\varphi(gw_{0})=exp(\frac{2\pi i}{e})\varphi(w_{0})$ 
(apply \eqref{e3.29b} with $k=1, z_1=s_1(b_0), z_2=s_1(b_1)=0$). 
Therefore $g$ generates $G(w)$. The closed path $\gamma_{1}$ is homotopic
to 
$\eta_{1}\cdot \beta_{1}\cdot \eta_{1}^{-}$, where 
$\eta_{1}:I\to Y\setminus D$ is a path such that $\eta_{1}(0)=y_{0}$,
$\eta_{1}(1)=b_{0}$. Let $\eta:I\to Y\times H\setminus B$ be the path 
$\eta(\tau)=(\eta_{1}(\tau),(D,m))$. Let $z_{0}=\zeta(D,m)\in 
\mathcal{C}(y_0)'_{(D,m)}$. Let 
$\tilde{\eta}:I\to \mathcal{C}(y_0)'_{(D,m)}\subset \mathcal{C}(y_0)'_{H}$ be 
the lifting of $\eta$ with initial point $\tilde{\eta}(0)=z_{0}$. Then 
$\tilde{\eta}(1)\in p'^{-1}(v_{0})$, so $\tilde{\eta}(1) = hw_{0}$ for some 
$h\in G$. Let $g_{1}=m([\gamma_{1}])$. By Proposition~\ref{3.12}(vi) lifting 
$\gamma_{1}$ in $\mathcal{C}(y_0)'$ with initial point $z_{0}$ the terminal 
point is $g_{1}z_{0}$. The lifted path is equal to the product 
$\tilde{\eta}\cdot \tilde{\beta}\cdot (\eta^{-})^{\sim}$ of the liftings of 
$\eta, \beta$ and $\eta^{-}$ respectively. Now, since 
$\mathcal{C}(y_0)'_{H}\to Y\times H\setminus B$ is a topological Galois 
covering map and the terminal point of $(\eta^{-})^{\sim}$ is $g_{1}z_{0}$, the 
initial point of $(\eta^{-})^{\sim}$ is $g_{1}\tilde{\eta}(1)=g_{1}(hw_{0})$. 
The initial point of $(\eta^{-})^{\sim}$ is the terminal point of 
$\tilde{\eta}\cdot \tilde{\beta}$.
This terminal point equals
\(
z_0(\eta\cdot \beta) = (hw_{0})\beta = h(w_{0}\beta) = h(gw_{0}), 
\)
hence $g_{1}(hw_{0})=h(gw_{0})$, which implies $g=h^{-1}g_{1}h$.
Parts (i), (ii) and (iii) are proved.
\par
(iv) The map $G\times^{G(w)}W \to p^{-1}(V)$, given by $(g,z)\mapsto gz$, is 
biholomorphic since $p^{-1}(V)=\bigsqcup_{i=1}^{r}g_{i}W$, where 
$g_{j}g_{i}^{-1}\notin G(w)$ if $j\neq i$.
\end{proof}
\begin{block}\label{3.32}
Let $p:\mathcal{C}(y_0)\to Y\times H^G_n(Y,y_0)$ be the disjoint union of 
$p_{H}:\mathcal{C}(y_0)_H\to Y\times H$, where $H$ runs over all connected 
components of $H^G_n(Y,y_0)$. There is a $G$-equivariant open embedding 
$\mathcal{C}(y_0)'\hookrightarrow \mathcal{C}(y_0)$ whose image is 
$p^{-1}(Y\times H^G_n(Y,y_0)\setminus B)$ and which we identify with 
$\mathcal{C}(y_0)'$. Let $\zeta:H^G_n(Y,y_0)\to \mathcal{C}(y_0)'$ be the 
morphism of Proposition~\ref{3.19}(iv).
\end{block}
\begin{theorem}\label{3.33}
Let $Y$ be a smooth, projective, irreducible  curve of genus $g\geq 0$.  Let 
$n$ be a positive integer. Let $y_{0}\in Y$. Let $G$ be a finite group which 
can be generated by $2g+n-1$ elements. The pair 
\begin{equation}\label{e3.33}
(p:\mathcal{C}(y_0)\to    Y\times    H^G_n(Y,y_0),     \zeta:H^G_n(Y,y_0)\to 
\mathcal{C}(y_0))
\end{equation}
is a smooth family of pointed $G$-covers of  $(Y,y_0)$  branched  in  $n$ 
points. For every $(D,m)\in H^G_n(Y,y_0)$ the pointed $G$-cover 
$(\mathcal{C}(y_0)_{(D,m)}\to Y,\zeta(D,m))$ of  $(Y,y_0)$ has monodromy 
invariant $(D,m)$ 
(cf. \S~\ref{3.3a}). Every pointed $G$-cover $(C\to Y,z_0)$ of 
$(Y,y_0)$ branched in $n$ points  is  $G$-equivalent  to  a  unique  pointed 
$G$-cover of $(Y,y_0)$ of the family \eqref{e3.33}.
\end{theorem}
\begin{proof}
The composition $f:\mathcal{C}(y_0)\overset{p}{\lto}Y\times H^G_n(Y,y_0)
\overset{\pi_{2}}{\lto}H^G_n(Y,y_0)$ is  a  proper  morphism  since  $p$  is 
finite and $\pi_{2}$ is  proper.  The  map  $f$  is  a  morphism  of  smooth 
algebraic varieties. For every $z\in \mathcal{C}(y_0)$  the  induced  linear 
map on the tangent spaces  $T_{z}f$  is  surjective  with  one-dimensional 
kernel. This is clear if $z\in 
\mathcal{C}(y_0)'$ and follows from Proposition~\ref{3.24}(iii) if  $p(z)\in 
B$. Therefore $f$ is a smooth morphism \cite[Ch.~III Prop.~10.4]{Hart}.
For every $(D,m)\in H^G_n(Y,y_0)$ the fiber $\mathcal{C}(y_0)_{(D,m)}$
is a smooth, projective curve and the restriction  
$p_{(D,m)}:\mathcal{C}(y_0)_{(D,m)}\to Y\times \{(D,m)\}$ is a finite, surjective 
morphism. The Zariski open subset $\mathcal{C}(y_0)'_{(D,m)}$, the preimage of 
$(Y\setminus D)\times \{(D,m)\}$, is connected in the Euclidean topology, so it is 
irreducible. Hence $\mathcal{C}(y_0)_{(D,m)}$ is irreducible. Furthermore
$\mathcal{C}(y_0)_{(D,m)}/G \cong Y\times \{(D,m)\}$ by Proposition~\ref{3.12}(vi).
The $G$-cover $\mathcal{C}(y_0)_{(D,m)}\to Y$ is ramified at every  point  of 
$D$ by Proposition~\ref{3.24}(i) and  is  unramified  over  $Y\setminus  D$, 
where the restriction  is  $G$-equivalent  to  $\mathcal{C}(y_0)'_{(D,m)}\to 
Y\setminus D$. By Proposition~\ref{3.12}(vi) the monodromy  homomorphism  of 
the pointed $G$-cover $(p_{(D,m)}: \mathcal{C}(y_0)_{(D,m)}\to Y,\zeta(D,m))$ of 
$(Y,y_0)$ is \linebreak
$m:\pi_1(Y\setminus D,y_0)\tto G$. The last statement is clear 
(cf. \S~\ref{3.3}).
\end{proof}
\section{Lifting of morphisms}\label{Serre}
\begin{block}\label{Serre1}
We  consider  separated  schemes  of  finite  type  over  the   base   field 
$k=\mathbb{C}$. Given such a  scheme  one  denotes  by  $X^{an}$  the 
associated complex space \cite[\S~1]{SGA1}. We recall that a morphism 
$(p,p^{\sharp}):(X,\mathcal{O}_{X})\to (Y,\mathcal{O}_{Y})$  of  schemes  is 
called unramified at a point $x\in X(\mathbb{C})$ if 
$p^{\sharp}_{x}(\mathfrak{m}_{p(x)})\mathcal{O}_x=\mathfrak{m}_{x}$.  This   
condition   is 
equivalent to each of the following ones: a) $\Omega^{1}_{X/Y}(x)=0$; 
b) $\Omega^{1}_{X^{an}/Y^{an}}(x)=0$; 
c) $p^{an}:X^{an}\to Y^{an}$ is an immersion at $x$
(cf.\cite[Ch.~VI Prop.~3.3]{A-K}, 
\cite[Prop.~3.1(ii)]{SGA1}, \cite[Prop.~3.1]{Gr2}, \cite[Prop.~1.24]{PR}). 
If moreover $p$ is flat at $x$, then $p$ is \'{e}tale at 
$x\in  X(\mathbb{C})$.  A  morphism  $p:X\to  Y$  is  \'{e}tale   at   $x\in 
X(\mathbb{C})$ if and only if $p^{an}:X^{an}\to Y^{an}$ 
is locally biholomorphic at $x$ (cf. \cite[Prop.~3.1(iii)]{SGA1} and 
\cite[Th\'{e}or\`{e}me~3.1]{Gr1}).
\end{block}
\begin{proposition}\label{Serre2}
Let $X,Y$ and $Z$ be schemes of finite type over $\mathbb{C}$. Let 
$f:Z\to Y$  and  $p:X\to  Y$  be  morphisms.  Suppose  that  $p:X\to  Y$  is 
unramified (at $\forall x\in X(\mathbb{C})$). Suppose that  there  exists  a 
holomorphic map $h:Z^{an}\to X^{an}$ such that 
$f^{an}=p^{an}\circ h$
\[
\xymatrix{
              &\, \, \, X^{an}\ar[d]^{p^{an}}\\
Z^{an}\ar[ru]^h\ar[r]_{f^{an}}&\, \, \, Y^{an}
}
\]
Then there exists a unique morphism $g:Z\to X$ such that $f=p\circ g$ and 
$g^{an}=h$.
\end{proposition}
\begin{proof}
We may assume, without loss of generality,  that $Z$ is connected. The diagonal 
morphism $\Delta:X\to X\times_{Y}X$ is a closed and open morphism since  $X$ 
and $Y$ are separated schemes and $p:X\to Y$  is  unramified.  This  implies 
that 
\[
\Delta^{an}: X^{an} \to (X\times_{Y}X)^{an} \cong X^{an}\times_{Y^{an}}X^{an}
\]
is   a   closed   and   open   immersion   (cf.   \cite[Cor.~0.32]{Fi}   and 
\cite[\S~1.2]{SGA1}). One has the following Cartesian diagram 
(cf. \cite[Prop.~9.3]{GW})
\begin{equation*}
\xymatrix{
Z^{an}\ar[r]^-{id\times 
h}\ar[d]_-{h}&Z^{an}\times_{Y^{an}}X^{an}\ar[d]^-{h\times id}\\
X^{an}\ar[r]^-{\Delta^{an}}&X^{an}\times_{Y^{an}}X^{an}
}
\end{equation*}
It implies that $id\times h$ is a closed and open immersion. 
Let $Z\times_{Y}X      =      \bigsqcup_{i=1}^{r}W_{i}$
be the decomposition of the scheme $Z\times_{Y}X$ into connected components.
According to \cite[Cor.~2.6]{SGA1} $Z^{an}$ is connected and 
\[
Z^{an}\times_{Y^{an}}X^{an} \cong (Z\times_{Y}X)^{an} = 
\bigsqcup_{i=1}^{r}W_{i}^{an}
\]
is the decomposition into connected components. Therefore 
$\Gamma_{h}=(id\times h)(Z^{an})$, the graph of $h$, coincides with 
$W_{i}^{an}$ for some $i$. Let $G=W_{i}$. Then the composition 
$G\hookrightarrow Z\times_{Y}X\overset{\pi_{1}}{\lto}Z$ is an isomorphism 
since $G^{an}=\Gamma_{h}
\hookrightarrow (Z\times_{Y}X)^{an}\to Z^{an}$ is a biholomorphic map
\cite[Prop.~3.1]{SGA1}.      The       composition       $g=\pi_{2}|_{G}\circ 
(\pi_{1}|_{G})^{-1}:Z\to G\to X$ is a morphism which satisfies $f=p\circ  g$ 
and $h=g^{an}$. The uniqueness of $g$ follows from  the  uniqueness  of  the 
closed subscheme $G\subset Z\times_{Y}X$ whose ideal sheaf satisfies 
$(\mathcal{J}_{G})^{an} = \mathcal{J}_{\Gamma_{h}}$ (cf. 
\cite[Prop.~1.3.1]{SGA1}).
\end{proof}
\begin{lemma}\label{Serre4}
Let $X,Y$ and $Z$ be complex spaces. Let 
 $(f,\tilde{f}):(Z,\mathcal{O}_{Z})\to (Y,\mathcal{O}_{Y})$ and
$(p,\tilde{p}):(X,\mathcal{O}_{X})\to  (Y,\mathcal{O}_{Y})$  be  holomorphic 
maps. Suppose that one of the following conditions holds:
\begin{enumerate}
\item
$(p,\tilde{p})$ is locally biholomorphic;
\item
$(p,\tilde{p})$ is an immersion at every $x\in X$ and  $(Z,\mathcal{O}_{Z})$ 
is reduced.
\end{enumerate}
Suppose that there is a continuous lifting $h$ of $f$: $f=p\circ h$
\[
\xymatrix{
              &\,   |X|\ar[d]^{p}\\
|Z|\ar[ru]^h\ar[r]_{f}&\,  |Y|
}
\]
Then $h$ is the underlying continuous map of a holomorphic map
$(h,\tilde{h}):(Z,\mathcal{O}_{Z})\to (X,\mathcal{O}_{X})$, which is 
a holomorphic lifting of $(f,\tilde{f})$:
$(f,\tilde{f})=(p,\tilde{p})\circ (h,\tilde{h})$.
\end{lemma}
\begin{proof}
Case (i). In this case the stalk map $\tilde{p}_{x}:\mathcal{O}_{Y,p(x)}\to \mathcal{O}_{X,x}$
is an isomorphism for every $x\in X$.
Let $z\in Z$, $x=h(z)$, $y=p(x)$. One defines a local homomorphism 
$\tilde{h}_{z}:\mathcal{O}_{X,x}\to \mathcal{O}_{Z,z}$ by  
$\tilde{h}_{z}(s_{x})=\tilde{f}_{z}(\tilde{p}_{x}^{-1})(s_{x})$.
Let $V$ be an open subset of $X$, let $U=h^{-1}(V)$. Let 
$s\in \Gamma(V,\mathcal{O}_{X})$. For every $z\in U$ let $t_{z}=
\tilde{h}_{z}(s_{h(z)})$. Let $t:U\to \bigsqcup_{z\in U}\mathcal{O}_{Z,z}$
be  the  map  defined   by   $z\mapsto   t_{z}$.   We   claim   that   $t\in 
\Gamma(U,\mathcal{O}_{Z})$.  For  every  $z\in  U$  let   $V_{x}$   be an open 
neighborhood   of   $x=h(z)$,   $V_{x}\subset   V$,   
which   is mapped   by   $p$ 
biholomorphically onto $W_{y}\subset Y$. Let $U_{z}=h^{-1}(V_{x})$. Let 
$r\in \Gamma(W_{y},\mathcal{O}_{Y})$ satisfy 
$\tilde{(p|_{V_x})}(r)=s|_{V_{x}}$.
Then 
$t|_{U_{z}}=\tilde{f}(r)$ belongs to  $\Gamma(U_{z},\mathcal{O}_{Z})$.  This 
shows that $t\in \Gamma(U,\mathcal{O}_{Z})$.  One  defines  in  this  way  a 
morphism of sheaves  $\tilde{h}:\mathcal{O}_{X}\to  h_{\ast}\mathcal{O}_{Z}$ 
and $(h,\tilde{h}):(Z,\mathcal{O}_{Z})\to (X,\mathcal{O}_{X})$ 
is a holomorphic lifting of 
$(f,\tilde{f}):(Z,\mathcal{O}_{Z})\to (Y,\mathcal{O}_{Y})$.
\par
Case (ii). Here one defines 
$\tilde{h}_{z}:\mathcal{O}_{X,x}\to \mathcal{O}_{Z,z}$ as follows. 
By  hypothesis  \linebreak
$\tilde{p}_{x}:\mathcal{O}_{Y,y}\to  \mathcal{O}_{X,x}$   is 
surjective. Let $s_{x}=\tilde{p}_{x}(r_{y})$. Let $t_{z}=
\tilde{f}_{z}(r_{y})$. We claim that $t_{z}$ does not depend on  the  choice 
of $r_{y}$. In fact, let $s_{x}=\tilde{p}_{x}(r_{y})=
\tilde{p}_{x}(r'_{y})$. There are neighborhoods $V_{x}$ of $x$  and  $W_{y}$ 
of $y$ such that $p|_{V_{x}}:V_{x}\to W_{y}$ is a closed embedding, 
$r_{y}$ and $r'_{y}$ are germs of
$r,r'\in \Gamma(W_{y},\mathcal{O}_{Y})$ and $\overset{\sim}{p|_{V_{x}}}(r-r')=0$.  In 
particular    $(r-r')(w)=0$    for    $\forall    w\in    p(V_{x})$.     Let 
$U_{z}=h^{-1}(V_{x})$. Then $f(U_{z})\subset p(V_{x})$, therefore 
$\tilde{f}(r-r')(u)=(r-r')(f(u))=0$ for $\forall u\in U_{z}$. 
This implies that $\overset{\sim}{f|_{U_z}}(r-r')=0$ by \cite[Ch.~4 \S~3 n.3]{GR} since 
$(Z,\mathcal{O}_{Z})$ is reduced. This shows that 
$\tilde{f}_{z}(r_{y}) = \tilde{f}_{z}(r'_{y})$  as  claimed,  so  the  local 
homomorphism   $\tilde{h}_{z}:\mathcal{O}_{X,x}\to   \mathcal{O}_{Z,z}$   is 
well-defined. One concludes the proof as in Case~(i).
\end{proof}
Combining Proposition~\ref{Serre2} and Lemma~\ref{Serre4}  one  obtains  the 
following proposition
\begin{proposition}\label{Serre7}
Let $X,Y$ and $Z$ be separated schemes of finite type over $\mathbb{C}$. Let 
$f:Z\to Y$ and $p:X\to  Y$  be  morphisms.  Suppose  one  of  the  following 
conditions holds:
\begin{enumerate}
\item
$p$ is \'{e}tale;
\item
$p$ is unramified and $Z$ is reduced.
\end{enumerate}
Suppose there is a continuous lifting $c$  of  $f^{an}$:  $f^{an}=p^{an}\circ 
c$
\[
\xymatrix{
              &\, \, \, |X^{an}|\ar[d]^{p^{an}}\\
|Z^{an}|\ar[ru]^c\ar[r]_{f^{an}}&\, \, \, |Y^{an}|
}
\]
Then   there   is    a    morphism    $(g,g^{\sharp}):(Z,\mathcal{O}_{Z})\to 
(X,\mathcal{O}_{X})$ such that $f=p\circ g$ and such that $g(z)=c(z)$ for 
$\forall z\in Z(\mathbb{C})$. 
\end{proposition}
Recall that if $p:X\to Y$ is an \'{e}tale cover of algebraic varieties, 
then \linebreak
$p^{an}:|X^{an}|\to |Y^{an}|$ is a topological covering map
(see e.g. the proof of Proposition~\ref{2.7}(iv)).
\begin{corollary}\label{Serre8}
Let $X,Y$ and $Z$ be algebraic varieties. Let 
\[
\xymatrix{
              &   X\ar[d]^{p}\\
Z\ar[ru]^g\ar[r]_{f}&  Y
}
\]
be a commutative diagram of maps,  where  $f$  is  a  morphism,  $p$  is  an 
unramified morphism and $g:|Z^{an}|\to |X^{an}|$ is continuous. Then $g$ is a 
morphism. In particular, if $p:X\to Y$ is an \'{e}tale cover, then 
$\Deck(|X^{an}|/|Y^{an}|) \cong \Aut (X/Y)$.
\end{corollary}
\begin{corollary}\label{Serre9}
Let $p:X\to Y$ be a map of sets. Suppose that $Y$  has  a  structure  of  an 
algebraic variety  $(Y,\mathcal{O}_{Y})$  and  $X$  has  two  structures  of 
algebraic varieties $(X,\mathcal{O}_{X})$  and  $(X,\mathcal{O}'_{X})$  such 
that $p:X\to Y$ is unramified morphism for both structures of  $X$.  Suppose 
that the Euclidean topologies on $X$ associated with 
$(X,\mathcal{O}_{X})$  and   $(X,\mathcal{O}'_{X})$   coincide.   Then   the 
two      Zariski's      topologies      on      $X$       coincide       and 
$\mathcal{O}_{X}=\mathcal{O'}_{X}$.
\end{corollary}
\begin{proof}
 The map $id_{X}:X\to X$ is a continuous lifting of  the 
two unramified morphisms, so by Proposition~\ref{Serre7}(ii) it is an 
isomorphism 
of the algebraic varieties $(X,\mathcal{O}_{X})$ and $(X,\mathcal{O'}_{X})$.
\end{proof}
\section{Hurwitz      moduli      varieties      parameterizing       pointed 
$G$-covers}\label{s5}
Let  $\Var_{\mathbb{C}}$  be  the  category  of  algebraic   varieties   over 
$\mathbb{C}$.
\begin{definition}\label{5.1}
For every $S\in \Var_{\mathbb{C}}$ we  denote  by  $\mathcal{H}^{G}_{Y,n}(S)$ 
the set of all smooth families of  $G$-covers  of  $Y$  branched  in  $n$ 
points $p:\mathcal{C}\to Y\times S$ modulo $G$-equivalence 
(cf. Definition~\ref{3.7}). We denote by $\mathcal{H}^G_{(Y,y_0),n}(S)$  the 
set of all smooth families of pointed $G$-covers of $(Y,y_0)$ branched in 
$n$ points $(p:\mathcal{C}\to Y\times S,\zeta)$
modulo 
$G$-equivalence (cf. Definition~\ref{3.7a}).
\end{definition}
\begin{block}\label{5.2}
Let $u:T\to S$ be a morphism of algebraic varieties. Given a smooth,  proper 
morphism $\mathcal{C}\to S$ of reduced, separated  schemes  of  finite  type 
over $\mathbb{C}$, the pullback morphism $\mathcal{C}_{T}:= 
\mathcal{C}\times_{S}T\to T$ is smooth and proper, in particular 
$\mathcal{C}\times_{S}T$ is a reduced scheme, since $T$ is reduced 
(cf.  \cite[p.184]{Ma2}).  Hence  the  scheme  $\mathcal{C}\times_{S}T$   is 
isomorphic to the  closed  algebraic  subvariety  of  $\mathcal{C}\times  T$ 
whose set of points is the set-theoretical fiber product 
$\mathcal{C}(\mathbb{C})\times_{S(\mathbb{C})}T(\mathbb{C})$.
\par
Given a family of $G$-covers $p:\mathcal{C}\to Y\times S$ as in 
Definition~\ref{3.7}  let\linebreak
$p_{T}:\mathcal{C}_{T}\to  Y\times  T$   be   the 
morphism obtained from $\mathcal{C}\times_{S}T\to T$ and 
$\mathcal{C}\times_{S}T\to \mathcal{C}\to Y$. This is the  pullback  family 
of covers. One defines an action of $G$ on $\mathcal{C}_{T}$  induced  by 
the action of $G$ on the first factor of $\mathcal{C}\times_{S}T$. Then 
$p_{T}$ is $G$-invariant and for every $t\in T$ the $G$-cover 
$(\mathcal{C}_{T})_{t}\to Y$ is $G$-equivalent to $\mathcal{C}_{u(t)}\to Y$. 
Therefore $p_{T}:\mathcal{C}_{T}\to Y\times T$ satisfies the conditions of 
Definition~\ref{3.7}. Clearly the pullbacks of $G$-equivalent families of 
$G$-covers are $G$-equivalent. This defines a moduli functor
\(
\mathcal{H}^{G}_{Y,n}: \Var_{\mathbb{C}} \to (\Sets)
\)
\par
Given a family $(p:\mathcal{C}\to Y\times S,\zeta)$
of pointed $G$-covers of $(Y,y_0)$ branched in $n$ points as in 
Definition~\ref{3.7a} 
and a morphism $u:T\to S$ the 
$G$-covers of the pullback family $p_{T}:\mathcal{C}_{T}\to  Y\times  T$ 
are unramified over $y_{0}$. Let  $\zeta_{T}:T\to  \mathcal{C}_{T}$  be  the 
morphism $\zeta_{T}(t)=(\zeta(u(t)),t)$. One has 
$p_{T}(\zeta_{T}(t))=(y_{0},t)$. Therefore 
$(p_{T}:\mathcal{C}_{T}\to Y\times T,\zeta_{T})$ satisfies the conditions of 
Definition~\ref{3.7a}. This defines a moduli functor
\(
\mathcal{H}^{G}_{(Y,y_0),n}: \Var_{\mathbb{C}} \to (\Sets)
\)
\end{block}
\begin{proposition}\label{5.4}
Let $(p:\mathcal{C}\to Y\times S,\zeta:S\to \mathcal{C})$ be a smooth family 
of pointed $G$-covers of $(Y,y_0)$ branched in $n$ points. Let 
$B\subset Y\times S$ be the branch locus of $p$. For every $s\in S$ let 
\begin{equation}\label{e5.4}
u(s) = (B_{s},m_{\zeta(s)})\in H^G_n(Y,y_0),\quad  
m_{\zeta(s)}:\pi_1(Y\setminus B_{s},y_0)\tto G
\end{equation}
be the monodromy  invariant  of  $(p_{s}:\mathcal{C}_{s}\to  Y,\zeta(s))$. 
Then $u:S\to H^G_n(Y,y_0)$ is a morphism.
\end{proposition}
\begin{proof}
The map $\beta:S\to (Y\setminus y_0)^{(n)}_{\ast}\subset Y^{(n)}$ given by 
$\beta(s)=B_{s}$ is a morphism by Proposition~\ref{2.7}(vi). The map $u$ 
fits in the following commutative diagram
\[
\xymatrix{
              &   H^G_n(Y,y_0)\ar[d]^-{\delta}\\
S\ar[ru]^-u\ar[r]_-{\beta}&  (Y\setminus y_0)^{(n)}_{\ast}
}
\]
Here $\delta$, defined by $\delta(D,m)=D$, is a finite, surjective, \'{e}tale 
morphism (cf. \cite[Prop.~1.8]{K1} and \S~\ref{3.5b}).
By Corollary~\ref{Serre8} it suffices to prove that $u:S\to H^G_n(Y,y_0)$
is a continuous map with respect to the Euclidean topologies of $S^{an}$  and 
that of $H^G_n(Y,y_0)$ defined in \S~\ref{3.5b}, i.e. we have to show that 
for every $s\in S$ and every neighborhood $W$ of $u(s)$ the point $s$
is internal of $u^{-1}(W)$.
Let $s_{0}$ be an arbitrary point of $S$. Let $u(s_0) = (D,m)$,
where $D= \{b_1,\ldots,b_n\}$, $z_{0}=\zeta(s_{0})$, 
$m=m_{z_{0}}:\pi_1(Y\setminus D,y_0)\tto G$.
 Let $N_{(D,m)}(U_1,\ldots,U_n)$ be 
any of the open sets of the neighborhood basis of $(D,m)$ in  $H^G_n(Y,y_0)$ 
(cf. \eqref{e3.6a1}). One has to prove that there exists a neighborhood 
$V\subset |S^{an}|$ of $s_{0}$ such that 
$u(V)\subset N_{(D,m)}(U_1,\ldots,U_n)$. There is a neighborhood 
$V_{1}\subset |S^{an}|$ of $s_{0}$ such that 
$\beta(V_{1})\subset  N_D(U_1,\ldots,U_n)$ since $\beta^{an}$ is a holomorphic 
map.  The  complex  space  $S^{an}$  is 
locally connected (cf. \cite[Ch.~9 \S~3 n.1]{GR}) and 
$\mathcal{C}^{an}\setminus p^{-1}(B)\to (Y\times S)^{an}\setminus  B$  is  a 
topological  covering map by  Proposition~\ref{2.7}(iv).  Therefore  there  is   
an embedded open disk $U\subset Y\setminus \cup_{i=1}^{n} \overline{U}_{i}$, 
$y_{0}\in U$  and  a  connected  neighborhood  $V$  of  $s_{0}$,  such  that 
$V\subset V_{1}$, $U\times V\subset Y\times S\setminus B$ and 
$p^{-1}(U\times V)$ is a disjoint union of connected open sets  homeomorphic 
to $U\times V$. Let $W$ be the connected component  of  $p^{-1}(U\times  V)$ 
which contains $\zeta(s_{0})$. One has 
$\zeta(V) \subset p^{-1}(\{y_{0}\}\times V)\subset p^{-1}(U\times V)$, so 
$\zeta(V)\subset W$ since $V$ is connected. We claim that 
$u(V)\subset N_{(D,m)}(U_1,\ldots,U_n)$. For every $s\in V$ one has 
$u(s)=(\beta(s),m_{\zeta(s)})$, so one has to prove that 
$m_{\zeta(s)}=m(\beta(s))$ for $\forall s\in V$. It suffices to  verify  the 
equality 
$m_{\zeta(s)}([\gamma]_{\beta(s)})=m(\beta(s))([\gamma]_{\beta(s)})$     for 
every closed path $\gamma:I\to Y\setminus \cup_{i=1}^{n}\overline{U}_{i}$, 
$\gamma(0)=\gamma(1)=y_{0}$. Let $g=m([\gamma]_{D})$. Then 
$m(\beta(s))([\gamma]_{\beta(s)}) = g$ according to \eqref{e3.6a}. Hence  we 
have to verity that for every $s\in V$ the lifting of $\gamma\times \{s\}$ in 
$\mathcal{C}_{s}\setminus p_{s}^{-1}(B_{s})$ with initial  point  $\zeta(s)$ 
has terminal point $g\zeta(s)$. This follows from  the  covering  homotopy 
property. Indeed, let $F$ be the continuous map
\[
F:[0,1]\times V \to Y\times S\setminus B, \quad F(t,s)=(\gamma(t),s).
\]
The map $\tilde{F}_{0}:\{0\}\times V\to \mathcal{C}\setminus p^{-1}(B)$, 
$\tilde{F}_{0}(0,s)=\zeta(s)$ is a  continuous  lifting  of 
$F|_{\{0\}\times V}$. Let $\tilde{F}$
\[
\xymatrix{
              &\mathcal{C}\setminus p^{-1}(B)\ar[d]^-p\\
[0,1]\times V\ar[ru]^-{\tilde{F}}\ar[r]^-{F}&  Y\times S\setminus B
}
\]
be the unique continuous lifting of $F$ which extends $\tilde{F}_{0}$ 
(cf. \cite[Ch.~2 \S~2 Th.~3]{Spa}). For every $s\in V$ the path 
$t\mapsto \tilde{F}(t,s)$ is the unique lifting 
in $\mathcal{C}_{s}\setminus p_{s}^{-1}(B_{s})$
of $\gamma\times \{s\}$
 with initial point $\zeta(s)$, so
\[
\tilde{F}(1,s) = m_{\zeta(s)}([\gamma]_{\beta(s)})\zeta(s).
\]
One  has  that 
$\tilde{F}(\{1\}\times   V)\subset   p^{-1}(\{y_{0}\}\times V)
\subset p^{-1}(U\times V)$ and $\tilde{F}(1,s_{0})=g\zeta(s_{0})\in gW$, 
therefore 
 $\tilde{F}(\{1\}\times V)\subset gW$, since $V$ is  connected.  This  shows 
that $\tilde{F}(1,s)=g\zeta(s)$, hence $m_{\zeta(s)}([\gamma]_{\beta(s)})=g$ 
for $\forall s\in V$.
\end{proof}
We need the following result \cite[Theorem~2]{K2} in the proof of 
Theorem~\ref{5.8} below. Here we state it in the form we will use it.
\begin{proposition}\label{extending}
Let $X,S$ and $P$ be  algebraic  varieties.  Let  $h:X\to  S$  be  a  smooth 
morphism whose nonempty fibers are irreducible curves. Let  $P\to  S$  be  a 
proper morphism. Let $U$ be an open subset of $X$ such that 
$U\cap h^{-1}(s)\neq \emptyset$ for every $s\in h(X)$. Let $\varphi:U\to  P$ 
be an $S$-morphism. Let $\Gamma = \{(x,\varphi(x))|x\in U\}\subset U\times P$ 
be its graph and let 
$\overline{\Gamma}\subset X\times P$ be its  closure.  Suppose  that  the 
projection morphism $\overline{\Gamma}\to X$ has finite fibers.  Then  there 
is a unique extension of $\varphi$ to $X$: an $S$-morphism 
$\tilde{\varphi}:X\to P$ such that $\tilde{\varphi}|_{U}=\varphi$.
\end{proposition}
\begin{theorem}\label{5.8}
Let $Y$ be a smooth, projective, irreducible  curve of genus $g\geq 0$.  Let 
$n$ be a positive integer. Let $y_{0}\in Y$. Let $G$ be a finite group which 
can be generated by $2g+n-1$ elements.  The algebraic variety 
$H^G_n(Y,y_0)$ is a fine moduli variety for the moduli functor 
$\mathcal{H}^{G}_{(Y,y_0),n}$ of smooth families of 
pointed $G$-covers of $(Y,y_0)$ branched in  $n$  points.  The  universal 
family is 
\begin{equation}\label{e5.8}
(p:\mathcal{C}(y_0)\to    Y\times    H^G_n(Y,y_0),     \zeta:H^G_n(Y,y_0)\to 
\mathcal{C}(y_0))                           
\end{equation}
(cf. Theorem~\ref{3.33})
\end{theorem}
\begin{proof}
Let $[q:X\to Y\times S,\eta:S\to X]\in \mathcal{H}^G_{(Y,y_0),n}(S)$. 
Let $B\subset Y\times S$ be the branch locus and let 
$u:S\to H^G_n(Y,y_0)$, $u(s)=(B_s,m_{\eta(s)})$ be the morphism of 
Proposition~\ref{5.4}. We want to prove that $(X\to Y\times S,\eta)$ is 
$G$-equivalent to the pullback by $u$ of the family \eqref{e5.8}.  This  is 
the unique morphism  with  this  property since  the  monodromy 
invariant classifies the pointed $G$-covers up to $G$-equivalence.  
 For every $s\in S$ there exists a unique $G$-equivariant 
isomorphism $f_{s}:X_{s}\to \mathcal{C}(y_0)_{u(s)}$ such that 
$p_{s}\circ f_{s} = q_{s}$ and $f_{s}(\eta(s))=\zeta(u(s))$. Let 
$f:X\to \mathcal{C}(y_0)$ be the $G$-equivariant map, which  equals  $f_{s}$ 
on every $X_{s}$. One obtains the following commutative diagram
\begin{equation}\label{e5.9}
\xymatrix{
X\ar[r]^-{f}\ar[d]_-{q}&\mathcal{C}(y_0)\ar[d]^-{p}\\
Y\times S\ar[r]^-{id\times u}&Y\times H^G_n(Y,y_0)
}
\end{equation}
We want to prove that $f$ is a morphism and that \eqref{e5.9} is a Cartesian diagram. 
Let $\mathcal{B}$ be the branch locus of $p$. 
Then $B=(id\times u)^{-1}(\mathcal{B})$ is the branch locus of $q$. Let 
$X'=X\setminus q^{-1}(B)$, $f'=f|_{X'}$, $q'=q|_{X'}$. Restricting 
\eqref{e5.9}  on  the  complements  of  the  branch  loci  one  obtains  the 
commutative diagram
\begin{equation*}
\xymatrix{
X'\ar[r]^-{f'}\ar[d]_-{q'}&\mathcal{C}(y_0)'\ar[d]^-{p'}\\
Y\times S\setminus B\ar[r]^-{(id\times u)'}&
Y\times H^G_n(Y,y_0)\setminus \mathcal{B}
}
\end{equation*}
We claim that $f'$ is continuous with respect to the Euclidean topologies of  
$X'^{an}$ 
and $\mathcal{C}(y_0)'$ (cf. Proposition~\ref{3.12}). Let $x\in X'_{s}$  and 
let $\lambda:I\to X'_{s}$ be a path such that $\lambda(0)=\eta(s)$, 
$\lambda(1)=x$. Then $q'\circ \lambda:I\to Y\setminus B_s$ is a  path  with 
initial point $y_{0}$. Lifting $q'\circ \lambda$ in 
$\mathcal{C}(y_0)'_{u(s)}$ with initial point 
$\zeta(u(s)) = (\Gamma_{m_{\eta(s)}}[c_{y_0}]_{B_s},B_s,m_{\eta(s)})$ its
terminal point is
\begin{equation}\label{e5.10}
f'(x) = f'_s(x) = 
(\Gamma_{m_{\eta(s)}}[q'\circ \lambda]_{B_s},B_s,m_{\eta(s)}) 
\in \mathcal{C}(y_0)'_{u(s)}
\end{equation}
(cf. \S~\ref{3.3}).
For every $x_{0}\in  X'$ and every neighborhood $N$ of $f'(x_0)$ in 
$\mathcal{C}(y_0)'$ we have
to prove that $x_0$ is an internal point of $f'^{-1}(N)$.
 Let $q'(x_{0})=(y,s_{0})$, $D=\beta(s_{0})$,\linebreak
$m=m_{\eta(s_{0})}:\pi_1(Y\setminus D,y_0)\tto G$, let 
$\lambda:I\to X'_{s_{0}}$ be a path such that $\lambda(0)=\eta(s_{0})$,
$\lambda(1)=x_{0}$ and let 
$\alpha=q'_{s_{0}}\circ  \lambda:I\to  Y\setminus D$. One has 
$\alpha(0)=y_{0}$, $\alpha(1)=y$ and 
$f'(x_{0})=(\Gamma_m[\alpha]_D,D,m)\in \mathcal{C}(y_0)'$. 
Let $N_{(\alpha,D,m)}(U,U_1,\ldots,U_n)$ be a  neighborhood  of  $f'(x_{0})$ 
as in \S~\ref{3.11} contained in $N$. We want to show, that there  exists  a 
neighborhood $W$ of $x_{0}$ such that 
$f'(W)\subset N_{(\alpha,D,m)}(U,U_1,\ldots,U_n)$. 
The argument is similar to the one of Proposition~\ref{5.4}.  Shrinking  $U$ 
one can choose a connected neighborhood $V$ of $s_{0}$ in $|S^{an}|$ such that 
$\beta(V)\subset N_D(U_1,\ldots,U_n)$, $U\times V\subset Y\times S\setminus B$ 
and $q'^{-1}(U\times  V)$  is  a  disjoint  union  of  connected  open  sets 
homeomorphic to $U\times V$. Let $W$ be the connected component of 
$q'^{-1}(U\times V)$ which contains $x_{0}$. We claim that 
$f'(W)\subset N_{(\alpha,D,m)}(U,U_1,\ldots,U_n)$. Consider the homotopy
\[
F : [0,1]\times V \to Y\times S\setminus B, \quad F(t,s)=(\alpha(t),s)
\]
By the covering homotopy property of topological covering maps 
(cf. \cite[Ch.~2 \S~2 Th.~3]{Spa})
there exists a unique continuous lifting
\[
\xymatrix{
              &X'\ar[d]^-{q'}\\
[0,1]\times V\ar[ru]^-{\tilde{F}}\ar[r]^-{F}&  Y\times S\setminus B
}
\]
such that $\tilde{F}(0,s)=\eta(s)$ for $\forall s\in V$. We have 
$q'\circ \tilde{F}(0,s) = q'(\eta(s)) = (y_{0},s) = (\alpha(0),s)$. 
Let $s\in V$. The path $t\mapsto \tilde{F}(t,s)$ is a lifting of 
$\alpha\times \{s\}$ with initial point $\eta(s)$, so 
$q'(\tilde{F}(1,s)) = (\alpha(1),s) = (y,s) \in U\times V$. If 
$s=s_{0}$, then $\tilde{F}(1,s_{0}) = \lambda(1) = x_{0}$. This implies 
that $\tilde{F}(\{1\}\times V)\subset W$ since $V$ is connected. 
The map $q'|_{W}:W\to U\times V$ is a homeomorphism. Let $x\in W$ and let 
$q'(x)=(z,s)$. We construct a path in $X'_{s}$ which connects $\eta(s)$ with 
$x$ as follows. The path $\tilde{\alpha}_{s}(t)=\tilde{F}(t,s)$ has  initial 
point $\eta(s)$ and terminal point $w\in W$ such that 
$q'(w) = (\alpha(1),s) = (y,s)$. Let $\tau:I\to U$ be a path such that 
$\tau(0)=y$, $\tau(1)=z$. Then 
$\mu = \tilde{\alpha}_{s}\cdot \left((q'|_{W})^{-1}\circ (\tau\times \{s\})\right)$ is  a 
path in $X'_{s}$ which connects $\eta(s)$ with $x$ and 
$q'_{s}\circ \mu = \alpha\cdot \tau$. Let $\beta(s)=E\in N_D(U_1,\ldots,U_n)$. 
We showed in Proposition~\ref{5.4} that $m_{\eta(s)}=m(E)$. So, according to 
\eqref{e5.10}
\[
f'(x) = (\Gamma_{m(E)}[\alpha\cdot \tau]_E,E,m(E))
\]
This shows that $f'(W)\subset N_{(\alpha,D,m)}(U,U_1,\ldots,U_n)$. The  claim 
that $f'$ is continuous is proved.
\par
We apply now Corollary~\ref{Serre8} to the commutative diagram
\[
\xymatrix{
              &\mathcal{C}(y_0)'\ar[d]^-{p'}\\
X'\ar[ru]^-{f'}\ar[r]_-{(id\times u)'\circ q'}&  Y\times 
H^G_n(Y,y_0)\setminus \mathcal{B}
}
\]
and conclude that $f':X'\to \mathcal{C}(y_0)'$ is a morphism.
\par
Let $\mathcal{C}(y_0)_{S}$ be the fiber product of 
$\mathcal{C}(y_0)\to   H^G_n(Y,y_0)$   and   $u:S\to   H^G_n(Y,y_0)$. The 
composition 
$X'\overset{f'}{\lto}   \mathcal{C}(y_0)'\hookrightarrow   \mathcal{C}(y_0)$ 
yields an $S$-morphism $\varphi:X'\to \mathcal{C}(y_0)_{S}$  which  fits  in 
the following commutative diagram of morphisms
\[
\xymatrix{
X\ar[dr]_-{q} 
&X'\ar@{_{(}->}[l]\ar[r]^-{\varphi}\ar[d]&\mathcal{C}(y_0)_{S}\ar[dl]^-{p_S}\\
&Y\times S\ar[d]^-{\pi_2}\\
&S
}
\]
The  graph   $\Gamma$   of   $\varphi$   is   contained   in   the 
set-theoretical  fiber  product  $X\times_{Y\times   S}\mathcal{C}(y_0)_{S}$ 
which is a Zariski closed subset of $X\times  \mathcal{C}(y_0)_{S}$,  so  it 
contains the closure $\overline{\Gamma}$. Therefore the projection morphism 
$\overline{\Gamma}\to X$ has finite fibers. Applying 
Proposition~\ref{extending} we conclude that $\varphi$ can be extended to an 
$S$-morphism $\tilde{\varphi}:X\to \mathcal{C}(y_0)_{S}$. For every \linebreak
$s\in S$ 
the composition 
$X_{s}\overset{\tilde{\varphi}_{s}}{\lto} (\mathcal{C}(y_0)_{S})_{s}
\overset{\sim}{\lto}\mathcal{C}(y_0)_{u(s)}$ is a morphism which coincides 
with $f'_{s}$ on $X'_{s}$.  Hence  this  composition  equals  $f_{s}$.  This 
implies on one hand that $f$ is equal to the composition 
$X\overset{\tilde{\varphi}}{\lto}\mathcal{C}(y_0)_{S}\lto 
\mathcal{C}(y_0)$, so $f$  is  a  morphism,  and  on  the  other  hand  that 
$\tilde{\varphi}_{s}$  is a $G$-equivariant isomorphism for every $s\in  S$. 
Applying \cite[Proposition~(4.6.7)]{EGAIII} we conclude that 
$\tilde{\varphi}:X\to    \mathcal{C}(y_0)_{S}$    is    a    $G$-equivariant 
isomorphism. It is clear that $p_{S}\circ \tilde{\varphi} = q$ and 
$\tilde{\varphi}(\eta(s))=\varphi(\eta(s))=\zeta_{S}(s)$ for every $s\in S$. 
Therefore \eqref{e5.9} is a Cartesian diagram, so $(q:X\to Y\times S,\eta)$ 
is $G$-equivalent to the pullback of 
the family \eqref{e5.8} by the morphism $u:S\to H^G_n(Y,y_0)$.
\end{proof}
\begin{definition}\label{5.15}
Let $O_{1},\ldots,O_{k}$ be conjugacy classes of $G$, $O_{i}\neq  O_{j}$  if 
$i\ne j$. Let $\underline{n}=n_1O_1+\cdots+n_k O_k$ be a formal sum, where 
$n_{i}\in \mathbb{N}$. Let  $|\underline{n}|=n_{1}+\cdots+n_{k}=n$.  We  say 
that a pointed $G$-cover $(p:C\to Y,z_0)$ of $(Y,y_0)$ 
branched in $n$  points  is  of 
branching type $\underline{n}$ if, for every $i=1,\ldots,k$, $n_i$ of the 
branch points of 
$p$ have local monodromies in $O_i$, i.e.  if its monodromy invariant 
$(D,m)=(D,m_{z_{0}})$ satisfies the property (cf. \S~\ref{3.2})
\begin{equation}\label{e5.15}
n_{i}\quad \text{of the elements}\quad m([\gamma_{j}])\quad 
\text{belong to}\quad O_{i}\quad \text{for}\quad i=1,\ldots,k.
\end{equation}
\end{definition}
\begin{block}\label{5.15a}
Let $H^G_{\underline{n}}(Y,y_0)$ be the subset of $H^G_n(Y,y_0)$  consisting 
of the elements $(D,m)$ which satisfy Condition~(\ref{e5.15}). One has 
\[
H^G_n(Y,y_0) = \bigsqcup_{|\underline{n}|=n} H^G_{\underline{n}}(Y,y_0)
\]
and every $H^G_{\underline{n}}(Y,y_0)$ is an open subset in the 
Euclidean topology of $H^G_n(Y,y_0)$. Therefore every nonempty 
$H^G_{\underline{n}}(Y,y_0)$ is a union of connected components of \linebreak
$|H^G_n(Y,y_0)^{an}|$, so every nonempty $H^G_{\underline{n}}(Y,y_0)$ is a union 
of connected components in the Zariski topology of $H^G_n(Y,y_0)$ 
\cite[Cor.~2.6]{SGA1} and inherits the structure of algebraic variety from 
$H^G_n(Y,y_0)$.
\par
Suppose $H^G_{\underline{n}}(Y,y_0)\neq \emptyset$. Let us denote by
$p_{\underline{n}}:\mathcal{C}_{\underline{n}}(y_0)\to Y\times 
H^G_{\underline{n}}(Y,y_0)$ the restriction of the family 
$p:\mathcal{C}(y_0)\to Y\times H^G_n(Y,y_0)$ and let 
$\zeta_{\underline{n}}:H^G_{\underline{n}}(Y,y_0)\to 
\mathcal{C}_{\underline{n}}(y_0)$ be the restriction of the morphism 
$\zeta:H^G_n(Y,y_0)\to \mathcal{C}(y_0)$. Let us denote by 
\[
\mathcal{H}^G_{(Y,y_0),\underline{n}}:\Var_{\mathbb{C}}\to (\Sets)
\]
the moduli functor which associates with every  algebraic  variety  $S$ 
the set \linebreak
$\{[\mathcal{C}\to Y\times S,\eta]\}$ of smooth families of 
pointed $G$-covers of $(Y,y_0)$ of branching type $\underline{n}$  modulo 
$G$-equivalence and with every morphism $T\to S$ the pullback of families of 
$G$-covers. Theorem~\ref{5.8} implies the following one.
\end{block}
\begin{theorem}\label{5.16}
Let $Y$ be a smooth, projective, irreducible curve. Let $y_{0}\in Y$.
Let  $G$  be  a  finite group. Let $O_{1},\ldots,O_{k}$ be congugacy 
classes in $G$, $O_{i}\neq O_{j}$ if $i\neq j$. Let 
$\underline{n} = n_1O_1+\cdots+n_kO_k$ be a formal sum, where 
$n_{i}, i=1,\ldots,k$ are positive integers. Suppose that 
$H^G_{\underline{n}}(Y,y_0)\neq \emptyset$. The algebraic variety 
$H^G_{\underline{n}}(Y,y_0)$ is a 
fine moduli variety for the moduli functor 
$\mathcal{H}^{G}_{(Y,y_0),\underline{n}}$. The universal family is 
\begin{equation}\label{e5.16}
(p_{\underline{n}}:\mathcal{C}_{\underline{n}}(y_0)\to
Y\times H^G_{\underline{n}}(Y,y_0),
\zeta_{\underline{n}}:H^G_{\underline{n}}(Y,y_0)\to 
\mathcal{C}_{\underline{n}}(y_0)).
\end{equation}
\end{theorem}
\section{Parameterization of $G$-covers}\label{s4}
\begin{block}\label{4.1}
Let $p:C\to Y$ be a $G$-cover branched in $D\subset  Y$,  $|D|=n\geq  1$. 
Endowing   $C$   and   $Y$   with   the   canonical   Euclidean   topologies 
of $|C^{an}|$ and $|Y^{an}|$  respectively, consider the topological covering 
map $p':C'\to Y\setminus D$, where $C'=p^{-1}(Y\setminus D)$,  $p'=p|_{C'}$. 
For every $y\in Y\setminus D$ and every $z\in p^{-1}(y)$ let 
$m_{z}:\pi_1(Y\setminus D,y)\to G$ be the monodromy epimorphism defined in 
\S~\ref{3.2}: $m_{z}[\alpha]=g$ if $gz=z\alpha$. Every $m_{z}$ satisfies 
Condition~(\eqref{e3.3}) with closed paths based at $y$. Every two 
$m_{z_{1}}$ and $m_{z_{2}}$ are \emph{pathwise connected} 
(cf. \cite[\S~1.3]{Vo2}, \cite[Section~1]{K1}): there 
is a path 
$\tau:I\to Y\setminus D$, $\tau(0)=y_{1}=p(z_{1})$, 
$\tau(1)=y_{2}=p(z_{2})$, image of a path in $C'$ with initial point $z_1$ and terminal point $z_2$,  such that 
\[
m_{z_{2}}([\alpha]) = m_{z_{1}}[\tau\cdot\alpha\cdot\tau^{-1}]
:= m_{z_{1}}^{\tau}([\alpha])
\]
for every $[\alpha]\in \pi_1(Y\setminus D,y_2)$. The set 
$\underline{m}_{p}=\{m_{z}|z\in p^{-1}(Y\setminus D)\}$ forms an equivalence 
class with respect to pathwise connectedness in the set of epimorphisms 
$m:\pi_1(Y\setminus D,y)\tto G$, where
$y\in Y\setminus D$.
\end{block}
\begin{definition}\label{4.2}
Given a $G$-cover $p:C\to Y$ branched in $D\subset Y$ the pair 
$(D,\underline{m}_{p})$ is called the \emph{monodromy invariant} of $p$.
\end{definition}
\begin{definition}\label{4.2a}
Let $g=g(Y)$, let $n$ be a positive integer.
Let $G$ be a finite group which can be generated by $2g+n-1$  elements. 
We denote by $H^G_n(Y)$ the set of pairs 
$(D,\underline{m})$, where $D\in Y^{(n)}_{\ast}$ and $\underline{m}$  is  an 
equivalence  class  of  pathwise  connected  epimorphisms 
$m:\pi_1(Y\setminus D,y)\to G$, where $y\in Y\setminus D$, which satisfy 
Condition~(\eqref{e3.3}).
\end{definition}
\begin{block}\label{4.2b}
$H^G_n(Y)\neq \emptyset$ and the map $H^G_n(Y)\to Y^{(n)}_{\ast}$ given by 
$((D,\underline{m})\to D)$ is surjective.
Riemann's existence theorem yields that the mapping $[p:C\to Y]\mapsto 
(D,\underline{m}_p)$ 
stabilizes a  bijective  correspondence  between 
the set of $G$-equivalence classes of $G$-covers branched in  $n$  points 
and the set $H^G_n(Y)$. 
\par
Let $y\in Y\setminus D$. Let $m:\pi_1(Y\setminus D,y)\to G$ be an epimorphism. 
Then \linebreak
$m_{1}:\pi_1(Y\setminus D,y)\to G$ is pathwise connected with $m$ if  and 
only if $m_{1}=gmg^{-1}$ for some $g\in G$. Let  $U(y)\subset  H^G_n(Y)$  be 
the subset $\{(D,\underline{m})|y\notin D\}$. One has \linebreak
$H^G_n(Y)=\cup_{y\in Y}U(y)$. The map $H^G_n(Y,y)\to U(y)$, defined by 
$(D,m)\to (D,\underline{m})$ is invariant with respect to the action of 
$G$ on the set $H^G_n(Y,y)$  defined  by  $h\ast  (D,m)=(D,hmh^{-1})$.  This  action 
induces a free action of $\overline{G}=G/Z(G)$ on $H^G_n(Y,y)$. The set 
$U(y)$ is bijective to the quotient set
$_{G}\backslash H^G_n(Y,y) = _{\overline{G}}\backslash H^G_n(Y,y)$.
\end{block}
\begin{remark}
The set $H^G_n(\mathbb{P}^1)$ is the one denoted by $\mathcal{H}_n^{in}(G)$ in 
\cite[\S~1.2]{FV}.
\end{remark}
\begin{proposition}\label{4.3}
For every $y\in Y$ the action of
$G$ on $H^G_n(Y,y)$  defined  by \linebreak
 $h\ast (D,m) = (D,hmh^{-1})$ is an action by covering automorphisms of the 
\'{e}tale cover 
$\delta:H^G_n(Y,y)\to (Y\setminus y)^{(n)}_{\ast}$, where $\delta(D,m)=D$. 
The set $H^G_n(Y)$ can
be endowed with a structure of an algebraic variety which has the following 
properties.
\begin{enumerate}
\item
The map $H^G_n(Y)\to Y^{(n)}_{\ast}$ defined by 
$(D,\underline{m})\to D$ is a surjective, \'{e}tale, finite morphism.
\item
For every $y\in Y$ the set $U(y)$ is an affine open subset in $H^G_n(Y)$ and 
the map $\nu:H^G_n(Y,y)\to U(y)$ defined by $\nu(D,m) = (D,\underline{m})$ is 
an \'{e}tale Galois cover with respect to the $\ast$-action 
with Galois group $\overline{G}=G/Z(G)$.
\item
$H^G_n(Y)$ is a quasi-projective variety. If $Y\cong \mathbb{P}^{1}$ it is an 
affine variety.
\end{enumerate}
\end{proposition}
\begin{proof}
The action of $\overline{G}=G/Z(G)$ on $H^G_n(Y,y)$ is by covering 
homeomorphisms of the 
topological covering map $\delta:H^G_n(Y,y)\to (Y\setminus  y)^{(n)}_{\ast}$ 
since
\[
h\ast N_{(D,m)}(U_1,\ldots,U_n) = N_{(D,hmh^{-1})}(U_1,\ldots,U_n)\quad
\text{for}\quad \forall h\in G.
\]
The map $\delta$ is on the other hand an  \'{e}tale  cover  of  affine 
varieties, so by Corollary~\ref{Serre8} this action is by automorphisms of 
$H^G_n(Y,y)$. Endow $U(y)$ with a structure of  an  affine  variety  as  the 
quotient    $_{\overline{G}}\backslash    H^G_n(Y,y)$.    The     associated 
Euclidean topology, i.e. that of 
$(_{\overline{G}}\backslash H^G_n(Y,y))^{an}$, is 
the quotient topology of  $H^G_n(Y,y)^{an}$ by Lemma~\ref{6.16}.  Let  us  verify  the  patching 
condition for the subsets $U(y)\subset H^G_n(Y)$, $y\in Y$. Let 
$y_{1}, y_{2}\in Y$, $y_{1}\neq y_{2}$. The set $U(y_{1})\cap  U(y_{2})$  is 
Zariski open in $U(y_{i})$ for $i=1,2$ since it is the preimage of 
$(Y\setminus \{y_{1},y_{2}\})^{(n)}_{\ast}$ with respect to the morphism 
$U(y_{i})\to (Y\setminus y_i)^{(n)}_{\ast}$. The Euclidean topology of 
$U(y_{1})\cap U(y_{2})$ inherited from $U(y_{i})$ has a basis consisting  of 
the open sets $\nu_{i}(N_{(D,m_{i})}(U_1,\ldots,U_n))$, where 
$D\subset Y\setminus \{y_{1},y_{2}\}$, $m_{i}:\pi_1(Y\setminus D,y_i)\to G$, 
$\cup_{j=1}^{n}\overline{U}_{j}\subset Y\setminus \{y_{1},y_{2}\}$ and 
$\nu_{i}:H^G_n(Y,y_{i})\to U(y_{i})$ is the quotient map. Now, given 
$D\in (Y\setminus \{y_{1},y_{2}\})^{(n)}_{\ast}$ and $U_1,\ldots,U_n$ as in 
\S~\ref{3.5b} such that 
$\cup_{j=1}^{n}\overline{U}_{j}\subset Y\setminus  \{y_{1},y_{2}\}$  let  us 
choose a path 
$\tau:I\to Y\setminus  \cup_{j=1}^{n}\overline{U}_{j}$ such that 
$\tau(0)=y_{1}$, $\tau(1)=y_{2}$. Then one  has   
\[
\nu_{1}(N_{(D,m_{1})}(U_1,\ldots,U_n)) = 
\nu_{2}(N_{(D,m_{1}^{\tau})}(U_1,\ldots,U_n)).
\]
This shows that the two  Euclidean  topologies  on 
$U(y_{1})\cap U(y_{2})$ coincide. Applying Corollary~\ref{Serre9} to the 
map
\[
 U(y_{1})\cap U(y_{2}) \to (Y\setminus \{y_{1},y_{2}\})^{(n)}_{\ast},\quad
(D,\underline{m})\mapsto D
\]
we conclude that the two structures of algebraic varieties on 
$U(y_{1})\cap U(y_{2})$ inherited from $U(y_{1})$ and $U(y_{2})$  coincide. 
This shows that one can define on $H^G_n(Y)$ a structure of a reduced scheme 
over $\mathbb{C}$ such that every $U(y)$ is an affine open subset of 
$H^G_n(Y)$.  The 
map $H^G_n(Y)\to Y^{(n)}_{\ast}$, defined by 
$(D,\underline{m})\mapsto D$ is a  finite,  \'{e}tale,  surjective  morphism 
since these properties hold for $U(y)\to (Y\setminus y)^{(n)}_{\ast}$ for 
$\forall y\in Y$. This implies, in particular, that $H^G_n(Y)$ is a reduced, 
separated scheme  of  finite  type  over  $\mathbb{C}$,  i.e.  an  algebraic 
variety, since the open subset  $Y^{(n)}_{\ast}\subset  Y^{(n)}$  has  these 
properties. Parts~(i) and (ii) are proved. Part~(iii) is proved in 
\cite[Proposition~1.9]{K1}. 
\end{proof}
Let $y_0\in Y$. Let us define a left action of $G$ on $Y\times H^G_n(Y,y_0)$ 
by
\begin{equation}\label{e4.6}
h\ast (y,(D,m)) = (y,(D,hmh^{-1})).
\end{equation}
The open subset $Y\times H^G_n(Y,y_0)\setminus B$ (cf. \eqref{e3.6a2}) is $G$-invariant and by 
Proposition~\ref{4.3} $G$ acts on it by covering automorphisms of the 
\'{e}tale cover 
\[
Y\times H^G_n(Y,y_0)\setminus B \to 
Y\times (Y\setminus y_0)^{(n)}_{\ast}\setminus A
\]
(cf. \S~\ref{2.1}). For every $h\in G$ and every $z = (\Gamma_m[\alpha]_D,D,m)\in \mathcal{C}(y_0)'$ let us define 
$h\ast z\in \mathcal{C}(y_0)'$ as follows. Let $h=m([\eta]_{D})$, where 
$\eta$ is a loop based at $y_0$. Let 
\begin{equation}\label{e4.9}
h\ast (\Gamma_m[\alpha]_D,D,m) = 
(\Gamma_{hmh^{-1}}[\eta^{-}\cdot\alpha]_D,D,hmh^{-1}).
\end{equation}
\begin{proposition}\label{4.6a}
The following properties hold.
\begin{enumerate}
\item
$(h,z)\mapsto h\ast z$ is a left action of $G$ on the set $\mathcal{C}(y_0)'$.
\item
The two actions of $G$ on $\mathcal{C}(y_0)'$ defined by 
$(g,z)\mapsto gz$ and $(h,z)\mapsto h\ast z$ commute.
\item
$p'(h\ast  z)  =  h\ast  p'(z)$  for  $\forall  h\in   G$ and 
$\forall   z\in \mathcal{C}(y_0)'$.
\item
The $\ast$-action of $G$ on $\mathcal{C}(y_0)'$ is an action by covering 
automorphisms of the composed \'{e}tale cover
\begin{equation}\label{e4.6a}
\mathcal{C}(y_0)'\overset{p'}{\lto}    Y\times    H^G_n(Y,y_0)\setminus B\lto 
Y\times (Y\setminus y_0)^{(n)}_{\ast}\setminus A.
\end{equation}
\item
The $\ast$-action of $G$ on $\mathcal{C}(y_0)'$ can be uniquely extended 
to a left action of $G$ on $\mathcal{C}(y_0)$ by covering automorphisms of 
the composed finite morphism
\[
\mathcal{C}(y_0)\overset{p}{\lto}    Y\times    H^G_n(Y,y_0)
\overset{id\times \delta}{\lto} Y\times (Y\setminus y_0)^{(n)}_{\ast}.
\]
This action commutes with the action of $G$ on $\mathcal{C}(y_0)$ 
relative to the  Galois cover 
$p:\mathcal{C}(y_0)\to Y\times H^G_n(Y,y_0)$ and $p$ is equivariant:
$p(h\ast z)=h\ast p(z)$ for $\forall h\in G$ and 
$\forall z\in \mathcal{C}(y_0)$.
\item
For every $h\in G$ the automorphism of $\mathcal{C}(y_0)$ defined by 
$z\mapsto h\ast z$ induces a $G$-equivalence between 
$\mathcal{C}(y_0)_{(D,m)}\to Y$ and 
$\mathcal{C}(y_0)_{(D,hmh^{-1})}
\to Y$ for every $(D,m)\in H^G_n(Y,y_0)$.
\end{enumerate}
\end{proposition}
\begin{proof}
Let $z=(\Gamma_m[\alpha]_D,D,m)\in \mathcal{C}(y_0)'$.
\par
(i) Let $h_{1},h_{2}\in G$, $h_{i}=m([\eta_{i}])_{D}$.
\[
\begin{split}
(h_1h_2)\ast z &=
(\Gamma_{h_1h_2m(h_1h_2)^{-1}}[\eta_{2}^{-}\cdot\eta_{1}^{-}\cdot\alpha]_D,
D,h_1h_2m(h_1h_2)^{-1})\\
h_2\ast z &=(\Gamma_{h_2mh_2^{-1}}[\eta_2^{-}\cdot\alpha]_D,D,h_2mh_2^{-1}).
\end{split}
\]
$h_{1}=m([\eta_{1}]_{D}) = 
h_{2}mh_{2}^{-1}([\eta_{2}^{-}\cdot\eta_{1}\cdot\eta_{2}]_{D})$ and 
$(\eta_{2}^{-}\cdot\eta_{1}\cdot\eta_{2})^{-} = 
\eta_{2}^{-}\cdot\eta_{1}^{-}\cdot\eta_{2}$, hence
\[
h_{1}\ast(h_{2}\ast z) = 
(\Gamma_{h_1(h_2mh_2^{-1})h_1^{-1}}
[(\eta_2^{-}\cdot\eta_1^{-}\cdot\eta_2)\cdot\eta_2^{-}\cdot\alpha]_D,
D,h_1(h_2mh_2^{-1})h_1^{-1}).
\]
We see that $(h_1h_2)\ast z = h_1\ast (h_2\ast z)$.
\par
(ii) Let $g=m([\sigma]_{D})$, $h=m([\eta]_{D})$. Then
\[
\begin{split}
h\ast (gz) &=h\ast (\Gamma_m[\sigma\cdot\alpha]_D,D,m)
=(\Gamma_{hmh^{-1}}[\eta^{-}\cdot\sigma\cdot\alpha]_D,D,hmh^{-1})\\
h\ast z &=(\Gamma_{hmh^{-1}}[\eta^{-}\cdot\alpha]_D,D,hmh^{-1}).
\end{split}
\]
One has 
$g=m([\sigma]_{D})=hmh^{-1}([\eta^{-}\cdot\sigma\cdot\eta]_{D})$, 
so
\[
g(h\ast z) = (\Gamma_{hmh^{-1}}
[\eta^{-}\cdot\sigma\cdot\eta\cdot\eta^{-}\cdot\alpha]_D,D,hmh^{-1}).
\]
We see that $h\ast(gz)=g(h\ast z)$.
\par
(iii) Let $y=\alpha(1)$. Then $p'(z)=(y,(D,m))$ and according to \eqref{e4.9}
\[
p'(h\ast z)=((\eta^{-}\cdot\alpha)(1),D,hmh^{-1}) =
(y,D,hmh^{-1}) = h\ast p'(z).
\]
\par
(iv) Let $N_{(\alpha,D,m)}(U,U_1,\ldots,U_n)$ be as in \eqref{e3.12}. Let 
$h\in G$, $h=m([\eta]_{D})$, where $\eta$ is a closed path contained in 
$Y\setminus \cup_{i=1}^{n}\overline{U}_{i}$. Then
\[
h\ast N_{(\alpha,D,m)}(U,U_1,\ldots,U_n) =
N_{(\eta^{-}\cdot\alpha,D,hmh^{-1})}(U,U_1,\ldots,U_n).
\]
This shows that the $\ast$-action of $G$ on $\mathcal{C}(y_0)'$ is an action 
by covering homeomorphisms of the topological covering map 
$\mathcal{C}(y_0)'\to  Y\times  (Y\setminus  y_0)^{(n)}_{\ast}\setminus  A$. 
This map is a composition of \'{e}tale covers \eqref{e4.6a}, so 
by Corollary~\ref{Serre8} the $\ast$-action is an action by  covering  automorphisms  of 
the composed \'{e}tale cover.
\par
(v) It is clear from the definition of $\mathcal{C}(y_0)$  as  the  disjoint 
union of normalizations (cf. \S~\ref{3.32}) that the $\ast$-action of $G$ on 
$\mathcal{C}(y_0)'$ can be uniquely extended to a left action of $G$ on 
the algebraic variety $\mathcal{C}(y_0)$. The other statements follow from 
(iv), (ii) and (iii).
\par
(vi) This follows from (v).
\end{proof}
Our next goal is, provided $Z(G)=1$, to construct a smooth family of 
$G$-covers of $Y$ branched in $n$ points $\pi:\mathcal{C}\to Y\times H^G_n(Y)$ 
such that 
for every $(D,\underline{m})\in H^G_n(Y)$ the  $G$-cover
$\mathcal{C}_{(D,\underline{m})}\to Y$ has monodromy invariant 
$(D,\underline{m})$. 
\begin{block}\label{4.7}
Let   $G$  be a finite group with  
trivial center. We use the following construction due to H.~V\"{o}lklein 
\cite[\S~10.1.3.1]{Vo1}. Let
\begin{equation}\label{e4.7a}
\begin{split}
\mathcal{C}'&= \{(y,D,m)|D\in Y^{(n)}_{\ast}, y\in Y\setminus D,\\
&m:\pi_1(Y\setminus D,y)\tto G\quad \text{satisfies Condition}~(\eqref{e3.3})\}
\end{split}
\end{equation}
Let $\underline{B}=\{(y,(D,\underline{m}))|y\in D\}\subset Y\times H^G_n(Y)$.  Consider the map 
\[
\pi':\mathcal{C}'\to Y\times H^G_n(Y)\setminus \underline{B},\quad 
(y,D,m)\overset{\pi'}{\mapsto} (y,(D,\underline{m})).
\]
 One defines a left action of $G$ on 
$\mathcal{C}'$ by
\begin{equation}\label{e4.7}
g(y,D,m) = (y,D,gmg^{-1}).
\end{equation}
This action is free since $Z(G)=1$, $\pi'$ is $G$-invariant and 
the quotient set $\mathcal{C}'/G$ is bijective to
$Y\times H^G_n(Y)\setminus \underline{B}$. 
\par
Let $y_{0}\in Y$. Let 
$\mathcal{C}[y_0]'=\pi'^{-1}(Y\times U(y_{0})\setminus \underline{B})$.
Let 
$\kappa':\mathcal{C}(y_0)'\to \mathcal{C}[y_0]'$
be the map defined by 
\begin{equation}\label{e4.8a}
\kappa'((\Gamma_m[\alpha]_D,D,m)) = (\alpha(1),D,m^{\alpha}).
\end{equation}
This map is $G$-equivariant. Indeed, let $g=m([\sigma]_{D})$. Then 
$g(\Gamma_m[\alpha]_D,D,m)   =    (\Gamma_m[\sigma\cdot\alpha]_D,D,m)$    is 
transformed by $\kappa'$ in 
$(\sigma\cdot\alpha(1),D,m^{\sigma\cdot\alpha}) = 
(\alpha(1),D,gm^{\alpha}g^{-1})$. The map $\kappa'$ fits in the following 
commutative diagram
\begin{equation}\label{e4.8b}
\xymatrix{
\mathcal{C}(y_0)'\ar[r]^-{\kappa'}\ar[d]_-{p'}&\mathcal{C}[y_{0}]'
\ar[d]^-{\pi'}\\
Y\times H^G_n(Y,y_0)\setminus B\ar[r]^-{(id\times \nu)'}&Y\times 
U(y_0)\setminus \underline{B}
}
\end{equation}
where $\nu:H^G_n(Y,y_0)\to U(y_{0})$, given by $\nu(D,m)=(D,\underline{m})$, 
is an \'{e}tale Galois cover with Galois group $G$ (cf. 
Proposition~\ref{4.3}(ii)) and $(id_Y\times \nu)'$ is the quotient morphism of the 
$\ast$-action of $G$ (cf. \eqref{e4.6}. This action is moreover free since $Z(G)=1$.
\end{block}
\begin{lemma}\label{4.9}
Suppose $G$ has trivial center.  The $\ast$-action of $G$ on 
$\mathcal{C}(y_0)'$ 
(cf. Proposition~\ref{4.6a}) is free, 
$\kappa':\mathcal{C}(y_0)'\to \mathcal{C}[y_0]'$ is invariant 
with respect to it and every fiber of $\kappa'$ is a $G$-orbit. 
The set $\mathcal{C}[y_0]'$ has a structure of a quotient 
algebraic variety variety $_{G}\backslash \mathcal{C}(y_0)'$, 
$\kappa':\mathcal{C}(y_0)'\to \mathcal{C}[y_0]'$ is an \'{e}tale 
Galois cover and the map 
$\pi':\mathcal{C}[y_0]'\to Y\times U(y_{0})\setminus \underline{B}$  is  an 
\'{e}tale Galois cover whose Galois group is isomorphic to $G$ with respect to the 
action \eqref{e4.7} of $G$ on 
$\mathcal{C}[y_0]'$. The morphism $\kappa':\mathcal{C}(y_0)'\to \mathcal{C}[y_0]'$
is equivariant with respect to the actions of $G$ as Galois groups of the 
covers $p'$ and $\pi'$. 
\end{lemma}
\begin{proof}
The morphism $p'$ is equivariant with respect to the $\ast$-action of $G$ by 
Proposition~\ref{4.6a}(iii) and $G$ acts without fixed points on 
$Y\times H^G_n(Y,y_0)\setminus B$, so the $\ast$-action of $G$ on 
$\mathcal{C}(y_0)'$ is free.
Let $h=m([\eta]_{D})$, $z=(\Gamma_m[\alpha]_D,D,m)\in \mathcal{C}(y_0)'$. Then
\[
k'(h\ast z) = 
(\eta^{-}\cdot\alpha(1),D,(hmh^{-1})^{\eta^{-}\cdot\alpha}) = 
(\alpha(1),D,m^{\alpha}) = \kappa'(z).
\]
Let $\kappa'(z)=\kappa'(z_{1})$, where 
$z_{1}=(\Gamma_{m_1}[\beta]_D,D,m_1)$. Then $\alpha(1)=\beta(1)$ and 
$m^{\alpha}=m_1^{\beta}$. Let $\eta=\alpha\cdot\beta^{-}$,  $h=m([\eta]_{D})$. 
Then $[\beta]_{D}=[\eta^{-}\cdot\alpha]_{D}$, 
$m_{1}=m^{\alpha\cdot\beta^{-}}=hmh^{-1}$. Therefore $z_{1}=h\ast z$.
\par
By Proposition~\ref{4.6a} the group $G\times G$ acts on  $\mathcal{C}(y_0)'$ 
by automorphisms as $(g,h)z=g(h\ast z)$. This action is free and 
$(id_{Y}\times \nu)'\circ p':\mathcal{C}(y_0)'\to 
Y\times U(y_{0})\setminus \underline{B}$ is the associated \'{e}tale Galois cover
with Galois group isomorphic to
$G\times G$. Let us endow the set $\mathcal{C}[y_0]'$ with 
the structure of the quotient algebraic variety 
$_{G}\backslash \mathcal{C}(y_0)'$ with respect to the $\ast$-action (cf. Lemma~\ref{6.16}). The
map $\kappa':\mathcal{C}(y_0)'\to \mathcal{C}[y_0]'$ becomes an 
\'{e}tale Galois cover with Galois group isomorphic to $G$. The action $(g,z)\mapsto gz$ descends to 
$\mathcal{C}[y_0]'$ as the action \eqref{e4.7}, since $\kappa'$ is 
$G$-equivariant. Hence 
$\pi':\mathcal{C}[y_0]'\to Y\times U(y_{0})\setminus \underline{B}$  is  an 
\'{e}tale Galois cover.
\end{proof}
\begin{block}\label{4.12a}
The topological space $|\mathcal{C}[y_0]'^{an}|$  is  the  quotient by $G$  of  the 
topological space $|\mathcal{C}(y_0)'^{an}|$ (cf. Lemma~\ref{6.16}). Let 
$(y,D,m)\in  \mathcal{C}[y_0]'$. 
Let $m=m_{0}^{\alpha}$, where 
\[
m_{0}:\pi_1(Y\setminus D,y_0)\tto G,\quad
\alpha:I\to Y\setminus D,\quad \alpha(0)=y_{0},\; \alpha(1)=y.
\] 
Let $N_{(\alpha,D,m_0)}(U,U_1,\ldots,U_n)$ be as in \eqref{e3.12}.
Then 
$\kappa'(N_{(\alpha,D,m_{0})}(U,U_1,\ldots,U_n))$ is a neighborhood 
of $(y,D,m)$ in the Euclidean topology of $\mathcal{C}[y_0]'$. Let us denote 
it by $N_{(y,D,m)}(U,U_1,\ldots,U_n)$. One has by \eqref{e4.8a} 
that
\begin{equation}\label{e4.12a}
\begin{split}
N_{(y,D,m)}(U,U_1,\ldots,U_n) =
\{&(z,E,m(E)^{\tau})|z\in U, E\in N_D(U_1,\ldots,U_n),\\
&\tau:I\to U, \tau(0)=y, \tau(1)=z\}.
\end{split}
\end{equation}
Varying the embedded open disks $U\ni y$, $U_{i}\ni b_{i}$,  $i=1,\ldots,n$, 
one obtains a neighborhood basis of $(y,D,m)$ in the topological space $|\mathcal{C}[y_0]'^{an}|$.
\end{block}
\begin{proposition}\label{4.13}
Suppose $G$ has trivial center. The set $\mathcal{C}'$  (cf.  \eqref{e4.7a}) 
has a structure of an algebraic variety such that for every $y\in Y$ the 
map $_{G}\backslash\mathcal{C}(y)'\to \mathcal{C}[y]'\subset \mathcal{C}'$ 
is an open embedding. The map $\pi':\mathcal{C}'\to 
Y\times H^G_n(Y)\setminus \underline{B}$ is an \'{e}tale Galois cover with Galois group
isomorphic to $G$, where 
$G$ acts as in \eqref{e4.7}.
\end{proposition}
\begin{proof}
One has $\mathcal{C}'=\cup_{y\in Y}\mathcal{C}[y]'$ and each 
$\mathcal{C}[y]'$ has the structure of an algebraic variety defined in 
Lemma~\ref{4.9}. Let us verify the patching conditions. Let 
$y_{1},y_{2}\in Y$, $y_1\neq y_2$. Then 
$\mathcal{C}[y_1]'\cap \mathcal{C}[y_2]'$ is Zariski open in 
$\mathcal{C}[y_i]'$, $i=1,2$, since it is the preimage of 
$Y\times (Y\setminus\{y_{1},y_{2}\})^{(n)}$ with respect to the 
composed morphism $\mathcal{C}[y_i]'\to Y\times U(y_{i})\to
Y\times (Y\setminus y_{i})^{(n)}$. 
It is clear from \S~\ref{4.12a} that  the 
two  Euclidean  topologies  on   $\mathcal{C}[y_1]'\cap   \mathcal{C}[y_2]'$ 
induced by $\mathcal{C}[y_1]'$ and $\mathcal{C}[y_2]'$ coincide. Applying 
Corollary~\ref{Serre9} to the map 
$\mathcal{C}[y_1]'\cap   \mathcal{C}[y_2]'\to 
Y\times (Y\setminus\{y_{1},y_{2}\})^{(n)}$ given by 
$(y,D,m)\mapsto (y,D)$ we conclude that the two structures of algebraic 
varieties on $\mathcal{C}[y_1]'\cap   \mathcal{C}[y_2]'$ inherited from 
$\mathcal{C}[y_1]'$ and $\mathcal{C}[y_2]'$ coincide. This shows that 
one can endow $\mathcal{C}'$ with a structure of a reduced scheme over 
$\mathbb{C}$ such that every $\mathcal{C}[y]'$ is a Zariski open subset 
of $\mathcal{C}'$. The map 
$\pi':\mathcal{C}'\to Y\times H^G_n(Y)\setminus \underline{B}$ given by 
$\pi'(y,D,m)=(y,(D,\underline{m}))$ is an \'{e}tale Galois cover since  this 
property holds  for  every  $U(y)$  by  Lemma~\ref{4.9}.  This  implies,  in 
particular, that $\mathcal{C}'$ is a separated scheme of finite type over 
$\mathbb{C}$ since these properties hold for 
$Y\times H^G_n(Y)\setminus \underline{B}$. This shows that $\mathcal{C}'$ is 
an algebraic variety.
\end{proof}
\begin{block}\label{4.15}
Let $H$ be a connected component of $H^G_n(Y)$. Let 
$\mathcal{C}'_{H}=
\pi'^{-1}(Y\times  H\setminus  \underline{B})$.  We  claim 
that this algebraic variety is irreducible. It suffices to prove the 
irreducibility of 
$\mathcal{C}'_{H\cap U(y_{0})}
=\pi'^{-1}(Y\times  H\cap U(y_{0})\setminus  \underline{B})$ for every 
$y_{0}\in Y$. Let $\tilde{H}$ be a  connected  component  of  $H^G_n(Y,y_0)$ 
which     maps     surjectively     to      $H\cap      U(y_{0})$.      Then 
$\mathcal{C}(y_0)'_{\tilde{H}}$ is irreducible by \S~\ref{3.23} and it 
maps surjectively onto $\mathcal{C}'_{H\cap U(y_{0})}$ by the morphism 
$\kappa':\mathcal{C}(y_0)'\to \mathcal{C}[y_0]'$ (cf. Lemma~\ref{4.9}), hence 
$\mathcal{C}'_{H\cap U(y_{0})}$ is irreducible. Let $\mathcal{C}_{H}$
be the normalization of  $Y\times  H$ 
in the field $\mathbb{C}(\mathcal{C}'_{H})$
and let $\pi_{H}:\mathcal{C}_{H}\to Y\times H$ be the corresponding 
finite, surjective morphism. The action of $G$ on $\mathcal{C}'_{H}$ 
(cf. Proposition~\ref{4.13}) can be uniquely extended to an action of 
$G$ on $\mathcal{C}_{H}$ by algebraic automorphisms.
Let 
\[
\mathcal{C} = \bigsqcup_{H\subset H^G_n(Y)}\mathcal{C}_{H}.
\]
Let $\pi:\mathcal{C}\to Y\times H^G_n(Y)$ be the finite morphism which 
restricts to $\pi_{H}$ over every connected component $H\subset H^G_n(Y)$. 
The $G$-invariant morphism $\pi:\mathcal{C}\to Y\times H^G_n(Y)$ is a 
Galois cover with Galois group $G$. 
\par
For every $y_{0}\in Y$ let $\mathcal{C}[y_0]=\pi^{-1}(Y\times  U(y_0))$  and 
let $\kappa:\mathcal{C}(y_0)\to \mathcal{C}[y_0]$ be the morphism,
extension of 
$\kappa':\mathcal{C}(y_0)'\to \mathcal{C}[y_0]'$ relative to the 
normalizations. Every $\mathcal{C}[y]$ is a Zariski open, dense subset of 
$\mathcal{C}$ and  
\[
  \mathcal{C}=\cup_{y\in Y}\mathcal{C}[y].
\]
\end{block}
\begin{lemma}\label{4.15a}
Let $Z(G) = 1$. Let $y_{0}\in Y$. The $\ast$-actions of $G$ on 
$Y\times H^G_n(Y,y_0)$ and $\mathcal{C}(y_0)$ defined in \eqref{e4.6} and 
Proposition~\ref{4.6a}(v) are without fixed  points.  There  is  a  
commutative diagram of morphisms,
extension of Diagram~\eqref{e4.8b},
\begin{equation}\label{e4.17}
\xymatrix{
\mathcal{C}(y_0)\ar[r]^-{\kappa}\ar[d]_-{p}&\mathcal{C}[y_{0}]
\ar[d]^-{\pi}\\
Y\times H^G_n(Y,y_0)\ar[r]^-{id\times \nu}&Y\times U(y_0)
}
\end{equation}
where  $\nu(D,m)=(D,\underline{m})$,  the  two  horizontal   morphisms   are 
invariant with respect to the free $\ast$-actions of $G$ and are isomorphic 
to the respective quotient morphisms. The variety $\mathcal{C}[y_0]$ 
is smooth. The morphism $\kappa$ is equivariant with respect to the 
actions of $G$ as Galois groups of the covers $p$ and $\pi$.
\end{lemma}
\begin{proof}
The morphism $p$ is equivariant 
with respect to the $\ast$-actions of $G$ by Proposition~\ref{4.6a}(v) and 
$G$ acts without fixed points  on  $Y\times  H^G_n(Y,y_0)$  since  
$Z(G)=~1$.\linebreak 
Hence the $\ast$-action of $G$ on $\mathcal{C}(y_0)$ is without fixed points.
The morphism \linebreak
$\kappa:\mathcal{C}(y_0)\to \mathcal{C}[y_0]$ is invariant 
with respect to the $\ast$-action of $G$ since this property holds for 
$\kappa':\mathcal{C}(y_0)'\to  \mathcal{C}[y_0]'$  by  Lemma~\ref{4.9}.  The 
quotient algebraic variety $_{G}\backslash \mathcal{C}(y_0)$ is well-defined 
by Proposition~\ref{4.6a}(v) and Lemma~\ref{6.16}. The smoothness of 
$\mathcal{C}(y_0)$ 
(cf. Proposition~\ref{3.24}) implies the smoothness of 
$_{G}\backslash \mathcal{C}(y_0)$  since  the  $\ast$-action of $G$ is  free.  The 
morphism $(id_{Y}\times \nu)\circ p:\mathcal{C}(y_0)\to Y\times U(y_{0})$ is 
finite and surjective, so the same holds for the morphism 
$_{G}\backslash \mathcal{C}(y_0)\to Y\times U(y_{0})$. Furthermore 
$_{G}\backslash \mathcal{C}(y_0)$ contains a Zariski open, dense subset 
isomorphic to $_{G}\backslash \mathcal{C}(y_0)'\cong \mathcal{C}[y_0]'$. 
Therefore 
$\kappa:\mathcal{C}(y_0)\to \mathcal{C}[y_0]$ induces an 
isomorphism 
$_{G}\backslash \mathcal{C}(y_0)\overset{\sim}{\lto}\mathcal{C}[y_0]$ by the 
uniqueness of normalizations. The morphism $\kappa$ is $G$-equivariant since 
this property holds for $\kappa'$ (cf. Lemma~\ref{4.9}).
\end{proof}
\begin{theorem}\label{4.16}
Let $Y$ be a smooth, projective, irreducible  curve of genus $g\geq 0$.  Let 
$n$ be a positive integer. Let $G$ be a finite group which 
can be generated by $2g+n-1$ elements. Suppose $G$ has trivial center. Then
the morphism 
\begin{equation}\label{e4.16}
\pi: \mathcal{C} \to Y\times H^G_n(Y)
\end{equation}
is a smooth family of $G$-covers of $Y$ branched in $n$ points. For every 
$(D,\underline{m})\in H^G_n(Y)$ the $G$-cover 
$\mathcal{C}_{(D,\underline{m})}\to Y$ has monodromy invariant 
$(D,\underline{m})$. Every $G$-cover $C\to Y$ branched in $n$ points is 
$G$-equivalent to a unique $G$-cover of $Y$ of the family \eqref{e4.16}.
\end{theorem}
\begin{proof}
The algebraic variety $\mathcal{C}$ is smooth since every one of its 
Zariski open subsets 
$\mathcal{C}[y]\cong {_{G}\backslash \mathcal{C}(y)}$ is smooth by 
Lemma~\ref{4.15a}. The composition 
$\mathcal{C}\overset{\pi}{\lto}Y\times H^G_n(Y)\to H^G_n(Y)$ is proper since 
$\pi$ is finite and $Y$ is projective. This is a morphism of smooth varieties 
of relative dimension 1 and for every $z\in \mathcal{C}$ the induced linear 
map on the tangent spaces is surjective. Indeed, this property holds for 
$\pi_2\circ p:\mathcal{C}(y_0)\to H^G_n(Y,y_0)$  for 
$\forall y_{0}\in Y$ (cf. Theorem~\ref{3.33}), the morphism $\pi_2\circ p$ is 
equivariant with respect to the free $\ast$-actions of $G$ and one applies 
Lemma~\ref{4.15a}. By 
\cite[Ch.~III Prop.~10.4]{Hart} we conclude that 
$\mathcal{C}\to H^G_n(Y)$ is a smooth morphism.
\par
Let $(D,\underline{m})\in H^G_n(Y)$. Let $y_{0}\in Y\setminus D$ and let 
$m:\pi_1(Y\setminus D,y_0)\to G$ belong to $\underline{m}$. One has 
$\mathcal{C}_{(D,\underline{m})}=\pi^{-1}(Y\times    \{(D,\underline{m})\})$ 
and by \eqref{e4.17}
\[
\kappa^{-1}(\mathcal{C}_{(D,\underline{m})})=
\bigsqcup_{h\in G}\mathcal{C}(y_0)_{(D,hmh^{-1})}
\]
where $k|_{\mathcal{C}(y_0)_{(D,m)}}:\mathcal{C}(y_0)_{(D,m)}
\to \mathcal{C}_{(D,\underline{m})}$ is a $G$-equivariant isomorphism. 
Using Theorem~\ref{3.33} we conclude that 
$\mathcal{C}_{(D,\underline{m})}\to    Y$    has     monodromy     invariant 
$(D,\underline{m})$. The last statement is clear from \S~\ref{4.1} and 
\S~\ref{3.3}.
\end{proof}
\begin{block}\label{4.20}
The map $\pi:\mathcal{C}\to Y\times H^G_n(Y)$ has the following local analytic 
form at the ramification points. Let 
$\pi(z)=(b,(D,\underline{m}))$, where $D=\{b_1,\ldots,b_{k},\ldots b_n\}$, $b=b_{k}$. 
The isotropy 
group $G(z)\subset G$ is cyclic of order $e\geq 2$. There are local analytic 
coordinates $(s,t_1,\ldots ,t_n)$ of $\mathcal{C}$ at $z$ such that the 
map $\pi$ and the action of $G(z)$ are given locally at $z$ by
\[
\begin{split}
\pi: (s,t_1,\ldots ,t_n) &\mapsto (s^{e}+t_{k},t_{1},\ldots,t_{k},\ldots,t_{n})\\
g(s,t_1,\ldots ,t_n) &= (\chi(g)s,t_1,\ldots ,t_n)
\end{split}
\]
where $\chi:G(z)\to \mathbb{C}^{\ast}$ is a primitive  character  of  $G(z)$. 
This follows from Proposition~\ref{3.24} and Lemma~\ref{4.15a}, since the 
horizontal morphisms $\kappa$ and $id_{Y}\times \nu$  in  \eqref{e4.17}  are 
locally biholomorphic.
\end{block}
\section{Hurwitz moduli varieties parameterizing $G$-covers}\label{s6}
In this section we assume that $Y$ is a smooth, projective, irreducible  curve 
of genus $g\geq 0$,  
$n$ is a positive integer and $G$ is a finite group which 
can be generated by $2g+n-1$ elements.
\begin{proposition}\label{6.1}
Let $q:X\to Y\times S$ be a smooth family of $G$-covers of $Y$ branched 
in $n$ points. Let $B\subset Y\times S$ be the branch locus of $q$. Let 
$v:S\to H^G_n(Y)$ be the map 
\begin{equation}\label{e6.1}
v(s) = (B_{s},\underline{m}_{s}),\quad s\in S
\end{equation}
where $(B_{s},\underline{m}_{s})$ is the monodromy invariant of 
$q_{s}:X_{s}\to Y$ (cf. Definition~\ref{4.2}). Then $v$ is a morphism.
\end{proposition}
\begin{proof}
Let $\beta:S\to Y^{(n)}$ be the morphism defined by $\beta(s)=B_{s}$ 
(cf. Proposition~\ref{2.7}(vi)). One has 
$H^G_n(Y) = \cup_{y\in Y}U(y)$, every $U(y)$ is a Zariski open subset 
and $v^{-1}(U(y)) = \beta^{-1}((Y\setminus y)^{(n)})$ is a Zariski open 
subset of $S$ for every $y\in Y$. Proving that $v$ is a morphism is a  local 
matter so we may assume, without loss of generality,  that there is a point 
$y_{0}\in Y$ such that every $q_{s}:X_{s}\to Y$ is unramified at $y_{0}$. 
Then $v(S)\subset U(y_{0})\subset H^G_n(Y)$. One has 
$\{y_{0}\}\times S\subset Y\times S\setminus B$. 
Let $T=q^{-1}(\{y_{0}\}\times S$). The composition 
$\mu:T\to \{y_{0}\}\times S\to S$ is finite, \'{e}tale, $G$-invariant and
$T/G \cong S$. In fact, $X/G \cong Y\times S$  by Proposition~\ref{2.7}(viii),
$T\to \{y_0\}\times S$ is a pullback of $q:X\to Y\times S$ by $\{y_0\}\times S\to Y\times S$
 and one applies  
\cite[Prop.~A~7.1.3]{KM}. Let $q_{T}:X_{T}=X\times_{S}T\to Y\times T$ be the 
pullback family. The morphism $\theta:T\to X\times_{S}T$ defined by 
$\theta(t)=(t,t)$ satisfies $q_{T}\circ \theta(t)=(y_{0},t)$ for 
$\forall t\in T$, so $(X_{T}\to X\times T,\theta)$ is a smooth family of 
pointed $G$-covers of $(Y,y_0)$. For every $t\in T$, if $s=\mu(t)$,
one has $(X_{T})_{t}=X_{s}\times \{t\}$, the monodromy invariant of 
$(X_{s},t)\to (Y,y_{0})$ is 
$(B_{s},m_{t}:\pi_1(Y\setminus B_{s},y_{0})\to G)$, $m_{t}$ belongs to 
$\underline{m}_{s}$ and  $m_{ht}=hm_{t}h^{-1}$  for  $\forall  h\in  G$.  We 
obtain the following commutative diagram of maps
\[
\xymatrix{
T\ar[r]^-{u}\ar[d]_-{\mu}&H^G_n(Y,y_0)\ar[d]^-{\nu}\\
S\ar[r]^-{v}&U(y_0)
}
\]
where $u(t)=(B_{\mu(t)},m_{t})$, $v(s)=(B_{s},\underline{m}_{s})$ and 
the vertical maps are quotient morphisms with respect to the actions of 
$G$ defined respectively by $t\mapsto ht$ and $h\ast (D,m)=(D,hmh^{-1})$ 
for $\forall h\in G$ (cf. Proposition~\ref{4.3}). By  Proposition~\ref{5.4} 
the map $u$ is a morphism, hence $v$ is a morphism.
\end{proof}
The next proposition is a partial inverse of Proposition~\ref{6.1}.
\begin{proposition}\label{6.3a}
Let $v:S\to H^G_n(Y)$ be a morphism. For every $s\in S$ there exists a 
Zariski open  neighborhood  $U$  of  $s$,  an  \'{e}tale  Galois 
cover $\mu: \tilde{U}\to U$ with Galois group $\overline{G}=G/Z(G)$ and a 
smooth family of $G$-covers of $Y$ branched in $n$ points 
$q:X\to Y\times \tilde{U}$ such that $v|_U\circ \mu$ equals the morphism 
$\tilde{v}: \tilde{U}\to H^G_n(Y)$ associated with $q:X\to Y\times \tilde{U}$.
\end{proposition}
\begin{proof}
Let $v(s)\in U(y_{0})$ for some $y_{0}\in Y$. Let $U=v^{-1}(U(y_0))$. Let 
$\nu:H^G_n(Y,y_0)\to U(y_0)$ be the \'{e}tale morphism 
$\nu(D,m)=(D,\underline{m})$ 
(cf. Proposition~\ref{4.3}(ii)). Let 
$\tilde{U}=U\times_{U(y_0)}H^G_n(Y,y_0)$.  One has a Cartesian diagram
\[
\xymatrix{
\tilde{U}\ar[r]^-{u}\ar[d]_-{\mu}&H^G_n(Y,y_0)\ar[d]^-{\nu}\\
U\ar[r]^-{v|_U}&U(y_0)
}
\]
in which $\mu:\tilde{U}\to U$ is an  \'{e}tale  Galois  cover 
with Galois group $\overline{G}$ since this property holds for $\nu$
(cf.\cite[Prop.~A~7.1.3]{KM}). Let 
$(q:X\to Y\times \tilde{U},\eta:\tilde{U}\to X)$ be the pullback by $u$ of 
the universal family \eqref{e5.8}. Then the morphism 
$\tilde{v}:\tilde{U}\to H^G_n(Y)$ associated with $q:X\to Y\times \tilde{U}$ 
equals $\nu\circ u$. Hence $v|_U\circ \mu = \tilde{v}$. 
\end{proof}
In the next proposition we give the local analytic form at the  ramification 
points of an arbitrary smooth family of $G$-covers of $Y$ branched in $n$ 
points.
\begin{proposition}\label{6.18}
Let $q:X\to Y\times S$ be a smooth family of $G$-covers of  $Y$  branched 
in $n$ points. Let $B\subset Y\times S$ be the branch locus of $q$. 
Let  $\beta:S\to  Y^{(n)}_{\ast}$ 
be the morphism defined by  $\beta(s)=B_{s}$  (cf.  Proposition~\ref{2.7}(vi)).
Let $x\in X$ be a point such that $q(x) = (b,s_0) \in B$. Let $B_{s_0} = D = \{b_1,\ldots,b_k,\ldots,b_n\}$,
$b=b_k$. Let $y_0\in Y\setminus D$ and let $U_i\ni b_i$, $s_i:U_i\to \mathbb{C}$ be as in \S~\ref{3.23b}.
Denote the restriction of $(s_1,\ldots,s_n)\circ\beta$ on $\beta^{-1}(N_D(U_1,\ldots,U_n))$
by $(\beta_1,\ldots,\beta_n)$. There exist open
neighborhoods $V\subset |S^{an}|$ of $s_0$ and $W\subset |X^{an}|$ of $x$ such that 
$\beta(V) \subset N_D(U_1,\ldots,U_n)$, $q(W)=U\times V$ and the following properties hold.
\begin{enumerate}
\item
The isotropy group $G(x)$ is cyclic of order $e\geq 2$.
Let $F\subset \mathbb{C}\times U\times V$ be the analytic subset
$F=\{(z,t,s)|z^e=t-\beta_k(s)\}$ and let $q_1:F\to U\times V$ be the projection map. There exists
a biholomorphic map $\phi:W\to F$ such that $q|_W=q_1\circ \phi$.
\item
The composition $\psi = (z,id_V)\circ \phi : W\to \mathbb{C}\times V$ maps $W$  biholomorphically
   onto an open neighborhood of $(0,s_0)$.
\item
$W$ is $G(x)$-invariant and there exists a primitive character $\chi:G(x)\to \mathbb{C}^*$ such that
$\phi$ and $\psi$ are $G(x)$-equivariant with respect to the actions of $G(x)$ on $F$ and $\mathbb{C}\times V$
defined respectively by $g(z,t,s)=(\chi(g)z,t,s)$ and $g(z,s) = (\chi(g)z,s)$.
\item
There is a $G$-equivariant biholomorphic map 
$q^{-1}(U\times V)\cong G\times^{G(x)}W$.
\end{enumerate}
\end{proposition}
\begin{proof}
Replacing $S$ by an \'{e}tale cover of a  Zariski  open  neighborhood  of 
$s_{0}$, as in the proof of Proposition~\ref{6.1}  we  may  assume  that  there 
exists a morphism $\eta:S\to X$ such that 
$(q:X\to Y\times S,\eta)$ is a smooth family of pointed $G$-covers of 
$(Y,y_0)$. Let $u:S\to H^G_n(Y,y_0)$ and $f:X\to  \mathcal{C}(y_0)$  be  the 
morphisms of Theorem~\ref{5.8} (cf. \eqref{e5.9}). Let  $u(s_{0})=(D,m),  f(x)=w$. 
Denote $H^G_n(Y,y_0)$ by $H$, $N_{(D,m)}(U_1,\ldots,U_n)$ of \S~\ref{3.23b} by 
$N$,   the    neighborhood    of    $w\in    \mathcal{C}(y_0)$    of 
Proposition~\ref{3.24}(iii) by $\mathcal{W}$. One has by Theorem~\ref{5.8} that 
$X \cong S\times_{H}\mathcal{C}(y_0)$. Therefore 
$X^{an} \cong S^{an}\times_{H^{an}}\mathcal{C}(y_0)^{an}$ 
(cf. \cite[\S~1.2]{SGA1}). Let $W=(f^{an})^{-1}(\mathcal{W})$. 
Then $W\cong \mathcal{W}\times_{\mathcal{C}(y_0)^{an}}X^{an}$ (cf. \cite[Prop.~0.27]{Fi}. Hence
\begin{equation}\label{e6.19}
W\cong X^{an}\times_{\mathcal{C}(y_0)^{an}}\mathcal{W}
\cong (S^{an}\times_{H^{an}}\mathcal{C}(y_0)^{an})
\times_{\mathcal{C}(y_0)^{an}}\mathcal{W}
\cong S^{an}\times_{H^{an}}\mathcal{W}.
\end{equation}
Let $V=u^{-1}(N)$. Then $\beta|_V=\delta|_N\circ u|_V$. We have $t_i=s_i\circ \delta|_N$,
$i=1,\ldots,n$ (cf. \S~\ref{3.23b}), therefore $\beta_i|_V = t_i\circ u^{an}|_V$. Let 
$E$ be the analytic subset of $\mathbb{C}\times U\times N$ defined by the equation 
$z^e=t-t_k$ (cf. Proposition~\ref{3.24}(iii)). The inverse image 
$E_1=(id_{\mathbb{C}}\times id_U\times u^{an})^{-1}(E)$ is the closed complex subspace
of $\mathbb{C}\times U\times V$ whose ideal sheaf is generated by the holomorphic function 
$h(z,t,s) = z^e - (t-\beta_k(s))$ (cf. \cite[Prop.~0.27]{Fi}). Let 
$\varphi:\mathcal{W}\to E$ be the biholomorphic map of Proposition~\ref{3.24}(iii). By base change
$\varphi_1:W_1=\mathcal{W}\times_EE_1\to E\times_EE_1\cong E_1$ is a biholomorphic map of complex spaces 
and one has the following commutative diagram in which every square is Cartesian
\begin{equation}\label{e6.20}
\xymatrix{
W_1\ar[d]\ar[r]^-{\sim}_-{\varphi_1}&E_1\ar[d]\ar@{^{(}->}[r]&\mathbb{C}\times U\times V\ar[d]^-{id_{\mathbb{C}}\times id_{U}\times u^{an}|_V}\ar[r]
&U\times V\ar[r]\ar[d]&V\ar@{^{(}->}[r]\ar[d]&S^{an}\ar[d]^-{u^{an}}\\
\mathcal{W}\ar[r]^-{\sim}_-{\varphi}&E\ar@{^{(}->}[r]&\mathbb{C}\times U\times N\ar[r]&U\times N\ar[r]&N\ar@{^{(}->}[r]&H^{an}
}
\end{equation}
Therefore the external rectangle is Cartesian  (cf.  \cite[Prop.~4.16]{GW}).  Comparing 
with \eqref{e6.19} we conclude  that  $W\cong  W_{1}$.  The complex space $E_{1}$ is reduced
since $W$ is an open subspace of $X^{an}$ which is reduced. Therefore $E_1 = F$. We may thus 
replace $\varphi_1:W_1\to E_1$ by a biholomorphic map $\phi:W\to F$ in \eqref{e6.20}. The composition of 
the bottom maps $\mathcal{W}\to U\times N$ in \eqref{e6.20} equals $p|_{\mathcal{W}}$ (cf. Proposition~\ref{3.24}(iii)), so
by the Cartesian diagram \eqref{e5.9} of Theorem~\ref{5.8} the composition of the top maps $W\to U\times V$
in \eqref{e6.20} equals $q|_W$, therefore $q|_W = q_1\circ \phi$. Part~(i) is proved. The other parts follow similarly
from Proposition~\ref{3.24}, parts (iii) and (iv), by pullback replacing $U\times N$ by $\mathbb{C}\times N$ and 
$U\times V$ by $\mathbb{C}\times V$ in \eqref{e6.20}.
\end{proof}
\begin{theorem}\label{6.3}
Let $Y$ be a smooth, projective, irreducible  curve of genus $g\geq 0$.  Let 
$n$ be a positive integer. Let $G$ be a finite group which 
can be generated by $2g+n-1$ elements. Suppose $G$ has trivial center.
The algebraic variety $H^G_n(Y)$  is a fine  moduli  variety  for  the 
moduli functor $\mathcal{H}^{G}_{Y,n}$ of smooth families of $G$-covers  of  $Y$ 
branched in $n$ points (cf. \S~\ref{5.1}). The universal family is (cf. Theorem~\ref{4.16})
\begin{equation}\label{e6.3}
\pi:\mathcal{C}\to Y\times H^G_n(Y).
\end{equation}
\end{theorem}
\begin{proof}
Let $[q:X\to Y\times S]\in \mathcal{H}^{G}_{Y,n}(S)$. Let 
$v:S\to H^G_n(Y)$, $v(s)=(\beta(s),\underline{m}_{s})$ be the morphism 
of Proposition~\ref{6.1}. We want to prove that $q:X\to Y\times S$ is 
\linebreak
$G$-equivariant to the pullback by $v$ of the family \eqref{e6.3}. This  is 
the unique morphism with this  property  since  the  monodromy 
invariant classifies the  $G$-covers  up  to  $G$-equivalence.   
Since $Z(G)=1$, for every $s\in S$ there exists a unique 
$G$-equivariant isomorphism $\varphi_{s}:X_{s}\to  \mathcal{C}_{v(s)}$  such 
that $\pi_{v(s)}\circ \varphi_{s}=(id_Y\times v)\circ q_{s}$. Let $\varphi:X\to \mathcal{C}$ be the 
$G$-equivariant map which equals $\varphi_{s}$ on every $X_{s}$. One obtains 
the following commutative diagram of maps
\begin{equation}\label{e6.4}
\xymatrix{
X\ar[r]^-{\varphi}\ar[d]_-{q}&\mathcal{C}
\ar[d]^-{\pi}\\
Y\times S\ar[r]^-{id\times v}&Y\times H^G_n(Y)
}
\end{equation}
We aim to prove that $\varphi$ is a morphism and \eqref{e6.4} is a Cartesian diagram. 
One has that 
$H^G_n(Y)=\cup_{y\in Y}U(y)$ is a covering of Zariski open sets  
(cf. Proposition~\ref{4.3}) and 
$\varphi^{-1}(\pi^{-1}(Y\times U(y))) = 
q^{-1}((id_{Y}\times v)^{-1}(Y\times U(y)))$ is a Zariski open subset 
of $X$ for $\forall y\in Y$. Proving that $\varphi$ is a morphism is a local 
matter so we may assume, without loss of generality,  that there exists a 
point $y_{0}\in Y$ such that $q_{s}:X_{s}\to Y$ is unramified at $y_{0}$ for 
every $s\in S$. Then $v(S)\subset U(y_{0})$ and 
$\varphi(X)\subset \mathcal{C}_{U(y_{0})}=\mathcal{C}[y_0]$. Let 
$T=q^{-1}(\{y_{0}\}\times S)$ and let $(q_{T}:X_{T}\to Y\times T,\theta)$ be 
the smooth family of pointed $G$-covers of $(Y,y_0)$ defined in the 
proof of Proposition~\ref{6.1}. Consider the following  commutative  diagram 
of morphisms 
\[
\xymatrix{
X_{T}\ar[r]^-{f}\ar[d]_-{q_{T}}&\mathcal{C}(y_0)\ar[d]^-{p}\ar[r]^-{\kappa}
&\mathcal{C}[y_{0}]\ar[d]^-{\pi}\\
Y\times T\ar[r]^-{id\times u}&Y\times H^G_n(Y,y_0)\ar[r]^-{id\times 
\nu}&Y\times U(y_0)
}
\]
where the left square is from \eqref{e5.9} and the right one from 
\eqref{e4.17}. There are two actions of $G$ on $X_{T}=X\times_{S}T$, namely 
$g(x,t)=(gx,t)$ and $h\ast(x,t)=(x,ht)$, and these actions commute. We claim 
that $f(h\ast z)=h\ast f(z)$ for $\forall h\in G, \forall z\in X_{T}$ 
(cf. \eqref{e4.9}). It suffices to prove this for 
$\forall z\in q_{T}^{-1}(Y\times T\setminus B_{T})$. Let 
$(x,t)\in X'_{s}\times \{t\}$, where $s=\mu(t)$. Let  $\lambda:I\to  X'_{s}$ 
be a path with $\lambda(0)=t,\lambda(1)=x$. Then the path 
$\lambda\times \{t\}$ connects $\theta(t)=(t,t)$ with $(x,t)$ in 
$X'_{s}\times \{t\} = (X\times_{S}T)'_{t}$, so by \eqref{e5.10}
\[
f(x,t) = (\Gamma_{m_{t}}[q_{s}\circ \lambda]_{\beta(s)},\beta(s),m_{t}).
\]
Let $h=m_{t}([\eta]_{\beta(s)})$ and let $\tilde{\eta}$ be the lifting of 
$\eta$ in $X'_{s}$ such that $\tilde{\eta}(0)=t$, $\tilde{\eta}(1)=ht$. 
Then $\tilde{\eta}^{-}\cdot\lambda$ is a path in $X'_{s}$ which connects 
$ht$ with $x$, therefore $(\tilde{\eta}^{-}\cdot\lambda,ht)$ is a path in 
$X'_{s}\times \{ht\} = (X\times_{S}T)'_{ht}$ which connects 
$\theta(ht)=(ht,ht)$ with $(x,ht)$. Therefore by \eqref{e5.10}
\[
\begin{split}
f(h\ast(x,t))=f(x,ht)&=
(\Gamma_{m_{ht}}[q_{s}\circ (\tilde{\eta}^{-}\cdot\lambda)]_{\beta(s)},
\beta(s),m_{ht})\\
&=(\Gamma_{hm_{t}h^{-1}}[\eta^{-}\cdot(q_{s}\circ \lambda)]_{\beta(s)},
\beta(s),hm_{t}h^{-1})\\
&=h\ast f(x,t).
\end{split}
\]
By  Lemma~\ref{4.15a}  this  defines  a  commutative 
diagram of quotient morphisms
\begin{equation}\label{e6.6}
\xymatrix{
{_{G}\backslash 
X_{T}}\ar[r]\ar[d]&{_{G}\backslash\mathcal{C}(y_0)}\ar[d]\ar[r]^-{\cong}
&\mathcal{C}[y_{0}]\ar[d]\\
{_{G}\backslash Y\times T}\ar[r]&{_{G}\backslash Y\times 
H^G_n(Y,y_0)}\ar[r]^-{\cong}&Y\times U(y_0)
}
\end{equation}
By \cite[Prop.~A.7.1.3]{KM} one has ${_{G}\backslash X_{T}}\cong X$ since 
${_{G}\backslash T}\cong S$. 
Furthermore the restriction of $\kappa\circ f$ on every fiber  $(X_{T})_{t}$ 
is the composition of the $G$-equivariant isomorphisms
\[
X_{s}\times \{t\} \to \mathcal{C}(y_0)_{(\beta(s),m_{t})}
\to \mathcal{C}_{(\beta(s),\underline{m}_{s})}.
\]
Hence Diagram~\eqref{e6.6} is, up  to the open embedding 
$\mathcal{C}[y_0]\eto \mathcal{C}$, the same as Diagram~\eqref{e6.4}. This 
proves that $\varphi:X\to \mathcal{C}$ is a $G$-equivariant morphism.
\par
Consider the decomposition of $\varphi$
\[
\xymatrix{
X\ar[r]^-{\varphi_{S}}\ar[dr]&\mathcal{C}_S\ar[d]\ar[r]
&\mathcal{C}\ar[d]\\
&S\ar[r]^-{v}&H^G_n(Y)
}
\]
where the right square is Cartesian. The morphisms $X\to S$ and 
$\mathcal{C}_{S}\to S$ are proper and smooth and  $\varphi_{S}$  induces  an 
isomorphism on every scheme-theoretical fiber 
$X_{s}\overset{\sim}{\lto}(\mathcal{C}_{S})_{s}\cong \mathcal{C}_{v(s)}$ for 
$\forall s\in S(\mathbb{C})$, therefore by \cite[Prop.~(4.6.7)]{EGAIII} 
$\varphi_{S}:X\to \mathcal{C}_{S}$ is an isomorphism. This shows that 
the family of $G$-covers $q:X\to Y\times S$ is $G$-equivalent to the 
pullback by $v:S\to H^G_n(Y)$ of  the  family 
$\pi:\mathcal{C}\to Y\times H^G_n(Y)$.
\end{proof}
\begin{block}\label{6.8}
Let $M$ be an algebraic variety. For every algebraic variety $S$ denote by 
$\Hom(S,M)$ the set of morphisms from $S$ to $M$. Let us denote by 
$\Spec \mathbb{C}$ the algebraic variety with one point $0\in \mathbb{C}$. 
Then $\Hom(\Spec \mathbb{C},M)=M(\mathbb{C})$ is a set bijective to $M$.  The 
mapping $h_{M}(S)=\Hom(S,M)$ defines a contravariant functor 
$h_{M}:\Var_{\mathbb{C}}\to (\Sets)$ from the category of algebraic varieties 
to the category of sets. 
\end{block}
In the next theorem we use the definition of coarse moduli variety of 
\cite[Definition~5.6]{M2}  adapted  to  the  category  $\Var_{\mathbb{C}}$  of 
algebraic varieties over $\mathbb{C}$. 
\begin{theorem}\label{6.8a}
Let $Y$ be a smooth, projective, irreducible  curve of genus $g\geq 0$.  Let 
$n$ be a positive integer. Let $G$ be a finite group which 
can be generated by $2g+n-1$ elements. The mapping which with every 
$[\mathcal{C}\to Y\times S]\in \mathcal{H}^{G}_{Y,n}(S)$ associates the 
morphism $v(S):S\to H^G_n(Y)$ of  Proposition~\ref{6.1} is a well-defined natural 
transformation of  contravariant functors 
$\phi:\mathcal{H}^{G}_{Y,n} \to h_{H^G_n(Y)}$. The couple 
$(H^G_n(Y),\phi)$  is  a   coarse   moduli   variety   for   the   moduli functor 
$\mathcal{H}^{G}_{Y,n}$.
\end{theorem}
\begin{proof}
If $\mathcal{C}\to Y\times S$ is $G$-equivalent to 
$\mathcal{C}_{1}\to Y\times S$, then for every $s\in S$ the $G$-covers 
$\mathcal{C}_{s}\to  Y$  and  $(\mathcal{C}_1)_{s}\to  Y$  have   the   same 
monodromy invariant, so both families define the same morphism 
$v(S):S\to H^G_n(Y)$. If $u:T\to S$ is a morphism and 
$\mathcal{C}_{T}\to Y\times T$ is the pullback of 
$\mathcal{C}\to Y\times S$, then for every $t\in T$, 
$(\mathcal{C}_{T})_{t}\to Y$ is $G$-equivalent to $\mathcal{C}_{u(t)}\to Y$, 
so $v(T)=v(S)\circ u$. This shows that the collection of mappings
\[
\phi(S): \mathcal{H}^{G}_{Y,n}(S)\to \Hom(S,H^G_n(Y))
\]
is a well-defined natural transformation  
$\phi: \mathcal{H}^{G}_{Y,n} \to h_{H^G_n(Y)}$. 
\par
If $S=\Spec \mathbb{C}$, then $\mathcal{H}^{G}_{Y,n}(\Spec  \mathbb{C})$  is 
the set of $G$-equivalence classes of \linebreak
$G$-covers of $Y$ branched  in  $n$ 
points and $\phi(\Spec \mathbb{C})$ transforms every $[p:C\to Y]$ in its 
monodromy invariant $(D,\underline{m})$. Hence 
\[
\phi(\Spec \mathbb{C}): \mathcal{H}^{G}_{Y,n}(\Spec  \mathbb{C})
\to H^G_n(Y)
\]
is a bijection by Riemann's existence theorem. This is Condition~(i) of 
\cite[Definition~5.6]{M2}. Let us verify Condition~(ii). Suppose there is 
an algebraic variety $N$ and a natural transformation 
$\psi:\mathcal{H}^{G}_{Y,n}\to h_{N}$. By Yoneda's lemma we have to prove 
that there exists a unique morphism $f:H^G_n(Y)\to N$ such that for every 
$[\mathcal{C}\to Y\times S]\in \mathcal{H}^{G}_{Y,n}(S)$ if 
the morphisms $v(S):S\to H^G_n(Y)$ and $w(S):S\to N$ are the images 
of $[\mathcal{C}\to Y\times S]$ by $\phi(S)$ and $\psi(S)$ respectively, 
then the following diagram commutes
\begin{equation}\label{e6.10}
\xymatrix{
S\ar[r]^-{v(S)}\ar[dr]_-{w(S)}&H^G_n(Y)\ar[d]^-{f}\\
&N
}
\end{equation}
The uniqueness of $f$, if it exists, is clear since $\phi(\Spec \mathbb{C})$ 
is bijective, so $\psi(\Spec \mathbb{C})$ determines in a unique way the 
map of sets $f:H^G_n(Y)\to N$. Let us prove the existence of such morphism. 
Define the map of sets $f:H^G_n(Y)\to N$ as follows. For every 
$(D,\underline{m})\in H^G_n(Y)$ let 
$[\mathcal{C}_{(D,\underline{m})}\to Y]\in 
\mathcal{H}^{G}_{Y,n}(\Spec \mathbb{C})$ be the equivalence class of 
$G$-covers with monodromy invariant $(D,\underline{m})$. Let 
$f(D,\underline{m})=\psi([\mathcal{C}_{(D,\underline{m})}\to Y])$. The 
equality of maps of sets $w(S)=f\circ v(S)$ holds since $v(S)$ and 
$w(S)$   are   functorial   with   respect   to   base   change   so,
by the definition of $f$,
evaluating  at  every  $s\in  S$,  one  verifies  that  Diagram~\eqref{e6.10} 
commutes. The variety $H^G_n(Y)$ is a union of the Zariski open subsets 
$U(y)$: $H^G_n(Y)=\cup_{y\in Y}U(y)$. In order to prove that 
$f$ is a morphism it suffices to prove that 
$f|_{U(y_{0})}:U(y_{0})\to N$ is a morphism for every $y_{0}\in Y$.
Let $S=H^G_n(Y,y_0)$ and let 
$(\mathcal{C}(y_0)\to Y\times H^G_n(Y,y_0),\zeta)$ be the universal family 
of pointed $G$-covers of $(Y,y_0)$ (cf. Theorem~\ref{3.33}). Let 
$w(S):S\to N$ be the associated morphism. We saw in 
Proposition~\ref{4.6a}(vi) that for $\forall h\in G$ the map defined by 
$z\mapsto h\ast z$ is a $G$-equivalence  between 
$\mathcal{C}(y_0)_{(D,m)}\to Y$ 
and $\mathcal{C}(y_0)_{(D,hmh^{-1})}\to Y$. Therefore $w(S):S\to N$ is 
$G$-invariant with respect to the action of $G$ on $H^G_n(Y,y_0)$ defined by 
$h\ast (D,m)=(D,hmh^{-1})$. By Proposition~\ref{4.3} 
$U(y_0)\cong {_{G}\backslash H^G_n(Y,y_0)}$ and clearly one has 
equality of maps $w(S)=f|_{U(y_0)}\circ \nu$, where $\nu: H^G_n(Y,y_0)\to U(y_0)$
is the quotient morphism.
Therefore $f|_{U(y_0)}$ is a morphism by the universal property of quotient varieties.
\end{proof}
\begin{block}\label{6.13}
Let $\underline{n}=n_1O_1+\cdots +n_kO_k$, $|\underline{n}|=n$ be as in 
$\S~\ref{5.15}$. Let $D\in Y^{(n)}_{\ast}$, let 
$y_{1},y_{2}\in Y\setminus D$, let 
$m_1:\pi_1(Y\setminus D,y_1)\tto G$ and $m_2:\pi_1(Y\setminus D,y_2)\tto G$ 
be  two  path-connected  epimorphisms  (cf.  \S~\ref{4.1}).   Then   $m_{1}$ 
satisfies Condition~(\eqref{e5.15}) if and only if $m_{2}$  satisfies  it.  We 
say that a $G$-cover $p:C\to Y$ branched in $n$ points  is  of  branching 
type $\underline{n}$ if its monodromy invariant $(D,\underline{m})$ has  the 
property that every epimorphism of $\underline{m}$ satisfies 
Condition~(\eqref{e5.15}). We denote by $H^G_{\underline{n}}(Y)$ the set 
of $(D,\underline{m})\in H^G_n(Y)$ of this type. One has 
\[
H^G_n(Y) = \bigsqcup_{|\underline{n}|=n} H^G_{\underline{n}}(Y),
\]
every nonempty $H^G_{\underline{n}}(Y)$ is a union of  connected  components 
in the Zariski topology of $H^G_n(Y)$ and $H^G_{\underline{n}}(Y)$  inherits 
the structure of algebraic variety from $H^G_n(Y)$. 
Let us denote by 
\[
\mathcal{H}^{G}_{Y,\underline{n}} :\Var_{\mathbb{C}} \to (\Sets)
\]
the moduli functor, which associates with every algebraic  variety  $S$ 
the set \linebreak
$\{[X\to Y\times S]\}$ of smooth families of $G$-covers of $Y$ of 
branching type $\underline{n}$ modulo $G$-equivalence and with every  morphism 
$T\to S$ the pullback of such families of $G$-covers. If the  center  of 
$G$ is trivial let us denote by 
\[
\pi_{\underline{n}}:\mathcal{C}_{\underline{n}}\to Y\times 
H^G_{\underline{n}}(Y)
\]
the restriction of the family $\pi:\mathcal{C}\to Y\times H^G_n(Y)$ 
(cf. Theorem~\ref{4.16}). Theorem~\ref{6.3} and Theorem~\ref{6.8a} imply the 
following ones.
\end{block}
\begin{theorem}\label{6.14}
Let $Y$ be a smooth, projective, irreducible curve.  Let  $G$  be  a  finite 
group with trivial center. 
Let $\underline{n}=n_1O_1+\cdots +n_kO_k$, $|\underline{n}|=n$ be as in 
$\S~\ref{5.15}$.
Suppose $H^G_{\underline{n}}(Y)\neq \emptyset$. The algebraic variety 
$H^G_{\underline{n}}(Y)$ is a fine moduli variety for the moduli functor 
$\mathcal{H}^{G}_{Y,\underline{n}}$ of smooth families of  $G$-covers  of 
$Y$ of branching type $\underline{n}$. The universal family is 
\[
\pi_{\underline{n}}:\mathcal{C}_{\underline{n}}\to Y\times 
H^G_{\underline{n}}(Y).
\]
\end{theorem}
\begin{theorem}\label{6.15}
Let $Y$ be a smooth, projective, irreducible curve.  Let  $G$  be  a  finite 
group. 
Let $\underline{n}=n_1O_1+\cdots +n_kO_k$, $|\underline{n}|=n$ be as in 
$\S~\ref{5.15}$.
Suppose $H^G_{\underline{n}}(Y)\neq \emptyset$. The mapping which with every 
$[\mathcal{C}\to   Y\times    S]\in    \mathcal{H}^{G}_{Y,\underline{n}}(S)$ 
associates the morphism \linebreak
$v(S):S\to H^G_{\underline{n}}(Y)$ of 
Proposition~\ref{6.1} is a well-defined  natural  transformation  of  contravariant 
functors 
$\phi:\mathcal{H}^{G}_{Y,\underline{n}}\to h_{H^G_{\underline{n}}(Y)}$. The 
couple $(H^G_{\underline{n}}(Y),\phi)$ is a coarse  moduli  variety  for  the 
moduli functor $\mathcal{H}^{G}_{Y,\underline{n}}$ of smooth families of 
$G$-covers of $Y$ of branching type $\underline{n}$.
\end{theorem}

\begin{acknowledgments}
 The author is associate member of the Institute of Mathematics and Informatics of the Bulgarian Academy of Sciences.
\end{acknowledgments}

\end{document}